\documentclass[12pt,leqno]{amsart}
\usepackage{amssymb}

\usepackage[all]{xy}
\usepackage{graphicx,amssymb,amsfonts,epsfig,amsthm,a4,amsmath,url}
\usepackage[latin1]{inputenc}

\swapnumbers

\newcounter{ictr}
\newenvironment{ilist}{\begin{list}
                         {\textup{(\arabic{ictr})}}
                         {\usecounter{ictr}
                          \setlength{\leftmargin}{0.6truein}
                          \setlength{\rightmargin}{0.5truein}
                          \setlength{\itemsep}{0.0truein}
                          \setlength{\labelwidth}{0.3truein}}}
                      {\end{list}}

\newcounter{actr}
\newenvironment{alist}{\begin{list}
                         {\textup{(\roman{actr})}}
                         {\usecounter{actr}
                          \setlength{\leftmargin}{0.6truein}
                          \setlength{\rightmargin}{0.5truein}
                          \setlength{\itemsep}{0.0truein}
                          \setlength{\labelwidth}{0.3truein}}}
                      {\end{list}}

\newtheorem{thm}{Theorem}[subsection]
\newtheorem{cor}[thm]{Corollary}
\newtheorem{lem}[thm]{Lemma}

\newtheorem{prop}[thm]{Proposition}

\newtheorem{thms}{Theorem}[section]

\newtheorem{lems}[thms]{Lemma}
\newtheorem{props}[thms]{Proposition}

\newtheorem*{thmm}{Theorem}
\newtheorem*{stablerigidity}{Stable Rigidity Theorem}
\newtheorem*{boundedrigidity}{Bounded Rigidity Theorem}

\newtheorem{fibering}[thm]{Fibering Theorem}
\newtheorem{finite-union}[thm]{Finite Union Theorem}
\newtheorem{union}[thm]{Union Theorem}
\newtheorem{invariance}[thm]{Coarse Invariance}

\theoremstyle{definition}

\newtheorem{remark}[thm]{Remark}
\newtheorem{example}[thm]{Example}
\newtheorem{defn}[thm]{Definition}

\newtheorem{que}[thm]{Question}

\newtheorem{remarks}[thms]{Remark}

\newtheorem{defns}[thms]{Definition}
\newtheorem{ques}[thms]{Question}

\theoremstyle{remark}

\numberwithin{equation}{section}

\newcommand{\Z}{\mathbb{Z}}
\newcommand{\N}{\mathbb{N}}
\newcommand{\R}{\mathbb{R}}
\newcommand{\Q}{\mathbb{Q}}

\renewcommand{\L}{\mathbb{L}}
\newcommand{\FF}{\mathcal{F}}
\newcommand{\GG}{\mathcal{G}}

\newcommand{\VV}{\mathfrak{V}}

\renewcommand{\O}{\mathcal{O}}
\newcommand{\CC}{\mathbb{C}}

\newcommand{\Cc}{\mathcal{C}}
\newcommand{\D}{\mathcal{D}}
\newcommand{\W}{\mathcal{W}}
\newcommand{\V}{\mathcal{V}}

\newcommand{\A}{\mathfrak{A}}
\newcommand{\B}{\mathcal{B}}
\newcommand{\C}{\mathfrak{D}}
\newcommand{\E}{\mathfrak{E}}
\newcommand{\F}{\mathfrak{F}}
\newcommand{\FFF}{\mathbb{F}}
\newcommand{\U}{\mathcal{U}}
\newcommand{\X}{\mathcal{X}}
\newcommand{\Y}{\mathcal{Y}}
\newcommand{\ZZ}{\mathcal{Z}}

\newcommand{\n}{\mathfrak{n}}
\newcommand{\p}{\mathfrak{p}}
\newcommand{\m}{\mathfrak{m}}

\DeclareMathOperator{\Aut}{Aut}
\DeclareMathOperator{\HNN}{HNN}
\DeclareMathOperator{\diam}{diam}
\DeclareMathOperator{\stab}{Stab}

\newcommand{\fin}{\text{fin}}
\newcommand{\simp}{\text{sim}}
\newcommand{\ind}{\text{sub}}

\newcommand{\GL}{\text{GL}}

\newcommand{\SL}{\textnormal{SL}}

\newcommand{\Hawaii}{Hawai\kern.05em`\kern.05em\relax i}
\newcommand{\Manoa}{M\=anoa}

\title[Geometric complexity and topological rigidity]%
{A notion of geometric  complexity  and \\ its application to topological rigidity}

\date{\today}

\author{Erik Guentner}
\address{University of \Hawaii ~ at \Manoa, Department of Mathematics,
2565 McCarthy Mall, Honolulu, HI 96822}
\email{erik@math.hawaii.edu}
\author{Romain Tessera}
\address{Vanderbilt University, Department of Mathematics, 1326
  Stevenson Center, Nashville, TN  37240}
\curraddr{UMPA, ENS de Lyon, 46 all\'ee d'Italie, 69364 Lyon Cedex 07, France} 
\email{romain.a.tessera@vanderbilt.edu}
\author{Guoliang Yu}
\address{Vanderbilt University, Department of Mathematics, 1326
  Stevenson Center, Nashville, TN  37240}
\email{guoliang.yu@vanderbilt.edu}

\thanks{The authors were partially supported by grants from the
U.S. National Science Foundation.}

\begin{document}

\begin{abstract}
  We introduce a geometric invariant, called
  finite decomposition complexity (FDC), to study topological rigidity
  of manifolds. We prove for instance that if the fundamental group of
  a compact aspherical manifold $M$ has FDC, and if $N$ is homotopy
  equivalent to $M$, then $M\times \R^n$ is homeomorphic to $N\times
  \R^n$, for $n$ large enough. This statement is known as the stable
  Borel conjecture. On the other hand, we show that the class of FDC
  groups includes all countable subgroups of $\GL(n,K)$, for any field
  $K$, all elementary amenable groups, and is closed under taking
  subgroups, extensions, free amalgamated products, $\HNN$ extensions,
  and direct unions.
  \end{abstract}

\maketitle


\section{Introduction}

We introduce the geometric concept of finite decomposition complexity
to study questions concerning the topological rigidity of manifolds.
Roughly speaking, a metric space has finite decomposition complexity
when there is an algorithm to decompose the space into simpler, more
manageable pieces in an asymptotic way.  The precise definition is
inspired by the property of finite asymptotic dimension of Gromov
\cite{G1} and is presented in Section~\ref{FDCsection}.

While the property of finite decomposition complexity is flexible --
the class of countable groups having finite decomposition complexity
includes all linear groups (over a field with arbitrary
characteristic), all hyperbolic groups and all elementary amenable
groups and is closed under various operations -- it is a powerful tool
for studying topological rigidity -- we
shall see, for example, that if the fundamental group of a closed
aspherical manifold has finite decomposition complexity then its
universal cover is boundedly rigid, and the manifold itself is stably
rigid.

The remainder of the introduction is divided into three parts.  In the
first we describe the applications of finite decomposition complexity
to topological rigidity; in the second we outline results on the class
of countable discrete groups having finite decomposition complexity;
we conclude with general remarks on the organization of the paper.

\subsection*{Topological rigidity}

A closed manifold $M$ is {\it rigid\/} if every homotopy equivalence
between $M$ and another closed manifold is homotopic to a
homeomorphism. The Borel conjecture asserts the rigidity of closed
aspherical manifolds.  Many important results on the Borel conjecture
have been obtained by Farrell and Jones \cite{FJ1,FJ2,FJ3,FJ4}, and more recently
Bartels and Lück \cite{BL}.
These results are proved
by studying dynamical properties of actions of the
fundamental group of $M$.  

Our approach to rigidity questions is different -- we shall focus not
on the dynamical properties but rather on the large scale geometry of
the fundamental group.  As a natural byproduct, we prove 
the {\it bounded Borel conjecture\/}, a `large-scale
geometric' version of the Borel conjecture.  Our principal result in
this direction is the following theorem.

\begin{thmm}
\label{BoundedBorelThmIntro} 
The bounded Borel isomorphism conjecture and the bounded Farrell-Jones
$L$-theory isomorphism conjecture hold for a metric space with bounded
geometry and finite decomposition complexity.
\end{thmm}

We refer the reader to Section~\ref{TopologicalResultsSection} for a
detailed discussion and precise statements, in particular,
Theorems~\ref{BoundedFarrelJonesThm} and \ref{BoundedBorelThm}.  In
this introduction we shall instead describe two concrete applications
to questions of rigidity.  Our first application concerns {\it bounded
  rigidity\/} of universal covers of closed aspherical manifolds.

\begin{boundedrigidity} 
  Let $M$ be a closed aspherical manifold of dimension at least five
  whose fundamental group has finite decomposition complexity
  \textup{(}as a metric space with a word metric\textup{)}.  For every
  closed manifold $N$ and homotopy equivalence $M\to N$ the
  corresponding bounded homotopy equivalence of universal covers
  is boundedly homotopic to a bounded homeomorphism.
\end{boundedrigidity}

The universal covers of $M$ an $N$ as in the statement are, in
particular, homeomorphic.  The conclusion is actually much
stronger -- being {\it boundedly} homeomorphic means that the
homeomorphism is at the same time a coarse equivalence.  We defer
discussion of the relevant notions concerning the {\it bounded
  category\/} to Section~\ref{TopologicalResultsSection}.  See, in
particular Theorem~\ref{thm:brig}, of which the previous result is a
special case.

Davis has given examples of aspherical manifolds whose universal
covers are not homeomorphic to the Euclidean space \cite{D}.  These
examples satisfy the hypothesis of the previous theorem.

A closed manifold $M$ is {\it stably rigid\/} if there exists an $n$
such that for every closed manifold $N$ and every homotopy equivalence
$M\to N$, the product with the identity $M\times \R^n\to N\times \R^n$
is homotopic to a homeomorphism.  The {\it stable Borel conjecture\/}
asserts that closed aspherical manifolds are stably rigid.  The first
result on the stable Borel conjecture is due to Farrell and Hsiang
\cite{FH} who proved that non positively curved Riemannian manifolds
are stably rigid.  Our second application is the following theorem
(see Corollary \ref{UniversalCoverCor}).

\begin{stablerigidity}
  A closed aspherical manifold whose fundamental group has finite
  decomposition complexity is stably rigid.
\end{stablerigidity}

\noindent
Observe that there is no restriction on the dimension.

\subsection*{Groups with finite decomposition complexity}

We consider countable groups equipped with a proper left-invariant metric. 
Recall that every countable group admits such a metric, and that any
two such metrics are coarsely equivalent.  As finite decomposition
complexity is a coarse invariant, the statement that a countable group
has finite decomposition complexity is independent of the choice of metric.

Our next results summarize the main examples of groups having finite
decomposition complexity, and the stability properties of this class of
groups.  For the first statement,
recall that a Lie group is {\it almost connected\/} if it has finitely
many connected components.

\begin{thmm}  
  The collection of countable groups having finite decomposition
  complexity contains all countable linear groups \textup{(}over a
  field of arbitrary characteristic\textup{)}, all countable subgroups
  of an almost connected Lie group, all hyperbolic groups and all
  elementary amenable groups.
\end{thmm}

The geometry of a {\it discrete\/} subgroup of, for example, a
connected semisimple Lie group such as $\SL(n,\R)$ reflects the
geometry of the ambient Lie group.  In this case, the theorem follows
from the well-known result that such groups have finite asymptotic
dimension.  The difficulty in the theorem concerns the case of {\it
  non-discrete or even dense\/} subgroups whose geometry exhibits
little apparent relationship to the geometry of the ambient group.  An
interesting example to which our theorem applies is $\SL(n,\Z[\pi])$,
which has infinite asymptotic dimension.  (Here, $\pi=3.14\dots$.) 

\begin{thmm}
  The collection of countable groups having finite decomposition
  complexity is closed under the formation of subgroups, products,
  extensions, free amalgamated products, $\HNN$ extensions and direct
  unions.
\end{thmm}

At the moment, we know of no group {\it not\/} having finite
decomposition complexity other than Gromov's random groups
\cite{G2,G3,AD}.  Since finite decomposition complexity appears as a
generalization of finite asymptotic dimension, we mention that in
general, solvable groups, or linear groups have infinite asymptotic
dimension.  On the other hand, we shall prove the following result
(see Theorem~\ref{linear}).

\begin{thmm}
  A finitely generated linear group over a field of positive
  characteristic has finite asymptotic dimension.
\end{thmm}

Combined, our results greatly extend an earlier result of Ji \cite{J}
proving the stable Borel conjecture for a special class of linear
groups with finite asymptotic dimension -- namely, subgroups of
$\GL(n, K)$ for a global field $K$, for example when $K=\Q$.

\subsection*{Organization and remarks}

The paper falls into essentially two parts.  In the first part,
comprising Sections~2 -- 6, we introduce and study finite
decomposition complexity.  More precisely, we introduce finite
decomposition complexity in Section~\ref{FDCsection} and study its
basic properties.  In the subsequent section we develop the
permanence characteristics of finite decomposition complexity.  In
Section \ref{Asection} we show that a metric space having finite
asymptotic dimension has finite decomposition complexity, and that one
having finite decomposition complexity has Property $A$.  As a consequence, any
sequence of expanding graphs (viewed as a metric space) does not have
finite decomposition complexity.

Sections~\ref{linearSection} -- \ref{MoreExamplesection} are devoted
to examples.  In these we prove that all (countable) linear groups
have finite decomposition complexity and that and all (countable)
elementary amenable groups have finite decomposition complexity.

The remainder of the paper, comprising
Section~\ref{TopologicalResultsSection} -- \ref{conclusion} and the
appendices, is devoted to applications to topological rigidity. We
have split Section~\ref{TopologicalResultsSection} into two parts. The
first outlines two essential results,
Theorems~\ref{AsymptoticVanishingK-theoryThm} 
and \ref{AssemblyMapThm} -- these are proven in Section~\ref{vanishing} and
\ref{assembly}, respectively. The remainder of
Section~\ref{TopologicalResultsSection} contains a 
self-contained exposition of our topological rigidity results in their
most general and natural setting -- these are deduced directly from
Theorems~\ref{AsymptoticVanishingK-theoryThm} and \ref{AssemblyMapThm}.

In the first appendix, we discuss several variants of the Rips
complex, the {\it relative\/} and {\it scaled\/} Rips complexes, which
are important technical tools for the proofs of our results on
topological rigidity. We believe this material may be of independent
interest. In the second appendix, we recall the controlled
Mayer-Vietoris sequences in $K$ and $L$-theory, as proved in
\cite{RY1,RY2}. These are essential tools for our proofs of the
bounded Borel and the Farrell-Jones $L$-theory isomorphism conjectures.

\subsection*{Acknowledgment}  

We would like to thank Yves de Cornulier for his helpful comments.

\section{Decomposition complexity}
\label{FDCsection}

Our proofs of the isomorphism conjectures will be based on
Mayer-Vietoris arguments -- we shall apply a {\it large-scale}
version of an appropriate Mayer-Vietoris sequence to prove that an
assembly map is an isomorphism. To carry out this idea, we shall
decompose a given metric space as a union of two subspaces, which are
{\it simpler\/} than the original. Roughly, {\it simpler\/} is
interpreted to mean that each subspace is itself a union of spaces
at a pairwise distance large enough that proving the isomorphism for
the subspace amounts to proving the isomorphism for these constituent
pieces `uniformly'. 
Further, this basic decomposition
step shall be iterated a finite number of times, until we reach a
bounded family.   
This is the idea behind finite decomposition complexity.

We shall need to formulate our notion of finite decomposition
complexity not for a single metric space, but rather for a 
{\it metric family}, a (countable) family of metric spaces which we
shall denote by $\X = \{\, X \,\}$; throughout we view a single metric
space as a family containing a single element.

In order to streamline our definitions we introduce some terminology
and notation for manipulating decompositions of metric spaces and
metric families. A collection of subspaces $\{\, Z_i \,\}$ of a metric
space $Z$ is {\it $r$-disjoint\/} if for all $i\neq j$ we have
$d(Z_i,Z_j)\geq r$. To express the idea that $Z$ is the union of
subspaces $Z_i$, and that the collection of these subspaces is
$r$-disjoint we write \begin{equation*}
  Z = \bigsqcup_{r-disjoint} Z_i.
\end{equation*}
A family of of metric spaces $\{\, Z_i \,\}$ is
{\it bounded\/} if there is a uniform bound on the diameter of the
individual $Z_i$:
\begin{equation*}
  \sup \diam(Z_i) <\infty.
\end{equation*}

\begin{defn}
\label{dfn:drdecomp}
A metric family $\X$ is {\it $r$-decomposable\/} over a metric
family $\Y$ if every $X\in \X$ admits an {\it $r$-decomposition\/}
\begin{equation*}
    X = X_0\cup X_1, \quad
     X_i = \bigsqcup_{r-disjoint} X_{ij},
\end{equation*}
where each $X_{ij}\in \Y$.  
We introduce the notation
${\X} \overset{r}{\longrightarrow} {\Y}$ to indicate that $\X$ is
$r$-decomposable over $\Y$.
\end{defn}

\begin{defn}
  Let $\A$ be a collection\footnote{While
    we generally prefer the term `collection' to `class', we do not mean
    to imply that a collection of metric families is a {\it set\/} of
   metric families. We shall not belabor the associated
   set-theoretic complications.} of metric families.  A metric
 family $\X$ is {\it decomposable\/} over $\A$ if, for every $r>0$,
 there exists a metric family $\Y\in\A$ and an $r$-decomposition of
 ${\X}$ over $\Y$.
 The collection $\A$ is {\it stable under decomposition\/} if every
 metric family which decomposes over $\A$ actually belongs to $\A$.
\end{defn}

\begin{defn} 
  The collection $\C$ of metric families with finite
  decomposition complexity is the minimal collection of metric families
  containing the bounded metric families and stable under decomposition.
  We abbreviate membership in $\C$ by saying that a metric family in
  $\C$ has FDC.
\end{defn}

This notion has been inspired by the property of {\it finite
  asymptotic dimension\/} of Gromov \cite{G1}. Recall that a metric
space $X$ has {\it finite asymptotic dimension\/} if there exists
$d\in \N$ such that for every $r>0$ the space $X$ may be written as a
union of $d+1$ subspaces, each of which may be further decomposed as
an $r$-disjoint union: \begin{equation} \label{eqn:fad}
  X = \bigcup_{i=0}^d X_i, \quad
     X_i = \bigsqcup_{r-disjoint} X_{ij},
\end{equation}
in which the family $\{\, X_{ij} \,\}$ (as both $i$ and $j$ vary) is
bounded.  It follows immediately from the definitions that metric
families with asymptotic dimension at most one (uniformly) have finite
decomposition complexity.

\begin{remark}
\label{weakFDCrem}
At the outset of this project, we defined a property weaker than FDC
which is more transparently related to finite asymptotic dimension.
The difference between this property -- {\it weak finite decomposition
  complexity\/} --  and the one defined here lies in the type of
decomposition -- we replace $r$-decomposability by the notion of
$(d,r)$-decomposability.  

A metric family $\X$ is {\it $(d,r)$-decomposable\/} over a metric
family $\Y$ if every $X\in \X$ admits a decomposition
\begin{equation*}
    X = X_0\cup\ldots\cup X_d , \quad
     X_i = \bigsqcup_{r-disjoint} X_{ij},
\end{equation*}
where each $X_{ij}\in \Y$.  
The metric family $\X$ {\it weakly} decomposes over the collection
$\A$ of metric families, if there exists a $d\in \N$ such that for
every $r>0$, there exists $\Y\in \A$ and a $(d,r)$-decomposition of
${\X}$ over $\Y$.
The collection of metric families with {\it weak finite decomposition
  complexity\/} is the smallest collection containing bounded metric
families, and stable under weak decomposition.

Clearly, both finite asymptotic dimension (again, uniformly in the
sense of Bell and Dranishnikov \cite{BD0}) and finite decomposition
complexity imply {\it weak\/} finite decomposition complexity. While
true that finite asymptotic dimension implies finite decomposition
complexity, at least for proper metric spaces, this is already
difficult.
\end{remark}

\begin{que}
Are finite and weak finite decomposition complexity equivalent?
\end{que}

\subsection{The metric decomposition game}

The game has two players and begins with a metric family.  The objective of
the first player is to successfully decompose the spaces comprising the
family, whereas the second player attempts to obstruct the decomposition.

Formally, let $\X=\Y_0$ be the starting family.  The game begins
with the first player asserting to the second they can decompose $\Y_0$.  
The second player challenges the first by
asserting an integer $r_1$.  The first turn ends with the first player
exhibiting a $r_1$-decomposition of $\Y_0$ over a metric family
$\Y_1$.

Subsequent turns are analogous: the first player asserts they can
decompose the family $\Y_i$; the second
challenges with an $r_{i+1}$; the first responds by exhibiting an
$r_{i+1}$-decomposition of $\Y_i$ over a metric family $\Y_{i+1}$.

The first player has a {\it winning strategy\/} if, roughly speaking,
they can produce decompositions ending in a bounded family {\it no
  matter what choices the second player makes.\/}  While the outcome is certain, 
the number of turns required for eventual victory may depend on the
choices made by the second player.
A game in which the first player applies their
winning strategy exhibits a sequence of decompositions beginning with
the spaces in $\X$ and ending in a bounded family:

\begin{equation}
\label{game}
\xymatrix@+15pt{    
      \X=\Y_0 \ar[r]^-{r_1} & \Y_1 \ar[r]^-{r_2} & \Y_2 \ar[r] &
      \dots \Y_{n-1} \ar[r]^-{r_n} & \Y_n, }
         \quad
      \text{$\Y_n$ bounded}.
\end{equation}

\subsection{Decomposition strategy}

We shall now formalize the idea of a winning strategy for the
decomposition game. As we shall see in the next section, the existence
of a winning strategy is equivalent to finite decomposition complexity.

A {\it decomposition tree\/} is a directed, rooted tree $T$
satisfying the following:
\begin{ilist}
  \item every non-root vertex of $T$ is the terminal vertex
    of a unique edge;
  \item every non-leaf vertex of $T$ is the initial vertex of countably
    many edges, which are labeled by the natural numbers;
  \item $T$ contains no infinite ray (geodesic edge-path).
\end{ilist}
Conditions (1) and (2) are succintly expressed by saying that $T$ is a
`$1$-up, $\infty$-down' rooted tree in which the `down' edges emanating
from each vertex are in explicit bijection with $\N$.
Figure~\ref{fig:block} 
depicts a typical caret in $T$ with labeled edges.

\begin{figure}[h] 
   \centering
   \includegraphics[height=1.0in]{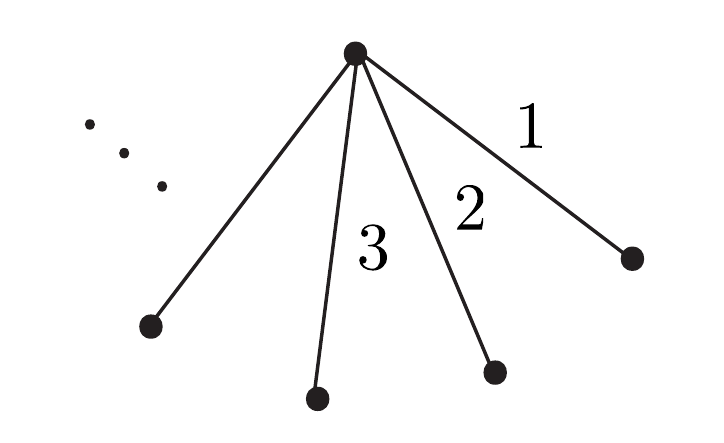}
   \caption{A caret}
   \label{fig:block}
\end{figure}

The vertices of a rooted tree are partially ordered by setting $v\leq w$
precisely when $w$ lies on the (unique) geodesic edge path from $v$ to
the root vertex.  Condition (3) asserts that for a decomposition tree
every decreasing chain in this order eventually terminates.

Let $\X$ be a metric family.
A {\it decomposition strategy\/} for $\X$ comprises a decomposition tree
$T$, the {\it support tree\/} of the strategy, together with a labeling
of the vertices of $T$ by metric families $\Y$ subject to the following requirements:
\begin{ilist}{\setcounter{ictr}{3}}
  \item the root vertex of $T$ is labeled $\X$;
  \item every leaf of $T$ is labeled by a bounded family;
  \item if $\Y$ labels the initial vertex and $\ZZ$ the
    terminal vertex of an edge labeled by $r\in\N$ then  $\Y$ is
    $r$-decomposable over $\ZZ$.
\end{ilist}

The intuition, of course, is that paths in $T$ beginning at its root and
ending at a leaf correspond to decompositions of the family $\X$ into
uniformly bounded parts.  Precisely, if the edges along the path are
labeled $r_1,\dots,r_n$ and the vertices are labeled $\X$,
$\Y_1, \dots, \Y_n$ we obtain the decomposition in
(\ref{game}).

An important feature of the decomposition game is the dependence among
the $r_i$.  Existence of a decomposition strategy implies that a
partially completed decomposition may always be continued, no matter
the choice of the subsequent $r_i$, and eventually terminates in a
bounded family.

\subsection{Equivalent formulations of FDC}

We shall present two equivalent descriptions of the collection of
families having finite decomposition complexity.  We require the
following lemma.

\begin{lem}
\label{lem:poset}
Let $T$ be a decomposition tree.  There exists a function 
$v\mapsto \alpha_v$ from the set of vertices of $T$ to a set of countable ordinal
numbers with the properties that $\alpha_v=0$ if $v$ is a leaf and
  \begin{equation*}
    \alpha_v = \sup_{w<v}\, \{\, \alpha_w + 1 \,\}
  \end{equation*}
otherwise.
\end{lem}
\begin{proof}
  Observe that, by virtue of the no-infinite-ray assumption, a
  decomposition tree has leaves.
  Define, for each countable ordinal $\alpha$, a subset $L_\alpha$ of
  the vertex set of $T$ by transfinite recursion:  $L_0$ is the set of
  leaves of $T$; for $\alpha>0$, 
  \begin{equation*}
    L_\alpha = \text{the set of leave of 
            $T\setminus \cup_{\beta<\alpha}L_\beta$},
  \end{equation*}
  if this set is nonempty, and $L_\alpha=\emptyset$ otherwise.  Note
  that, if it is non-empty, the
  set $T\setminus \cup_{\beta<\alpha}L_\beta$ is again a decomposition tree,
  and therefore has leaves.

Let $\alpha_0 = \{\, \alpha \colon L_\alpha\neq \emptyset \,\}$ and
let $\mathfrak{L}= \{L_{\alpha}, \alpha< \alpha_0\}$. Clearly,
$\mathfrak{L}$ is a partition of the set of vertices of $T$, and the
map $\alpha\mapsto L_\alpha:\alpha_0\to \mathfrak{L}$ is a bijection.
It follows that $\alpha_0$ is countable.  Finally, for every vertex $v$,
let $\alpha_v$ be the unique $\alpha$ such that $v\in L_\alpha$. It is
not difficult to see that $\alpha_v$ satisfies the desired properties.
\end{proof}

\begin{defn}
\label{dfn:heirarchy}
We define, for each ordinal $\alpha$, a collection of metric families
according to the following prescription:
\begin{ilist}
  \item Let $\C_0$ be the collection of bounded families:
    \begin{equation*}
      \C_0 = \{\, \X \colon \text{$\X$ is bounded} \,\}.
    \end{equation*}
  \item If $\alpha$ is an ordinal greater than $0$, let $\C_{\alpha}$
    be the collection of metric families decomposable over
    $\cup_{\beta<\alpha}\C_\beta$: 
    \begin{equation*}
      \C_\alpha = \{\, \X \colon \text{$\forall\, r$
         $\exists\, \beta<\alpha$ $\exists\, \Y\in \C_{\beta}$
         such that $\X\overset{r}{\longrightarrow}\Y$} \,\}.
    \end{equation*}
\end{ilist}
We introduce the notation $\C_{\fin}$ for the union of the $\C_n$,
over $n\in\N$. 
\end{defn}

\begin{thm}
\label{thm:char_complexity}
The following statements concerning a metric family $\X$ are equivalent:
\begin{ilist}
\item $\X$ has finite decomposition complexity;
\item $\X$ admits a decomposition strategy;
\item there exists a countable ordinal $\alpha$ such that $\X\in\C_\alpha$.
\end{ilist} 
\end{thm}

\begin{proof}
  For purposes of the proof let $\C'$ be collection of families
  admitting a decomposition strategy; let $\C''$ be the collection of
  families belonging to $\C_\alpha$ for some countable ordinal
  $\alpha$. We must show $\C''= \C'= \C$.

A simple transfinite induction shows that $\C_\alpha\subset \C$ for
every ordinal $\alpha$.  Thus, $\C''\subset \C$.

Next, we show that $\C\subset \C'$.  Since a bounded family
trivially admits a decomposition strategy, it suffices to show that the
collection $\C'$ is closed under decomposability.  Let $\X$ be a
family decomposable over $\C'$.  For every
$r\in\N$, obtain a family $\Y_r\in\C'$ such that 
$\X$ is $r$-decomposable over $\Y_r$.  A
decomposition strategy for $\X$ is obtained by attaching
strategies for the $\Y_r$ to the bottom of an `infinite caret' whose
root vertex is labeled $\X$ and whose edges are labeled by $\N$ as
shown in Figure~\ref{fig:concat}.

\begin{figure}[h] 
   \centering
   \includegraphics[height=2.0in]{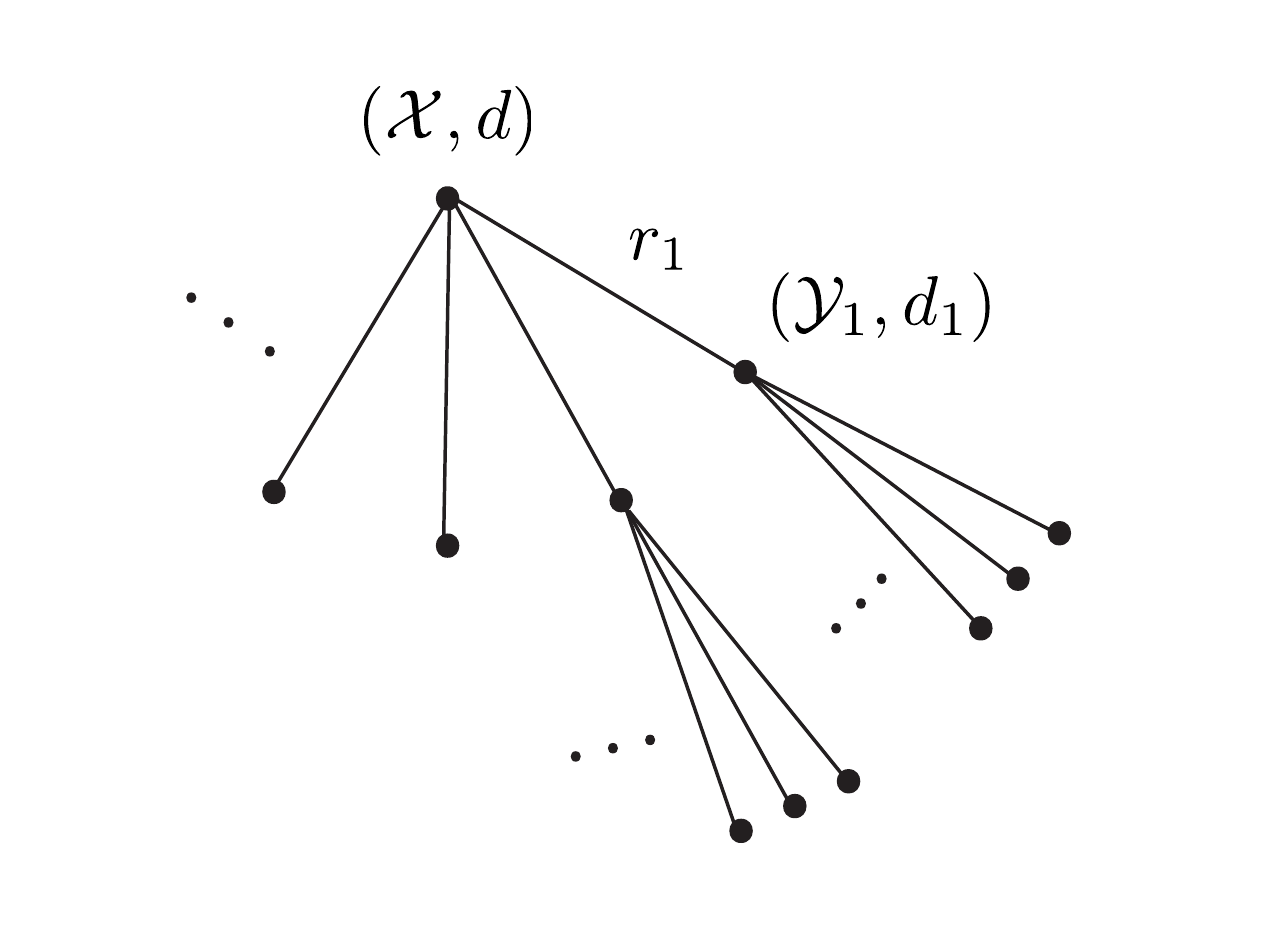}
   \caption{Concatenating strategies}
   \label{fig:concat}
\end{figure}

Finally, we show that $\C'\subset \C''$.  Let $\X\in \C'$.  Let $T$ be the
support tree of a decomposition strategy for $\X$; denote the label of a
vertex $v$ by $\Y_v$ and let $v\mapsto \alpha_v$ be a function
with the properties outlined in Lemma~\ref{lem:poset}.  It suffices to
show that for every ordinal $\alpha$ we have: if $\alpha_v\leq \alpha$
then $\Y_v\in\C_\alpha$.  This follows easily by transfinite induction.
\end{proof}

\begin{example}
  We shall require the fact, easily verified (by induction),
  that $\Z^n\in\C_n$, for each natural number $n$. 
\end{example}

\begin{example}
  Let $G=\oplus\Z$ (countably infinite direct sum), equipped with a proper
  left-invariant metric; for concreteness use the metric determined by
  the requirement that the generator having a single $1$ in the
  $i^{\mbox{\tiny{th}}}$ position has length $i$.
  We claim that $G\in \C_\omega$.  Indeed, let $r>0$ be given.  If
  $n$ is a natural number greater than $r$ then the decomposition of $G$ into
  cosets of the subgroup $\Z^n$ (via the first $n$ coordinates) is
  $r$-disjoint and the family comprised of these cosets is in $\C_n$
  by the previous example.

  Moreover, we shall shortly see that since $G$ does not have finite
  asymptotic dimension it is {\it not\/} in $\C_{\fin}$.
\end{example}

\begin{example} 
Let $G=\Z \wr \Z$. 
By 
considering the extension $0\to \oplus_{n\in \Z} \Z \to G\to \Z\to 0$
we see that $G \in \C_{\omega+1}$
(compare to Remark~\ref{fiberingremark-finoverfin}).  
\end{example}

\begin{example}   
Let $G=\oplus G_n$, where $G_n=(\dots((\Z\wr\Z)\wr \Z)\dots)\wr\Z$,
the wreath product of $n$ copies of $\Z$.  Then 
$G\in \C_{\omega^2}$.  It is an open question whether $G\in\C_\alpha$
for some $\alpha<\omega^2$. 
\end{example}

\section{Permanence of FDC}\label{PermanenceSection}

We shall study the permanence characteristics of finite decomposition
complexity. While we shall focus on finite decomposition complexity,
all permanence results stated in this section hold for weak finite
decomposition complexity as well.

We begin by recalling some elementary concepts from coarse
geometry. Let $\X$ and $\Y$ be metric families. A {\it subspace\/} of
the family $\Y$ is a family $\ZZ$, every element of which is a
subspace of some element of $\Y$. A {\it map of families\/} from $\X$
to $\Y$ is a collection of functions $F = \{\, f \,\}$, each mapping
some $X\in \X$ to some $Y\in\Y$ and such that every $X\in \X$ is the
domain of at least one $f\in F$. We use the notation $F:\X\to \Y$ and,
when confusion could occur, write $f:X_f\to Y_f$ to refer to an
individual function in $F$. The {\it inverse image\/} of the subspace
$\ZZ$ is the collection \begin{equation*}
  F^{-1}(\ZZ)=  \{\, f^{-1}(Z)  \colon  \text{$Z\in\ZZ$, $f\in F$} \,\}.
\end{equation*}
The inverse image is a subspace of $\X$.

A map of families $F:\X\to \Y$ is {\it uniformly expansive\/} if there
exists a non-decreasing function $\rho:[0,\infty)\to [0,\infty)$ such
that for every $f\in F$ and every $x$, $y\in X_f$
\begin{equation}
\label{eqn:ue}
   d(f(x),f(y)) \leq \rho(d(x,y));
\end{equation}
it is {\it effectively proper\/} if there exists a proper
non-decreasing function $\delta:[0,\infty)\to [0,\infty)$ such that for
every $f\in F$ and every $x$, $y\in X_f$
\begin{equation}
\label{eqn:ep}
  \delta(d(x,y))\leq d(f(x),f(y));
\end{equation}
it is a {\it coarse embedding\/} if it is both uniformly expansive and
effectively proper.  (In this case, if $\X$ is unbounded then $\rho$ is
also proper.)  Summarizing, a map of families $F$ is a coarse embedding
if the individual $f$ are coarse embeddings {\it admitting a common
  $\delta$ and $\rho$\/}.  Similar remarks apply to uniformly expansive
and effectively proper maps.

Recall that a coarse embedding $f:X\to Y$ of metric spaces is a coarse
equivalence if it admits an `inverse' -- a coarse embedding $g:Y\to X$
for which the compositions $f\circ g$ and $g\circ f$ are {\it close\/}
to the identity maps on $X$ and $Y$, respectively: 
\begin{equation}
\label{Cdense}
  \text{there exists $C>0$ such that $d(x,gf(x))\leq C$
           and $d(y,gf(y))\leq C$,}
\end{equation}
for all $x\in X$ and $y\in Y$. So motivated, a coarse embedding
$F:\X\to\Y$ of metric families is a {\it coarse equivalence\/} if each
$f\in F$ is a coarse equivalence admitting an inverse $g$ satisfying
the following two conditions: first, the collection $G= \{\, g \,\}$
is a coarse embedding $\Y\to\X$ of metric families; second, the
composites $f\circ g$ and $g\circ f$ are {\it uniformly close\/} to
the identity maps on the spaces comprising $\X$ and $\Y$, in the sense
that the constant $C$ in (\ref{Cdense}) may be chosen independently of
the spaces $X\in\X$ and $Y\in\Y$.  Two metric
families $\X$ and $\Y$ are {\it coarsely equivalent\/} if there exists
a coarse equivalence $\X\to \Y$. Coarse equivalence is an equivalence
relation.

\subsection{Permanence for spaces}

The primitive permanence properties for metric families are Coarse
Invariance, the Fibering and Union Theorems.  We shall prove these in
this section.

\begin{lem}
\label{lem:ue}
Let $\X$ and $\Y$ be metric families and let $F:\X\to \Y$ be a uniformly
expansive map.  For every $r>0$ there exists an $s>0$ such that if $\ZZ$
and $\ZZ'$ are subspaces of $\Y$ and
$\ZZ'\overset{s}{\longrightarrow}\ZZ$ then
$F^{-1}(\ZZ')\overset{r}{\longrightarrow} F^{-1}(\ZZ)$.  Further, $s$
depends only on $r$ and on the non-decreasing function $\rho$ satisfying
\textup{(\ref{eqn:ue})}.

\end{lem}
\begin{proof}
  Assuming $F$ is uniformly expansive let $\rho$ be such that
  (\ref{eqn:ue}) holds.  Set $s=\rho(r)$ and assume
  $\ZZ'\overset{s}{\longrightarrow}\ZZ$.  An element of $F^{-1}(\ZZ')$
  has the form $f^{-1}(Z)$ for some $Z\in\ZZ'$ and $f\in F$.
  Given such an element obtain a decomposition
  \begin{equation*}
    Z = Z_0\cup Z_1, \quad Z_i = \bigsqcup_{s-disjoint} Z_{ij},
  \end{equation*}
  in which the $Z_{ij}\in\ZZ$.  We then have a decomposition
\begin{equation*}
  f^{-1}(Z) =  f^{-1}(Z_0) \cup f^{-1}(Z_1), \quad 
             f^{-1}(Z_i) = \bigcup f^{-1}(Z_{ij}),
\end{equation*}
in which the $f^{-1}(Z_{ij})\in F^{-1}(\ZZ)$.  From the definition of
$s$ we see immediately that the union on the right is $r$-disjoint.
\end{proof}

\begin{lem}
\label{lem:ep}
Let $\X$ and $\Y$ be metric families and let $F:\X\to \Y$ be an
effectively proper map.  If $\ZZ$ is a bounded subspace of $\Y$ then
$F^{-1}(\ZZ)$ is a bounded subspace of $\X$.
\end{lem}
\begin{proof}
  Assuming $F$ is effectively proper let $\delta$ be such that
  (\ref{eqn:ep}) holds.  
  Let $B$ bound the diameter of the metric spaces in the family
  $\ZZ$.  Using the hypothesis that $\delta$ is proper, let $A$ be such
  that $\delta(A)\geq B$.  Then $F^{-1}(\ZZ)$ is bounded by $A$.
\end{proof}

\begin{invariance}
  Let $\X$ and $\Y$ be metric families.  If there is a coarse embedding
  from $\X$ to $\Y$ and $\Y$ has finite decomposition complexity, then
  so does $\X$.  In particular:
  \begin{ilist}
    \item a subspace of a metric family with FDC itself has FDC;
    \item if $\X$ and $\Y$ are coarsely
      equivalent, then $\X$ has FDC if and only if $\Y$ does.
  \end{ilist}
\end{invariance}
\begin{proof}
  By pruning and relabeling we can pull back a decomposition strategy
  for $\Y$ to $\X$.  Precisely, select an increasing sequence of natural
  numbers $s_1,s_2,\dots$ such that $s_i\geq i$.  Prune $T$ by removing
  a vertex $v$, together with the entire `downward' subtree based at $v$
  and the unique upward edge incident at $v$, when this upward edge is
  labeled by an element of $\N\setminus \{\, s_i \,\}$.  The resulting
  graph $T'$ is a subtree of $T$ and a vertex of $T'$ is a leaf of $T'$
  exactly when it is a leaf of $T$.  
  Relabel a typical edge as shown in Figure~\ref{fig:relabel}. 
  \begin{figure}[h]
    \centering
    \includegraphics[height=1.0in]{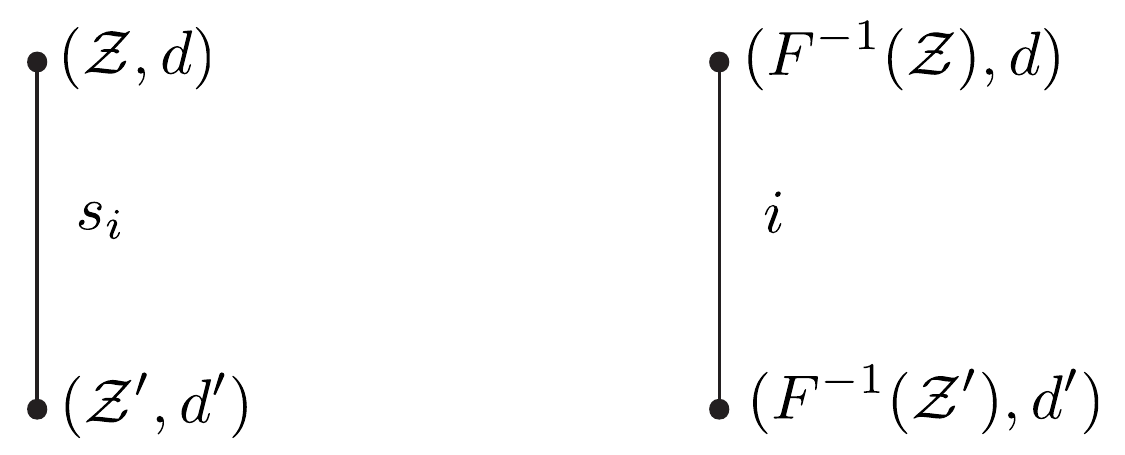}
    \caption{Relabeling}
    \label{fig:relabel}
  \end{figure}
  It follows from Lemmas~\ref{lem:ue} and \ref{lem:ep}
  that the labeling requirements for a decomposition strategy are
  fulfilled.
\end{proof}

\begin{fibering}
\label{fibering}
  Let $\X$ and $\Y$ be metric families and let $F:\X\to \Y$ be a
  uniformly expansive map.  Assume $\Y$ has finite decomposition
  complexity, and that for every bounded subspace $\ZZ$ of $\Y$ the
  inverse image $F^{-1}(\ZZ)$ has finite decomposition complexity.  Then
  $\X$ has finite decomposition complexity.
\end{fibering}
\begin{proof}
  A decomposition strategy for $\Y$ pulls back, as in the previous
  proof, to a {\it partial\/} decomposition strategy for $\X$.  It is
  partial in that the leaves of its support tree are labeled by families
  which are not (necessarily) bounded but rather are the inverse images of
  bounded subspaces of $\Y$.  We complete the partial strategy by
  attaching to a leaf labeled by $F^{-1}(\ZZ)$ a strategy for this
  family. 
\end{proof}

\begin{remark}
\label{fiberingremark}
Directly from the definitions we see that $\X\in\C_n$ precisely when
$\X$ admits a decomposition strategy in which the strategy tree has
{\it depth\/} not greater than $n$, meaning that the length of a
geodesic emanating from the root vertex is at most $n$.  In the
notation of the Fibering Theorem, the previous proof shows the
following: suppose that $\Y\in\C_n$ and that there exists a natural
number $m$ such that $F^{-1}(\ZZ)\in \C_m$ for every bounded subspace 
$\ZZ$ of $\Y$; then $\X\in \C_{n+m}$.
\end{remark}

\begin{remark}
\label{fiberingremark-finoverfin}
Continuing in the spirit of the previous remark, suppose that
$\Y\in\C_{\fin}$ and that $F^{-1}(\ZZ)\in \C_{\fin}$ for every
bounded subspace $\ZZ$ of $\Y$.  Then $\X\in\C_{\omega+\fin}$, meaning
that for some natural number $n$ we have $\X\in\C_{\omega+n}$.  The
distinction between this remark and the previous is that here we
assume merely that each $F^{-1}(\ZZ)\in \C_m$ for some natural number
$m$, {\it which may depend on $\ZZ$}.
\end{remark}

\begin{finite-union}
\label{finite-union-thm}
Let $X$ be a metric space, expressed as a union of finitely many metric
subspaces $X=\cup_{i=0}^n X_i$.  If the metric family $\{\, X_i \,\}$
has finite decomposition complexity so does $X$.
\end{finite-union}
\begin{proof}
  Consider first the case $n=2$, illustrated in Figure~\ref{fig:finite-union}.
  For every $r>0$, the metric space $X=X_1\cup X_2$
  is $r$-decomposable over the family $\{\, X_1,X_2 \,\}\in\C$.
  Thus $X\in\C$. The general case follows by induction.  
\end{proof}

\begin{union}
\label{union-theorem}
Let $X$ be a metric space, expressed as a union of metric subspaces
$X=\cup_{i\in I} X_i$.  Suppose that the metric family $\{\, X_i \,\}$
has finite decomposition complexity and that for every $r>0$ there exists a
metric subspace $Y(r)\subset X$ having finite decomposition complexity
and such that the
subspaces $Z_i(r) = X_i \setminus Y(r)$ are pairwise $r$-disjoint. Then
$X$ has finite decomposition complexity.
\end{union}
\begin{proof}
To conclude that $X$ has finite decomposition complexity, it suffices
to show that $X$ is decomposable over $\C$.
The proof of this is illustrated in Figure~\ref{fig:union}.
Formally, for every $r>0$ let $Y(r)$ and $Z_i(r)$ be as in the
statement.  The decomposition 
\begin{equation*}
  X = Y(r) \bigcup Z(r), \quad Z(r) = \bigsqcup_{r-disjoint} Z_i(r)
\end{equation*}
is a $r$-decomposition of $X$ over the family 
$\Y_r=\{\, Y(r) \,\}\cup \{\, Z_i(r) : i\in I\,\}$.  Since the $Z_i(r)$
are subspaces of the $X_i$ and the family $\{\, X_i \,\}$ has finite
decomposition complexity, the family $\{\, Z_i(r) : i\in I \,\}$ does as
well; since $Y(r)$ has finite decomposition complexity, 
the family $\Y_r$ does as well.  
\end{proof}

\begin{figure}[h]
    \begin{minipage}[b]{0.45\linewidth}
        \centering
        \includegraphics[scale=0.5]{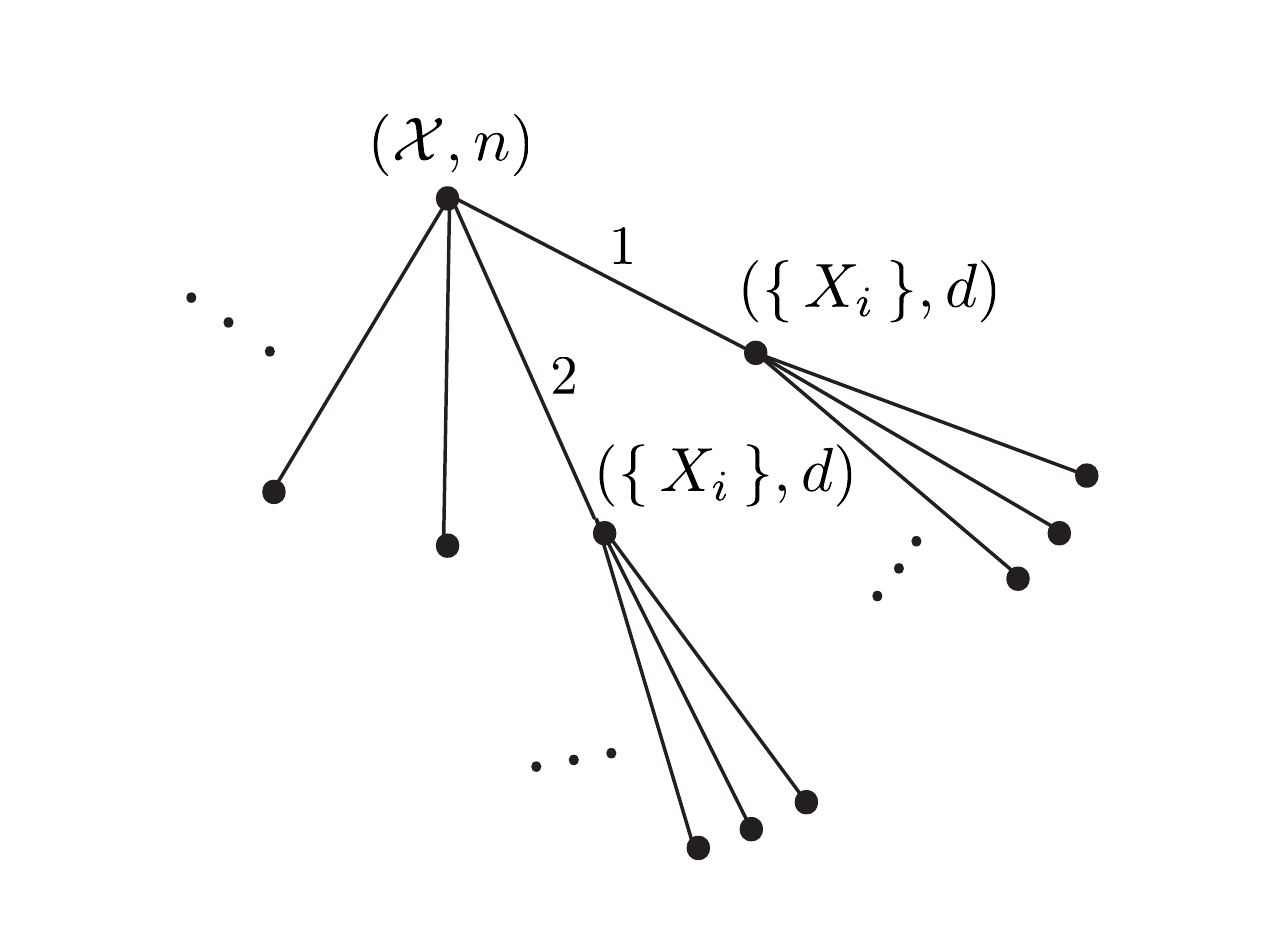}
        \caption{A finite union}
        \label{fig:finite-union}
    \end{minipage}
    \begin{minipage}[b]{0.45\linewidth}
        \centering
        \includegraphics[scale=0.5]{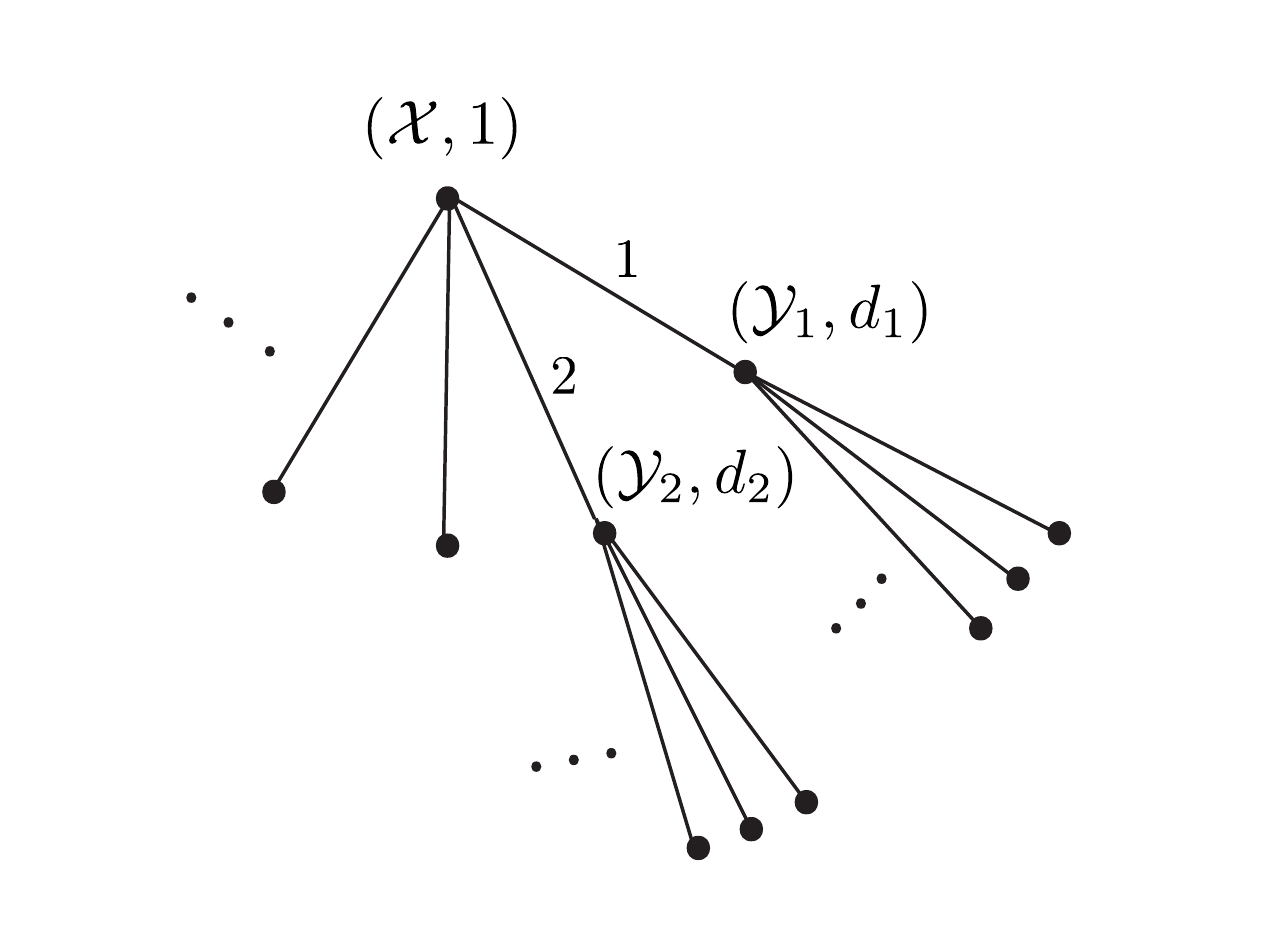}
        \caption{A union}
        \label{fig:union}
    \end{minipage}
\end{figure}

\begin{remark}
 While we could state union theorems in the context of metric families
 (instead of single metric spaces) we shall not require this level of
 generality.
\end{remark}

\subsection{Permanence for groups}

Most (though not all) permanence properties for discrete groups are
deduced by allowing the group to act on an appropriate metric space, and
applying the permanence results for spaces detailed in the previous
section.

Let $G$ be a countable discrete group.  Recall that a countable discrete
group admits a proper length function $\ell$ and that any two metrics
defined from proper length functions by the formula 
\begin{equation*}
  d(s,t) = \ell(s^{-1}t)
\end{equation*}
are coarsely equivalent (in fact, the identity map is a coarse
equivalence).  As a consequence, a coarsely invariant property of
metric spaces is a property of countable discrete groups -- whether or
not a group has the property is not an artifact of the particular
metric chosen.  Consequently, we say that a discrete group has finite
decomposition complexity if its underlying metric space has finite
decomposition complexity for some (equivalently every) metric defined
as above.

\begin{prop}
\label{limits}
  A countable direct union of groups with finite decomposition
  complexity has finite decomposition complexity.  Equivalently, a
  countable discrete group has 
  finite decomposition complexity if and only if every finitely
  generated subgroup does.
\end{prop}
\begin{proof}
Let $G$ be a countable discrete group, expressed as the union of a
collection of subgroups each of which has finite decomposition
complexity:
$G=\cup G_i$.  Equip $G$ with a proper length function and associated
metric.  We shall show that for every $r>0$ (the metric space) $G$ is
$r$-decomposable over a  metric family with finite decomposition complexity; by
Theorem~\ref{thm:char_complexity} this will suffice.

Let $r>0$.  Since the ball of radius $r$ centered at the identity in $G$
is finite there exists $i=i(r)$ such that this ball is contained in
$G_i$. It follows that the decomposition of $G$ into the cosets of $G_i$
is $r$-disjoint.  Further, the family comprised of these cosets has finite decomposition complexity since each coset is isometric to $G_i$, which has finite decomposition complexity 
(in any proper metric so in the subspace metric)
by assumption.
\end{proof}

Let now $X$ be a metric space, and suppose that $G$ acts (by isometries)
on $X$.  For $R>0$ the {\it $R$-coarse stabilizer of $x$\/} is 
\begin{equation*}
  \stab(x,R) = \{\, g\in G : d(x,g\cdot x) < R \,\}.
\end{equation*}
In general an $R$-coarse stabilizer is a {\it subset\/} of $G$.  The
$0$-coarse stabilizer of $x$ is its stabilizer, a subgroup of $G$.  The
space $X$ is {\it locally finite\/} if every ball is finite.  


\begin{lem}
  For every $x\in X$ the orbit map $g\mapsto g\cdot x : G\to X$ 
  is uniformly expansive.  \qed
\end{lem}

\begin{prop}
\label{fiberingprop}
Let $G$ be a countable discrete group acting on a metric space $X$
with finite decomposition complexity. If there exists $x_0\in X$ such
that for every $R>0$ the $R$-coarse stabilizer of $x_0$ has finite
decomposition complexity then $G$ has finite decomposition complexity.
\end{prop}
\begin{proof}
By restricting to the orbit of $x_0$ we may assume the action is
transitive.  Together with the coarse 
stabilizer condition, the fact that the orbit map $g\to g\cdot x_0$ is a
surjective and equivariant map $G\to X$
implies that the hypothesis of
the Fibering Theorem~\ref{fibering} are fulfilled.  The proposition
follows. 
\end{proof}

\begin{cor}
  Let $G$ and $X$ be as in the previous proposition.  If $X$ is locally
  finite, and if there exists $x_0\in X$ such that the stabilizer of
  $x_0$ has finite decomposition complexity, then $G$ has finite
  decomposition complexity. 
\end{cor}
\begin{proof}
  Under the stated hypotheses the Finite Union Theorem implies that the
  coarse stabilizers of $x_0$ have finite decomposition complexity.
  Thus, the previous proposition applies.
\end{proof}

\begin{cor}
\label{extension}
The collection of countable discrete groups  with finite
decomposition complexity is closed under extensions. \qed
\end{cor}

\begin{prop}
If a countable discrete group acts \textup{(}without
inversion\textup{)} on a tree, and the vertex stabilizers of the
action have finite decomposition complexity, then the group itself has
finite decomposition complexity. 
\end{prop}
\begin{proof} 
  According to the Bass-Serre theory, a group as in the statement is
  built from vertex stabilizers of the action by iterated free
  products (with amalgam), $\HNN$ extensions and direct unions.  An
  $\HNN$ extension, in turn, is built from free products (with
  amalgam), a direct union and a group extension.  As we have seen
  that the class of (countable discrete) groups with finite
  decomposition complexity is closed under direct unions, subgroups
  and extensions the proposition follows once we show that a
  free product with amalgam has finite decomposition complexity if the
  factors do.  But, this follows axiomatically from the above proven
  permanence results -- essentially, apply fibering to the action on
  the Bass-Serre tree using the union theorem to conclude that the
  coarse stabilizers have finite decomposition complexity.  For a more
  detailed discussion see \cite{GP} and the references therein.
\end{proof}

\section{FDC, Property A and finite asymptotic dimension}
\label{Asection}

In this section we shall discuss how the property of finite
decomposition complexity relates to other familiar properties from
coarse geometry, notably to Property $A$ and to finite asymptotic
dimension.

Above we have discussed how the definition of finite decomposition
complexity is {\it motivated\/} by finite asymptotic dimension.  We
shall now pursue this discussion further, our goal being to prove that
a metric space having finite asymptotic dimension has finite
decomposition complexity as well.

Recall that a metric space is {\it proper\/} if closed and bounded
sets are compact.  A discrete metric space is proper precisely when it
is {\it locally finite\/} in the sense that every ball is finite.  It
is not difficult to see that a proper metric space having finite
asymptotic dimension has finite decomposition complexity.  Indeed,
according to a theorem of Dranishnikov-Zarichnyi a proper metric space
having finite asymptotic dimension admits a coarse emebdding into the
product of finitely many locally finite trees \cite{DZ}.  As trees
have finite decomposition complexity, we may apply our permanence
results to conclude.  More generally, for metric spaces which are {\it
  not necessarily proper\/} we have the following theorem.

\begin{thms}
\label{fad&fdc}
 A metric space has finite asymptotic dimension if and only if it
 belongs to $\C_{\fin}$.  In particular, a metric space having finite
 asymptotic dimension has finite
 decomposition complexity as well.
\end{thms}

We are primarily interested in the forward implication, and 
shall reduce the general case to the case of proper metric spaces
using an ultralimit construction.  Before turning to the proof, we recall
the relevant background notions.  Let $X$ be a (pseudo-)metric space.
The {\it Gromov triple product\/} (with respect to a base point
$x_0$) is
\begin{equation*}
  (x|y)=\tfrac{1}{2}\left( d(x,x_0)+d(y,x_0)-d(x,y) \right). 
\end{equation*}
The (pseudo)-metric space $X$ is {\it Gromov $0$-hyperbolic\/} if
\begin{equation*}
  (x|z) \geq \min \{\, (x|y), (y|z) \,\},
\end{equation*}
for all $x$, $y$ and $z\in X$.  The notion of $0$-hyperbolicity is
independent of the choice of base point \cite [Prop.~2.2]{ABetc}.  A
Gromov $0$-hyperbolic (pseudo)-metric space has asymptotic dimension
at most $1$.  (See \cite{Roe} for a direct argument.)  Hence, a Gromov
$0$-hyperbolic (pseudo-)metric space has finite decomposition
complexity.

\begin{proof}[Proof of Theorem~\ref{fad&fdc}]
A simple induction shows that a
(pseudo)-metric space belonging to $\C_n$ admist, for every $r$, a
$(2^n,r)$-decomposition over a bounded family.  In particular, its
asymptotic dimension is at most $2^n-1$.  

For the converse, let $X$ be a (pseudo)-metric space having finite asymptotic
  dimension at most $n$.  We shall show that $X$ has finite
  decomposition complexity, indeed that $X\in \C_{\fin}$.   
  Apply the result of
  Drashnikov-Zarichnyi \cite{DZ} to the finite subsets of $X$ -- these are
  locally finite metric spaces and the essential observation here is
  that the result of Drashnikov-Zarichnyi applies {\it uniformly\/}.
  Precisely, there exists $\rho$ and $\delta$ and for each finite
  subset $F\subset X$ a $\rho$-uniformly expansive and
  $\delta$-effectively proper map into a product of trees:
  \begin{equation*}
    F \to T_0^F \times \cdots \times T_n^F.
  \end{equation*}
  Projecting to the individual factors we lift the tree metrics back to
  $F$ to obtain a family of (pseudo-)metrics $d_0^F,\dots,d_n^F$ on
  $F$ with the following two properties.  First, each $d_i^F$ is Gromov
  $0$-hyperbolic -- recall here that an $\R$-tree is Gromov
  $0$-hyperbolic.  Second, the identity $F\to F$ is $\rho$-uniformly
  expansive and $\delta$-effectively proper, when the domain is
  equipped with the subspace metric from $X$ and the range the sum
  metric $d_0^F+\cdots+d_n^F$ -- explicitly, for all $x$, $y\in F$ we
  have
  \begin{equation}
\label{CEexplicit}
    \delta(d_X(x,y)) \leq d_0^F(x,y) + \cdots + d_n^F(x,y) 
        \leq \rho(d_X(x,y)).
  \end{equation}

  Let now $\F$ be the collection of finite subsets of $X$ containing a
  fixed base point $x_0$, viewed as a directed set under inclusion.
  Let $\omega$ be an ultrafilter on the {\it set\/} $\F$ with the
  following property:   for every convergent net $(t_F)_{F\in\F}$ of
  real numbers we have
  \begin{equation*}
    \text{$\lim t_F = \omega$-$\lim t_F$},
  \end{equation*}
where the limit of the left is the ordinary limit of the convergent
net, and the limit of the right is the limit with respect to the
ultrafilter $\omega$.

  For each fixed $i=0,\dots,n$ form the 
  ultraproduct $X_i = \omega$-$\lim F_i$, where we write $F_i$ for $F$
  equipped with the metric $d_i^F$.  Precisely, $X_i$ is the space of
  of $\F$-indexed nets $\overline x = (x_{F})$, with
  $x_{F}\in F$, for which $d_i^F(x_{F},x_0)$ is bounded independent of
  $F$.\footnote{As we work with pseudo-metric spaces it is not
    necessary to consider equivalence classes as would be typical.}
  Define a pseudo-metric on $X_i$ by
  \begin{equation*}
    \text{$d_i(\overline x,\overline y) = 
               \omega$-$\lim d_i^F(x_{F},y_{F})$,}
  \end{equation*}
where $\overline x = (x_F)$ and $\overline y = (y_F)$ are elements of
$X_i$.  Define a map $\alpha_i:X\to X_i$ by associating to $x$ the
`constant sequence'; it follows immediately from
(\ref{CEexplicit}) that
\begin{equation*}
  \alpha_i(x)_F = \begin{cases} x, &x\in F \\ 
                  x_0, &\text{else} \end{cases}
\end{equation*}
satisfies the boundedness condition required of elements of $X_i$.

Now, the individual $X_i$ are Gromov $0$-hyperbolic, essentially
because the condition for $0$-hyperbolicity, satisfied by the
individual $d_i^F$, involves only finitely many points and passes to
the limit intact.  Thus, each $X_i$ has finite decomposition
complexity and indeed belongs to $\D_1$.  An elementary application of
permenance shows that 
the product $X_0\times\cdots\times X_n$ belongs to $\C_{n+1}$.  See
Remark~\ref{fiberingremark}. 

The proof concludes with the observation that the product of the
$\alpha_i$ is a coarse embedding $X\to X_0\times\cdots \times X_n$.
To verify this observe that for $x\in X$ we have $\alpha_i(x)=x$ for
$\omega$-almost every $F$.  So, if $y\in X$ as well we have
\begin{equation*}
  \sum_{i=0}^n d_i(\alpha_i(x),\alpha_i(y)) = 
     \text{$\omega$-$\lim$}  \sum_{i=0}^n d_i^F(x,y)
\end{equation*}
which by (\ref{CEexplicit}) is bounded above by $\rho(d_X(x,y))$ and
below by $\delta(d_X(x,y))$.
\end{proof}

\begin{remarks}
  We are unable to find a reference for the existence of an
  ultrafilter as required in the previous proof; we provide instead
  the following simple argument.  In the notation of the proof, the
  collection of all subsets of $\F$ containing a set of the form
  \begin{equation*}
    \{\, F\in\F : F_0\subset F \,\}
  \end{equation*}
  is a filter, the {\it filter of tails in $\F$\/}.  
  An ultrafilter containing the filter of tails is as required --
  existence of an ultrafilter containing a given filter is a classic
  application of Zorn's lemma. 
\end{remarks}

We turn now to a discussion of 
Property $A$, a geometric property guaranteeing coarse embeddability
into Hilbert space \cite{Y2}. We shall show that a
metric space with (weak) finite decomposition complexity has Property
$A$. As a consequence, any sequence of expanding graphs (as a metric
space) does not have (weak) finite decomposition complexity since it
does not admit a coarse embedding into Hilbert space.

To prove the main result of this section, it is convenient
to work with a characterization of Property $A$ introduced by Dadarlat
and Guentner \cite{DG}. A metric family $\U=\{\, U \,\}$ is a {\it
  cover\/} of a metric space $X$ if every $U\in\U$ is a metric subspace
of $X$ and
\begin{equation*}
  X= \bigcup_{u\in\U} U.  
\end{equation*}
A {\it partition of unity on $X$ subordinate to a cover $\U$\/} is a
family of maps $\phi_U:X\to [0,1]$, one for each $U\in\U$, such that
each $\phi_U$ is supported in $U$ and such that for every $x\in X$
\begin{equation*}
  \sum_{U\in\U} \phi_U(x) = 1.
\end{equation*}
We do not require that the sum is finite for any particular $x\in X$.

\begin{defns}
\label{def-unif-exactness}
A metric family $\X$ is exact if 
for every $R>0$ and $\varepsilon>0$ and for every $X\in \X$ there
is a partition of unity $\{\, \psi_U^X \,\}$ on $X$
subordinate to a cover $\U_X$ of $X$
such that the collection
\begin{equation*}
  \U = \{\, U : \text{$U\in \U_X$, some $X$} \,\}
\end{equation*}
is a bounded metric family and such that for every $X\in\X$ and every
$x$, $y\in X$ 
\begin{equation*}
    d(x,y)\leq R \Longrightarrow
    \sum _{U\in \U_X}|\psi_U^X (x)-\psi_U^X (y)|\leq \varepsilon.
\end{equation*}
\end{defns}

\begin{remarks}
  Our definition of exactness is equivalent to the notion of an {\it
    equi-exact family of metric spaces\/} introduced by Dadarlat and
  Guentner (compare \cite{DG} Defs.~2.7 and 2.8).  However, we have
  indexed our partition of unity and cover differently so our definition
  is not identical to the one in \cite{DG}.
\end{remarks}

For the statements of the next two results, recall that a metric space
has {\it bounded geometry\/} if for every $r>0$ there exists an
$N=N(r)$ such that every ball of radius $r$ contains at most $N$
points.

\begin{thms}[\cite{DG} Prop.~2.10]
  A metric space having Property $A$ is exact.  A bounded geometry exact
  metric space has Property $A$. \qed
\end{thms}

\begin{thms}
  A metric family having \textup{(}weak\textup{)} finite decomposition
  complexity is exact.  A bounded geometry metric space having finite
  decomposition complexity has Property $A$.
\end{thms}
\begin{proof}
Let $\E$ be the collection of exact metric families.  By
Theorem~\ref{thm:char_complexity} it suffices to show that $\E$
contains the bounded families and is closed under decomposability.

Clearly, $\E$ contains the bounded families -- for $X$ selected from a
bounded family the partition of unity comprised of the constant function
at $1$, subordinate to the cover $\{\, X \,\}$, fulfills the definition.

It remains to to check that $\E$ is closed under decomposability.  Let
$\X$ be a family and assume $\X$ is decomposable over $\E$ -- for
every $r$ there exists $\Y\in\E$ such that $\X$ 
is $r$-decomposable over $\Y$. We shall apply \cite[Theorem~4.4]{DG}
to show that $\X\in\E$.  Let $\delta>0$.  Select $r$ large enough so
that $r\delta\geq 2$ and obtain $\Y$ as above.  Translating the
notion of decomposability into the language of \cite{DG} we see that
$\Y$ is an equi-exact family with the property that that every $X\in\X$
admits a $r$-separated cover, the pieces of which belong to $\Y$.
Thus, the hypotheses of \cite[Theorem~4.4]{DG} are satisfied and we
conclude that $\X$ is an equi-exact family.  In other words, $\X\in\E$.
\end{proof}

\begin{remarks}
  \cite[Theorem~4.4]{DG} is stated for a single metric space.  The
  same argument can be used to verify that it applies to a metric
  family.
\end{remarks}

\section{Linear groups have FDC}
\label{linearSection}

\setcounter{thm}{0}

We devote the present section to the proof of the following result.

\begin{thm}
\label{linear}
If a countable group admits a faithful, finite dimensional
representation \textup{(}as matrices over a field of arbitrary
characteristic\textup{)}, then it has finite decomposition complexity.
Precisely, let $G$ be a finitely generated subgroup of
$\GL(n,K)$, where $K$ is a field.  If $K$ has characteristic zero then
$G\in\C_{\omega+\fin}$; if $K$ has positive characteristic then $G$ has
finite asymptotic dimension. 
\end{thm}

\begin{example}
The wreath product
$\Z\wr \Z$ can be realized as a subgroup of $\SL(2,\Z[X,X^{-1}])$ but
does not have finite  asymptotic dimension (it contains an infinite
rank abelian subgroup).  Concretely, $\Z\wr \Z$ is isomorphic to
the group comprised of all matrices of the form 
\begin{equation*}
 \left(\begin{array}{cc} X^n&p(X^2) \\ 0&X^{-n} \end{array}\right),
\end{equation*}
where $n\in \Z$ and $p$ is a Laurent polynomial with $\Z$ coefficients
in the variable $X^2$.  
On the other hand $\Z/p\Z\wr \Z$, which can be similarly realized as a
subgroup of $\SL(2,\Z/p\Z[X,X^{-1}])$, has finite asymptotic dimension by
results of Bell and Dranishnikov \cite{BD2} and Dranishnikov and Smith
\cite{DS}. Indeed, it is an extension with both quotient $\Z$ and kernel 
$\oplus \Z/p\Z$ having finite asymptotic dimension.
These examples show that the conclusion in the theorem is optimal.
\end{example}

In light of our permanence results, the first assertion in
Theorem~\ref{linear} follows from the second.  For the second, let $K$
be a field and let $G$ be a finitely generated subgroup of $\GL(n,K)$.
The subring of $K$ generated by the matrix entries of a finite
generating set for $G$ is a finitely generated domain $A$, and we have
$G\subset \GL(n,A)$.  Thus, we are lead to consider finitely generated
domains, and their fraction fields.

\subsection{Preliminaries on fields}

The proof of Theorem~\ref{linear} relies on a
refinement of the notion of discrete embeddability introduced earlier
by Guentner, Higson and Weinberger \cite{GHW}.  
A {\it norm}\footnote{Guentner-Higson-Weinberger use the term valuation.}
on a field $K$ is a map $\gamma:K\to [0,\infty)$ satisfying, for all
$x$, $y\in K$
\begin{ilist}
  \item $\gamma(x)=0 \; \Leftrightarrow \;x=0$ 
  \item $\gamma(xy)=\gamma(x) \gamma(y)$
  \item $\gamma(x+y)\leq \gamma(x)+\gamma(y)$
\end{ilist}
A norm obtained as the restriction of the usual absolute value on $\CC$
via a field embedding $K\to \CC$ is {\it archimedean\/}.  
A norm satisfying the stronger {\it ultra-metric inequality\/}
\begin{ilist}
\setcounter{ictr}{3}
  \item $\gamma(x+y)\leq \max\{\, \gamma(x),\gamma(y) \,\}$
\end{ilist}
in place of the triangle inequality (3) is {\it non-archimedean\/}.
If in addition the range of $\gamma$ on $K^\times$ is a discrete subgroup of
the multiplicative group $(0,\infty)$ the norm is {\it discrete\/}.
If $\gamma$ is a discrete norm on a field $K$ the subset
\begin{equation*}
  \O = \{\, x\in K \colon \text{$\gamma(x)\leq 1$} \,\}
\end{equation*}
is a subring of $K$, the {\it ring of integers of $\gamma$\/};  the subset
\begin{equation*}
  \m = \{\, x\in K \colon \text{$\gamma(x)<1$} \,\}
\end{equation*}
is a principal ideal in $\O$; a generator for $\m$ is a {\it
  uniformizer\/}.

\begin{defn}
\label{def:sde}
A field $K$ is strongly discretely embeddable (for short SDE) if for
every finitely 
generated subring $A$ of $K$ there exists a finite set $N_A$ of
discrete 
norms on $K$, and  countable set $M_A$ of archimedean norms on $K$
with the following property:  for every real number $k$ there exists 
a finite subset $F_A(k)$ of $M_{A}$ such that for every
$s>0$ the set
\begin{equation*}
  \B_A(k,s)=\{\, a\in A \colon
    \text{$\forall \gamma\in N_A \;\; \gamma(a)\leq e^k$ and 
         $\forall \gamma\in F_A(k) \;\; \gamma(a)\leq s$} \,\}
\end{equation*}
is finite. 
\end{defn}

\begin{remark}[SDE versus DE]
  In discrete embeddability \cite[Definition~2.1]{GHW} the family of
  norms depends only on the subring $A$.  In Definition~\ref{def:sde},
  the subset of discrete norms depends only on $A$, but is required to
  be finite; the subset $F_A(k)$ of archimedean norms is also 
  is required to be finite, but depends on $k$.  One readily
  verifies that a strongly discretely embeddable field in the sense of
  Definition~\ref{def:sde} is discretely embeddable in the sense of
  \cite{GHW}.
\end{remark}

\begin{remark}
A field of positive characteristic admits no archimedean norms.
In particular, a field of nonzero characteristic is strongly discretely
embeddable if and only if for every finitely generated subring $A$
there exists a finite set $N_A$ of (discrete) norms such that for
every $k\in \N$ the set
\begin{equation*}
  \B_A(k) = \{\, a\in A \colon 
         \text{$\forall \gamma\in N_A \; \gamma(a)\leq e^k$} \,\}
\end{equation*}
is finite.  
\end{remark}

\begin{example}
Let $q$ be a positive power of the prime $p$ and let $\FFF_q$ be the
finite field with $q$ elements.  Let $K=\FFF_q(X)$ be the rational
function field.   We shall show that $K$ satisfies
the definition of SDE with respect to subring of polynomials
$A=\FFF_q[X]\subset K=\FFF_q(X)$.  Indeed, consider the norm 
\begin{equation}
\label{degreenorm}
  \gamma(P/Q)=e^{deg(P)-deg(Q)}, 
\end{equation}
where $P$ and $Q$ are nonzero polynomials.
For all $k\in \N$, we have 
\begin{equation*}
  B_A(k)=\{a\in A : \gamma(a)\leq e^k\} = \FFF_q[X]_k,
\end{equation*}
the set of polynomials of degree at most $k$.  As this set is already
finite, it suffices to take $N_A=\{\gamma\}$. 

A similar analysis applies to
$A=\FFF_q[X_1,X_2,\ldots,X_n]\subset K=\FFF_q(X_1,X_2,\ldots, X_n)$.
Indeed, observe that $K=K_i(X_i)$, where
$K_i=\FFF_q(X_1,\ldots,\hat{X_i},\ldots, X_n)$.  Thus, in analogy with 
(\ref{degreenorm}), we can define a
norm reflecting the degree in the variable $X_i$:
$\gamma_i(P/Q)=e^{deg(P)-deg(Q)}$, where $P,Q\in K_i[X_i]$.
The definition is satisfied with $N_A=\{\gamma_i : 1\leq i\leq n\}$. 
\end{example}

\begin{example}
The case of characteristic zero is more involved, since we will have
to deal with archimedean norms.  Let us treat the simplest non trivial
case:  $A=\Z[X]\subset K=\Q(X)$.   The set $N_A$ of discrete norms
will again contain the single norm $\gamma$, defined as above in
(\ref{degreenorm}).  As in the previous example, 
\begin{equation*}
  B_A(k)=\{\, a\in A : \gamma(a)\leq  e^k\} = \Z[X]_k,
\end{equation*}
the polynomials of degree at most $k$.  Unfortunately, since the
coefficients are integers, this set is infinite -- we shall need
to add some archimedean norms.  

Evaluation of a rational function at a transcendental element 
$t\in \CC$ defines an embedding $\Q(X)\subset \CC$, and hence an
archimedean norm on $\Q(X)$.  Consider the set 
$M_{A}=\{\gamma^0,\gamma^1\ldots \}$, where the $\gamma^i$ are
archimedean norms constructed in this way from distinct transcendental
elements, $t_0,\dots,t_k$.  We are to show that for each $s$ the set
\begin{equation*}
  \B_A(k,s) = \{\, P\in \Z[X]_k : 
      \text{$|P(t_i)|\leq s$ for all $i=0,\dots,k$} \,\}
\end{equation*}
is finite.  This is, however, straightforward:  the assignment
\begin{equation*}
  P \mapsto (P(t_0), \dots P(t_k))
\end{equation*}
defines an isomorphism of complex vector spaces 
$\CC[X]_k\to \CC^{k+1}$ (with the obvious notation) and 
$\Z[X]_k\subset \CC[X]_k$ is discrete.

The multi-variable case $\Z[X_1,\ldots, X_n]\subset \Q(X_1,\ldots,X_n)$  can be
treated as in the previous example, by replacing the single discrete
norm $\gamma$ by the discrete norms $\gamma_i$, for $i=1,\ldots,n$.
\end{example}

\begin{remark}
With these two examples in hand, the reader can omit the remainder of
this section and proceed directly to
Section~\ref{lineargroupSubsection}  to complete a proof that
$\GL(d,\FFF_q[X_1,\ldots, X_n])$ has finite asymptotic dimension, and
that $\GL(d,\Z[X_1,\ldots, X_n])$ has finite decomposition complexity.
\end{remark}

\begin{prop}
A finitely generated field is strongly discretely embeddable.  
\end{prop}

This proposition follows from an adaptation either of the proof of
\cite[Theorem~2.2]{GHW}, or of \cite[Proposition~1.2]{AS} (which
relies on Noether's normalization theorem).  Below, we follow
\cite{GHW}.  The proof comprises three lemmas: in the first we show
that finite fields and the field of rational numbers are SDE; in the
second and third we show that SDE is stable under transcendental and
finite extensions, respectively.

\begin{lem}[Finite fields and the rationals]
Finite fields and the field of rational numbers are strongly
discretely embeddable.
\end{lem}
\begin{proof}
The assertion is obvious for finite fields.  Turning to the rationals, 
let $A$ be a finitely generated subring of $\Q$.  Thus, there exists a
positive integer $n$ such that $A=\Z[1/n]$.  
Let $N_A=\{\; \gamma_{p_1},\dots, \gamma_{p_{m}}  \;\}$ where, for
each prime divisor $p_i$ of $n$, we denote by $\gamma_{p_i}$  the
(discrete) $p$-adic norm on $\Q$. Let $M_{A}$ consist solely of the
archimedean norm coming from the inclusion $\Q\subset\CC$.  We leave
to the reader to verify that these choices satisfy
Definition~\ref{def:sde}. 
\end{proof}

\begin{lem}[Transcendental extensions]
Strong discrete embeddability is stable under the formation of
transcendental extensions.   
\end{lem} 
\begin{proof}
  We shall show that the field of rational
  functions over a (countable) SDE field is itself SDE.  To this end, let $K$ be
  an SDE field and let $B$ be a
  finitely generated subring of $K(X)$.  There exist monic prime
  polynomials $Q_1,\ldots, Q_m\in K[X]$ and a finitely generated
  subring $A$ of $K$ such that 
  $B\subset A[X][Q_1^{-1}, \ldots, Q_m^{-1}]$.  According to
  Definition~\ref{def:sde}, applied to the subring $A$ of $K$, we
  obtain (finitely many) discrete norms $N_A$, and (countably many)
  archimedean norms $M_A$. 

Let $N_B$ be the following (finite) set of discrete norms on $K(X)$:
\begin{ilist}
  \item the elements of $N_A$ extended to $K(X)$;
\end{ilist}
(At one place below we shall require the fact that if
$\gamma$ is a discrete norm on $K$ then its extension $\tilde\gamma$
to $K(X)$ satisfies $\tilde{\gamma}(P) = \max \{\, \gamma(a) \,\}$,
where the maximum is taken over the coefficients $a$ of the polynomial
$P\in K[X]$.)
\begin{ilist}\setcounter{ictr}{1}
  \item the norm $\gamma_{\infty}(P/Q)=e^{deg(P)-deg(Q)}$;
  \item the norms $\gamma_{Q_i}(PQ_i^l)=e^{-l}$ where $\gcd(Q_i,P)=1$
    and $l\in\Z$ \newline
    (there are $m$ norms of this type, one for each $i=1,\dots,m$). 
\end{ilist}

Each of the archimedean norms $\gamma\in M_A$ arises from an embedding
of fields $\phi_\gamma:K\to \CC$.  
Let $t_0,t_1,\dots$ be a countable family of distinct transcendentals
in $\CC$ that are {\it not\/} in the subfield of $\CC$
generated by the images of these embeddings -- to see that this is
possible, observe that since
both $M_A$ and $K$ are countable so is the subfield generated by the
images.  With these
choices, each embedding $\phi_\gamma$ extends to an embedding
$K(X)\to\CC$ by sending $X$ to $t_i$; we denote 
the corresponding norm on $K(X)$ by $\gamma_i$.  Let
\begin{equation*}
  M_B = 
   \{\, \gamma_i \colon \text{$\gamma\in M_A$ and $i=0,1,\dots$} \,\},
\end{equation*}
a countable set of archimedean norms on $K(X)$.  We record for future
use that in our notation $\gamma_i(P)=|\phi_\gamma(P)(t_i)|$, for
every $P\in K[X]$; here, $\phi_\gamma(P)\in \CC[X]$ is the
polynomial obtained by applying $\phi_\gamma$ to the coefficients of
$P$.

We shall show that $N_B$ and $M_B$ satisfy the condition in
Definition~\ref{def:sde}.  For this, let $k>0$ be given.  An element
of $\B_{B}(k)$ necessarily has the form
\begin{equation}
\label{bbk}
  \frac{P}{Q}=\frac{P}{Q_1^{n_1}\ldots Q_m^{n_m}},
\end{equation}
where $n_1,\ldots, n_m$ are $\leq k$, so that also $\deg P \leq
k'=k\left(1+\sum \deg Q_i\right)$ -- here we are using the norms in $N_B$
of types (2) ad (3) above.  In particular, the set of possible
denominators $Q$ is finite; denote it by $\mathcal Q_k$.  Set 
\begin{equation*}
  k''= k + \log \max \{\, \gamma(Q) \colon 
         Q\in {\mathcal Q_k}, \gamma\in N_B \,\}
\end{equation*}
(actually, taking the maximum over $\gamma\in N_B$ of type (1) would
suffice).   Summarizing, an element of $\B_{B}(k)$
has the form (\ref{bbk}) in which $Q$ belongs to the finite set
$\mathcal Q_k$, the degree of $P$ is at most 
$k'$ and all coefficients of $P$ belong
to $\B_{A}(k'')$ -- the last assertion follows from the formula for the
extension of an element of $N_A$ to an element of $N_B$ of type (1).

Define a finite set of archimedean norms on $K(X)$ by 
\begin{equation*}
  F_B(k) = \{\, \gamma_i\in M_B \colon
                  \text{$\gamma\in F_A(k'')$ and $i=0,\dots,k'$} \,\}
\end{equation*}
Let now $s>0$; it
remains to show that $\B_{B}(k,s)$ is finite.  We claim
that an element of $\B_{B}(k,s)$ satisfies, in addition to the
conditions outlined above for membership in $\B_{B}(k)$, the following
condition:  there exists an
$s''$ such that for every 
norm $\gamma\in F_A(k'')$ the value of $\gamma$ on each coefficient of
$P$ is at most $s''$; in other words, form some $s''$ the
coefficients of $P$ belong to $\B_{A}(k'',s'')$.  If indeed
this is the case, the proof is complete -- $\B_{A}(k'',s'')$ is a
finite set, so only finitely many polynomials $P$ can appear in
(\ref{bbk}) which, combined with our remarks above concludes the
proof.  

It remains to prove the existence of $s''$.  Let
\begin{equation*}
  s' = s \cdot \max \{\, \gamma(Q) \colon
          \text{$Q\in\mathcal Q_k$, $\gamma\in F_B(k)$} \,\}
\end{equation*}
so that for an element of $\B_{B}(k,s)$ written in the form
(\ref{bbk}) we have $\gamma_i(P)\leq s'$ for every 
$\gamma\in F_A(k'')$ and $i=0,\dots,k'$.
Now, the linear transformation 
\begin{equation*}
   P\longmapsto (P(t_0),\dots,P(t_{k'})), \qquad
   \CC[X]_{k'} \to {\bigoplus}_0^{k'} \CC
\end{equation*}
is invertible -- identifying a polynomial
$P\in \CC[X]_{k'}$ with the column vector formed by its coefficients
it is given by the Vandermonde matrix corresponding to the
distinct transcendentals $t_0,\dots,t_{k'}$.  The condition that
$\gamma_i(P)\leq s'$ for every $i=0,\dots,k'$ and $\gamma\in F_A(k'')$
means that the coefficients of the polynomial $\phi_\gamma(P)$ lie in the
subset of the domain mapping into the compact subset of the
range defined by the requirement that the absolute value of each entry
is at most $s'$.  This is a compact set so that there is an $s''$ such that the
absolute value of the coeffecients of the polynomial $\phi_\gamma(P)$
are bounded by $s''$; in other words, the coefficients of $P$ are in
the set $\B_{A}(k'',s'')$ as required.  
\end{proof}

\begin{lem}[Finite extensions]
Strong discrete embeddability is stable under the formation of finite
extensions. 
\end{lem} 
\begin{proof}
We shall show that a finite extension of an SDE field is SDE.  To this
end, let $L$ be a finite extension of an SDE field $K$.  As a
subfield of an SDE field is itself SDE we may, enlarging $L$ as
necessary, assume that $L$ is a finite normal extension of $K$.

Let $B$ be a finitely generated subring of $L$.  Fix a basis of the
$K$-vector space $L$ and let $A$ be a finitely generated
subring of $K$ containing the matrix entries
of each element of $B$, viewed as a $K$-linear transformation of $L$.
This is possible -- we may take for $A$ any subring containing the
matrix entries of a finite generating set for $B$.

According to Definition~\ref{def:sde} applied to the subring $A$ of
$K$, we obtain (finitely many) discrete norms $N_A$ and (countably
many) archimedean norms $M_A$.  Now, every
discrete norm on $K$ admits at least
one extension to a discrete norm on $L$; a similar statement applies
to archimedean norms.  See \cite[Chapter 12]{Lang}.  Moreover, the
finite group $\Aut_K(L)$ of $K$-automorphisms of $L$ acts on
the set of extensions of each individual norm on $K$.  

Let $N_B$ 
be a (finite) set of discrete norms on $L$ comprising exactly one
$\Aut_K(L)$-orbit of extensions of each norm in $N_A$; let $M_B$ be a
(countable) set of archimedean norms on $L$ defined similarly with
respect to $M_A$.  Finally, for each $k$ let
\begin{equation*}
  F_B(k) = \{\, \gamma\in M_B \colon
     \text{$\gamma$ extends a norm in $F_A(k')$} \,\};
\end{equation*}
here 
$k' = \max \{\, |f(x_0,\dots,x_n)| \,\}$, where $n$ is the degree of
the extension and the maximum is over
all elementary symmetric functions $f$ and all tuples of real numbers
$x_0,\dots,x_n$ each of which has absolute value at most $k$.  Each
$F_B(k)$ is a finite set of archimedean norms invariant under the
action of $\Aut_K(L)$.

Let $k$ and $s>0$ be given.  We must show that $\B_{B}(k,s)$ is
finite.  We shall do this by showing that the coefficients of the
characteristic polynomial of each element of $\B_{B}(k,s)$, again
viewed as a $K$-linear transformation of $L$, belong to the finite set
$\B_{A}(k',s')$ where $s'$ is defined in terms of $s$ as $k'$ was in
terms of $k$.  Thus, every element of $\B_{B}(k,s)$ is the root of one
of finitely many polynomials and $\B_{B}(k,s)$ is itself finite.

Let now $b\in\B_{B}(k,s)$.  Since the extension is normal, the minimal
polynomial of $b$ (in the sense of field theory) splits in $L$ and the
group $\Aut_K(L)$ acts transitively on its roots.  The minimal
polynomials of $b$ in the sense of field theory and as a $K$-linear
transformation of $L$ agree.  Hence the group $\Aut_K(L)$ acts
transitively on the roots of the characteristic polynomial $\Xi_b $ of
$b$.\footnote{Recall that the characteristic polynomial and the
  minimal polynomial of a linear transformation have the same roots
  (in the algebraic closure of the ground field), possibly with
  different multiplicities.}  It follows that all the roots of the
$\Xi_b$ belong to $\B_{B}(k,s)$.  Since the coefficients of
$\Xi_b$ are symmetric functions of degree $\leq n$ of the roots, every
such coefficient belongs to $\B_{A}(k',s')$ by virtue of the
definitions of $k'$ and $s'$.
\end{proof}

\subsection{The general linear group}
\label{lineargroupSubsection}

Let $\gamma$ be a norm on a field $K$.  Following Guentner, Higson and
Weinberger define a (pseudo)-length function $\ell_{\gamma}$ on
$\GL(n,K)$ as follows: if $\gamma$ is non-archimedean
\begin{equation}
\label{ell-discrete}
  \ell_{\gamma}(g)=\log\max_{ij} \{\, \gamma(g_{ij}),\gamma(g^{ij}) \,\},  
\end{equation}
where $g_{ij}$ and $g^{ij}$ are the matrix coefficients of $g$ and
$g^{-1}$, respectively; if $\gamma$ is archimedian, arising from an embedding
$K\hookrightarrow \CC$ then 
\begin{equation}
\label{ell-arch}
  \ell_{\gamma}(g) = \log\max \{\, \|g\|,\|g^{-1}\| \,\},
\end{equation}
where $\|g\|$ is the norm of $g$ viewed as an element of $\GL(n,\CC)$, and
similarly for $g^{-1}$.  The following proposition is central to our
discussion of linear groups.

\begin{prop}
\label{Gfad}
Let $\gamma$ be an archimedean or a discrete norm on a field $K$.  The
group $\GL(n,K)$, equipped with the 
\textup{(}left-invariant pseudo-\textup{)}metric induced by $\ell_\gamma$,
is in $\C_{\fin}$.
\end{prop}

\begin{proof}[Proof of Proposition~\ref{Gfad} \textup{(}archimedean case\textup{)}]
  The result follows immediately from the corresponding result for
  $\GL(n,\CC)$; indeed, the metric on $\GL(n,K)$ is the subspace
  metric it inherits from an embedding into $\GL(n,\CC)$.  For
  $\GL(n,\CC)$ the result follows from standard arguments, once we observe
  that the length function (\ref{ell-arch}) is continuous, hence
  bounded on compact sets, and proper, meaning that bounded sets are
  compact.  In brief, $\GL(n,\CC)$ is coarsely equivalent to the
  subgroup $T(n,\CC)$ of all upper triangular matrices and a fibering
  argument based on Theorem~\ref{fibering} show thats the solvable
  group $T(n,\CC)$ has finite decomposition complexity.
\end{proof}

The discrete case is more subtle than the archimedean case, primarily
because we do not assume that $K$ is locally compact.  In this case the
result was proven by Matsnev \cite{Ma}.  We shall present a 
simplified proof, based essentially on the same ideas.

Let $\gamma$ be a discrete norm on a field $K$ and fix a uniformizer $\pi$.
For the proof we shall introduce some subgroups of $\GL(n,K)$.  Let
$D$ denote the subgroup of diagonal matrices with powers of the
uniformizer on the diagonal and let $U$ denote the unipotent upper
triangular matrices.  Observe that $D$ normalizes $U$ so that $T=DU$
is also a subgroup (namely the group upper triangular matrices).
Restrict the length function $\ell_{\gamma}$ to each subgroup and
equip each with the associated (left-invariant pseudo-)metric (which
is in fact the subspace pseudo-metric from $G$).

\begin{lem}
The group $U$ has asymptotic dimension zero.  In particular, $U\in\C_1$.
\end{lem}
\begin{proof}
The {\it dilation by \textup{(}a nonzero\textup{)} $\theta\in K$\/} is 
the function $\Theta: U\to U$ defined by 
\begin{equation*}
  \Theta(u)_{ij} = \theta^{j-i} u_{ij};
\end{equation*}
the entries on the $k^{th}$-superdiagonal of $n$ are multiplied
by $\theta^{k}$.  (For $k=0,\dots,n-1$ the {\it
  $k^{th}$-superdiagonal\/} of an $n\times n$ matrix consists of the
positions $(i,j)$ for which $j-i=k$.)
The formula for matrix multiplication shows that
$\Theta$ is an endomorphism of $U$.  Further, it is an automorphism with
inverse the dilation by $\theta^{-1}$.

Fix $\theta\in K$ of norm greater than one -- the inverse of a
uniformizer will do.  Let $U_0$ be the subgroup of $U$ comprised
of elements of length zero, and define a sequence of subgroups of $U$ by
$U_{k}=\Theta(U_{k-1})$.  We shall show that
\begin{equation}
\label{nested}
   B(1,k\log \gamma(\theta))\subset U_k \subset B(1,k(n-1)\log \gamma(\theta)).
\end{equation}
The lemma follows immediately.  Indeed, $U$ is the union of the cosets
of $U_k$ and the family of these cosets is both bounded and $r$-disjoint,
provided $k\log \gamma(\theta)>r$.

In order to verify (\ref{nested}) observe that the length function on
$U$ is given by
\begin{equation}
\label{length-on-U}
  \ell_{\gamma}(u) = \log \max_{i<j} \{\, 1,\gamma(u_{ij}), \gamma(u^{ij}) \,\}.
\end{equation}
For the first inclusion in (\ref{nested}) suppose 
$\ell_{\gamma}(u)\leq k\log \gamma(\theta)$ so that in particular 
$\gamma(u_{ij})\leq \gamma(\theta)^k$ for all $i<j$.  The non-diagonal $(i,j)$
entry of $\Theta^{-k}(u)$ is $u_{ij}\theta^{k(i-j)}$ so that
each has norm at most one.  Elementary properties of the norm and
(\ref{length-on-U}) show that
this implies $\Theta^{-k}(u)\in U_0$, or $u\in U_k$.

The second inclusion in (\ref{nested}) follows by induction from
\begin{equation*}
  \ell_{\gamma}(\Theta(u)) \leq \ell_{\gamma}(u) + (n-1)\log \gamma(\theta).
\end{equation*}
To verify this inequality, note that the non-diagonal $(i,j)$ entry of
$\Theta(u)$ is $u_{ij}\theta^{j-i}$ which has norm bounded by
$\gamma(u_{ij})\gamma(\theta)^{n-1}$.  Since $\Theta$ is an automorphism a similar
statement applies to the entries of $\Theta(u)^{-1}=\Theta(u^{-1})$.
The inequality now follows from (\ref{length-on-U}).
\end{proof}

\begin{lem}
The group $T$ is in $\C_{n+1}$.
\end{lem}
\begin{proof}
  Observe that $D\cong \Z^{n}$, and that the restriction of
  $\ell_{\gamma}$ to $D$ is a proper length function -- indeed it
  corresponds (up to a multiplicative factor) with the supremum norm
  on $\Z^n$:
\begin{equation*}
  \ell_{\gamma}(a) = \max |k_i| \cdot \log \gamma(\pi^{-1}),
\end{equation*}
where $a$ is the diagonal matrix with entries $\pi^{k_i}$.  Hence $D$
is in $\C_n$.  It remains to check, as an application of fibering,
that $T$ is indeed in $\C_{n+1}$.\footnote{Since $D\subset T$
isometrically, if $T$ is in $\C_\alpha$ then necessarily
$\alpha\geq n$.  An argument more refined than the one we
present here achieves this bound:  indeed $T\in \C_n$.}

We require two observations.  First, the map $T\to D$ associating to
each matrix in $T$ the matrix of its diagonal entries is a contraction.
Indeed, it is a homomorphism and from the definition of $\ell_{\gamma}$ we see
that it decreases length.  Second, if $B\subset D$ is a bounded subset
and $b_1\in B$ then the subset $b_1 U\subset BU$ is $\diam(B)$-coarsely
dense.  Indeed, if $bu\in BU$ then $d(bu,bub^{-1}b_1)\leq \diam(B)$ and,
since $D$ normalizes $U$, 
\begin{equation*}
  bub^{-1}b_1 = b_1 (b_1^{-1}b)u(b^{-1}b_1)\in b_1U.
\end{equation*}
We conclude by applying the Fibering
Theorem~\ref{fibering} or, more accurately, the subsequent
Remark~\ref{fiberingremark}, to the map $T\to D$.
\end{proof}

\begin{proof}[Proof of Proposition~\ref{Gfad} \textup{(}discrete case\textup{)}]
  The inclusion of $T$ in $G$ is isometric.  Further, it
  is metrically onto in the sense that every element of $G$ is at
  distance zero from an element of $T$.  Indeed, let $H$ be the subgroup
  of those $g\in \GL(n,K)$ for which the entries of $g$ and $g^{-1}$ are
  in $\O$.  Then $G=TH$ \cite[Lemma~4.5]{GHW} and elementary calculations
  show that every $h\in H$ has length zero.  Hence, if $g=th$ then
  $d(t,g)=\ell(h)=0$.
\end{proof}

\subsection{Finite decomposition complexity}

We have previously reduced Theorem~\ref{linear} to the case of
$G=\GL(n,A)$, where $A$ is a finitely generated domain.  Denoting the
fraction field of $A$ by $K$, our strategy
is to embed $\GL(n,A)$ into the product of several copies of
$\GL(n,K)$ equipped with metrics associated to various norms.  The
proof rests on a permanence property summarized in the following
lemma.

\begin{lem}
\label{oddpermanence}
Let $G$ be a countable discrete group.  Suppose there exists a
\textup{(}pseudo-\textup{)}length  
function $\ell'$ on $G$ with the following properties:
  \begin{ilist}
    \item $G$ is in $\C_{\fin}$ with respect to the associated 
           \textup{(}pseudo-\textup{)}metric $d'$ 
    \item $\forall\,r>0$ $\exists\,\ell_r$, a 
           \textup{(}pseudo-\textup{)}length function on
      $G$, for which 
      \begin{alist}
         \item $G$ is in $\C_{\fin}$ with respect to  the associated
             \textup{(}pseudo-\textup{)}metric $d_r$,
         \item $\ell_r$ is proper when restricted to $B_{\ell'}(r)$.
      \end{alist}
   \end{ilist}
Then $G$ has finite decomposition complexity, and indeed 
$G\in \C_{\omega+\fin}$.  
\end{lem}

\noindent
Condition \textup{(}\rm{ii}\textup{)} in the lemma means precisely that
$B_{\ell_r}(s)\cap B_{\ell'}(r)$ is finite for every $s>0$.

\begin{proof}
  Fix a proper length function $\ell$ on $G$, with associated metric
  $d$.  By Proposition~\ref{fiberingprop}, applied to the action of
  $G$ on the metric space $(G,d')$, it suffices to show that for every
  $r>0$ the ball $B_{\ell'}(r)$ is in $\C_{\fin}$ when equipped with
  the metric $d$.

Let $r>0$.  Obtain $\ell_{2r}$ as in the statement.  The ball
$B_{\ell'}(r)$ is in $\C_{\fin}$ with respect to the metric $d_{2r}$.
Thus, it remains to show that the metrics $d$ and $d_{2r}$ on
$B_{\ell'}(r)$ are coarsely equivalent.

Since $\ell$-balls in $G$ are finite, we easily see that for every $s$
there exists $s'$ such that if $d(g,h)\leq s$ then $d_{2r}(g,h)\leq s'$;
this holds for every $g$ and $h\in G$.  Conversely, for every $s$ the
set $B_{\ell'}(2r)\cap B_{\ell_{2r}}(s)$ is finite by assumption, and we
obtain $s'$ such that for every $g$ in this set $\ell(g)\leq s'$.  If
now $g$ and $h\in B_{\ell'}(r)$ are such that $d_{2r}(g,h)\leq s$ then
$g^{-1}h\in B_{\ell'}(2r)$ and
\begin{equation*}
  d(g,h) = \ell(g^{-1}h) \leq s'.
\end{equation*}
\end{proof}

\begin{proof}[Proof of Theorem~\ref{linear}]
Let $A$ be a finitely generated domain, $K$ the fraction field of $A$
and $G=\GL(n,A)$.  (We have previously reduced the theorem to this
case.)
Obtain a finite family 
$N_A=\{\, \gamma_1,\dots,\gamma_q \,\}$ of discrete norms on
$K$ as in the definition of strong discrete embeddability.  For each
norm $\gamma_i$ 
we have the corresponding length function $\ell_{\gamma_i}$ and metric on $\GL(n,K)$
defined as in (\ref{ell-discrete}). Define a length function on $G$ by
\begin{equation*}
  \ell' = \ell_{\gamma_1} + \cdots + \ell_{\gamma_q}.
\end{equation*}
Thus, $G$ is metrized so that the diagonal embedding
\begin{equation*}
  G \hookrightarrow \GL(n,K) \times \cdots \times \GL(n,K)
\end{equation*}
is an isometry when the $i^{\mbox{\tiny{th}}}$ factor in the product is
equipped with the metric associated to the norm $\gamma_i$ and the product is
given the sum metric.  Equipped with this metric $G$ is in $\C_{\fin}$
by Proposition~\ref{Gfad}, and Remark~\ref{fiberingremark}.  To apply the lemma, we 
shall study the balls $B_{\ell'}(r)$ of the identity in $G$.

Let $r=e^k$.  Obtain a family of archimedean norms 
$F_{A}(k)$ as in the definition of strong
discrete embeddability.  For each we have the
corresponding length function and metric on $\GL(n,K)$ defined as in
(\ref{ell-arch}).  Define a length function on $G$ by
\begin{equation*}
  \ell_r = \sum_{\gamma\in F_A(k)}\ell_{\gamma}.
\end{equation*}
Thus, $G$ is metrized so that the diagonal embedding 
\begin{equation*}
  G\hookrightarrow \GL(n,K) \times \cdots \times \GL(n,K)
\end{equation*}
is an isometry when each factor in the product is
equipped with the metric associated to the corresponding norm $\gamma$, and the
product is given the sum metric.   Equipped with this metric $G$ is in $\C_{\fin}$
by Proposition~\ref{Gfad}, and Remark~\ref{fiberingremark}.  To apply the lemma, we 
shall study the balls $B_{\ell'}(r)$ of the identity in $G$.

It remains only to show that for every $s>0$ the set
$B_{\ell_r}(s)\cap B_{\ell'}(r)$
is finite.  Suppose $g$ is in this set.  From the definitions of the
length functions it follows that the entries of $g$ and $g^{-1}$ satisfy
inequalities 
\begin{equation*}
  \gamma(g_{ij})\leq r, \quad  \gamma(g^{ij})\leq r,
\end{equation*}
for $\gamma\in N_A$, and also the inequalities
\begin{equation*}
  \gamma(g_{ij})\leq s, \quad  \gamma(g^{ij})\leq s,
\end{equation*}
for $\gamma\in F_A(k)$.  But, these norms were chosen according to the
definition of strong discrete embeddability, so that the subset of
those elements of $A$ satisfying these inequalities is finite.  In
particular, the number of matrices containing only these elements as
their entries is finite and the proof of the general case is complete.
Further, in the 
case of positive characteristic, there are no archimedean norms and
the above inequalities show that $B_{\ell'}(r)$ is already finite for
every $r$.  In this case, we conclude that $G$ belongs to $\C_{\fin}$
so that by Theorem~\ref{fad&fdc} it has finite asymptotic dimension.
\end{proof}

\begin{remark}
  Essentially, the proofs of Lemma~\ref{oddpermanence} and
  Theorem~\ref{linear} yield the following result: if the finitely
  generated domain $A$ has characteristic zero there is an action of
  $\GL(n,A)$ on a metric space in $\C_{\fin}$ such that each coarse
  stabilizer is in $\C_{\fin}$. 
\end{remark}
\section{Further examples}
\label{MoreExamplesection}

Additional examples of groups having finite decomposition complexity
are readily exhibited based on our results. In this section, we prove
that all countable elementary amenable groups, all countable subgroups
of almost connected Lie groups, and all countable subgroups of $\GL(n,
R)$ for any commutative ring $R$ with unit have finite decomposition
complexity.

The class of {\it elementary amenable groups\/} is the smallest class
of countable discrete groups containing all finite groups and all
(countable) abelian groups, and closed under the formation of
subgroups, quotients, extensions and direct unions.

\begin{props}[\cite{C}]
The class of elementary amenable groups is the smallest class of
countable discrete groups containing all finite groups and all
\textup{(}countable\textup{)} abelian groups and closed under the
formation of extensions and direct unions.
\end{props}
\begin{proof}[Sketch of proof]
  Define a class of groups $\A$ by transfinite recursion as follows:
  $\A_0$ is the class of all finite and countable abelian groups; for a
  successor ordinal $\alpha$ define
  $\A_{\alpha}$ to be the class of all groups obtained as a (countable) direct
  union or extension of groups in $\A_{\alpha-1}$; for a limit ordinal
  $\alpha$ define $\A_\alpha = \cup_{\beta<\alpha}\A_\beta$; finally,
  $\A$ is the collection of groups belonging to some $\A_\alpha$.  

  From its construction $\A$ is closed under extensions and (countable)
  direct unions, and is clearly contained in the collection of
  elementary amenable groups.  It remains to show that $\A$ is closed
  under subgroups and quotients.  Indeed, it is readily verified by
  transfinite induction that each $\A_\alpha$ is closed under these
  operations. 
\end{proof}

\begin{thms}
Elementary amenable groups have finite decomposition complexity.
\end{thms}
\begin{proof}
We have observed that the class of countable discrete groups having
finite decomposition complexity is closed under the formation of
extensions and direct unions.  Finite groups
have finite decomposition complexity, as do (countable) abelian groups.
Indeed, a (countable) abelian group is the direct union of its finitely generated
subgroups which, according to their general structure theory, have
finite decomposition complexity .
\end{proof}

\begin{ques}
  Does every countable amenable group have FDC?  In particular, does a
  Grigorchuk group of intermediate growth have FDC?
\end{ques}

\begin{thms}
A countable subgroup of an almost connected Lie group has finite
decomposition complexity.  \qed
\end{thms}
\begin{proof}
A group as in the statement is realized as an extension with finite
quotient and with kernel a subgroup of a {\it connected\/} Lie group.
A subgroup of a connected Lie group is realized as an extension with
linear quotient and abelian kernel.  Thus, the result follows from the
stability of FDC under extensions.
Compare \cite[Thm.~6.5]{GHW}.
\end{proof}


\begin{thms}\label{FDCringThm}
Let $R$ be a commutative ring with unit.  A countable subgroup of
$\GL(n,R)$ has finite decomposition complexity.
\end{thms}

The essential piece of commutative algebra we require is summarized in
the following lemma.

\begin{lems}
Let $R$ be a finitely generated commutative ring with unit and let $\n$
be the nilpotent radical of $R$,
\begin{equation*}
  \n = \{\, r\in R : \text{$\exists n$ such that $r^n=0$} \,\}.
\end{equation*}
The quotient ring $S=R/\n$ contains a finite number of prime ideals
$\p_1,\dots,\p_n$ such that the diagonal map
\begin{equation*}
  S \to S/{\p_1} \oplus \cdots \oplus S/{\p_n}
\end{equation*}
embeds $S$ into a finite direct sum of domains. 
\end{lems}
\begin{proof} 
  This classical fact is a consequence of the Associated Prime Theorem 
  which states that the set of associated primes of a finitely
  generated module over a Noetherian ring is finite
  \cite[Thm.~3.1]{Eisenbud}. Here, the module 
  is the ring itself which is Noetherian since it is finitely
  generated. The mentioned theorem then says that $R$ has finitely
  many minimal prime ideals $\p_1,\dots,\p_n$. The conclusion follows
  from the fact that their intersection is $\n$.   
\end{proof}

\begin{proof}[Proof of Theorem~\ref{FDCringThm}]
In views of Proposition~\ref{limits}, it is enough to treat the case
of $\GL_n(R)$, where $R$ is finitely generated.  With $\n$ and
$S$ as in the previous lemma, we have an exact sequence
\begin{equation*}
  1 \to I+M_n(\n) \to \GL(n,R) \to \GL(n,S) \to 1,
\end{equation*}
in which $I+M_n(\n)$ is nilpotent, and therefore has finite
decomposition complexity by Corollary~\ref{extension}.  In the
notation of the previous lemma, we have
\begin{equation*}
  \GL(n,S) \to \GL(n,S/{\p_1}) \times \cdots \times \GL(n,S/{\p_n}).
\end{equation*}
So, the quotient has finite decomposition complexity by our earlier
results. 
\end{proof}

\section{Decomposition Complexity and Topological Rigidity}
\label{TopologicalResultsSection}

This section is organized into two parts.  In the first part we shall
state two essential results,
Theorems~\ref{AsymptoticVanishingK-theoryThm} and
\ref{AssemblyMapThm}, the proofs of which are defered to later
sections.  In the second part we shall discuss applications to
topological rigidity.  We shall begin by describing the bounded
category, a natural framework in which to discuss bounded rigidity.
We shall then state and prove our results concerning the bounded Borel
and bounded Farrell-Jones $L$-theory isomorphism conjectures for
spaces with finite decomposition complexity,
Theorems~\ref{BoundedBorelThm} and \ref{BoundedFarrelJonesThm},
respectively.  Finally, from these we deduce concrete applications to
topological rigidity.

\subsection{Two main results}

Throughout, we shall work with a metric space $\Gamma$ having {\it
  bounded geometry\/}: for every $r>0$ there exists $N=N(r)$ such that
every ball of radius $r$ contains at most $N$ elements. In several
places the weaker hypothesis of {\it local finiteness\/} would suffice:
every ball contains finitely many elements.

\begin{defn} \label{RipsDef}
  For $d\geq 0$ we define the Rips complex $P_d(\Gamma)$ to be the
  simplicial polyhedron with vertex set $\Gamma$, and in which a finite
  subset
  $\{ \gamma_0,\dots,\gamma_n \} \subseteq \Gamma$ spans a simplex
  precisely when $d (\gamma_i,\gamma_j) \leq d$ for all 
  $0 \leq i, j \leq n$.
\end{defn}

If $\Gamma$ has bounded geometry the Rips complex is finite dimensional,
with dimension bounded by $N(d)-1$; if $\Gamma$ is merely locally finite
the Rips complex $P_d(\Gamma)$ is a locally finite simplicial complex.

There are in general several ways to equip the Rips complex with a
metric.  The {\it simplicial metric\/} is the metric induced by the
(pseudo) Riemannian metric whose restriction to each $n$-simplex is
the Riemannian metric obtained by identifying the $n$-simplex with the
standard simplex in the Euclidean space $\R ^{n}$.  By convention, the
distance between points in different connected components of
$P_d(\Gamma)$ is infinite.  Equipped with the simplicial metric the
Rips complex is a {\it geodesic space\/} in the sense that every two
points (at finite distance) are joined by a geodesic path.

Our first essential result is a vanishing
result for the Whitehead and algebraic $K$-theory groups.  To state
the result we introduce the following notation: for a locally compact
metric space $X$ and for each $\delta\geq 0$ and $i\geq 0$ the
$\delta$-controlled locally finite Whitehead group is denoted
$Wh^{\delta}_{1-i}(X)$; the $\delta$-controlled reduced locally finite
algebraic $K$-theory group is denoted
$\widetilde{K}_{-i}^{\delta}(X)$.  Both groups are defined in
\cite{RY1}.\footnote{The group we denote $\widetilde{K}_0^\delta(X)$ is the group
  $\widetilde{K}_0(X,p_X,0,\delta)$ defined on page 14 of \cite{RY1}, taking $p_X$
to be the identity map $X\to X$; our $\widetilde{K}_{-i}^{\delta}(X)$
is then defined to be $\widetilde{K}_0^{\delta}(X\times \R^i)$.  The
group we denote
$Wh^\delta(X)$ is the group $Wh(X,p_X,1,\delta)$ defined on page 22
of \cite{RY1}, where again $p_X$ is the identity map $X\to X$;
$Wh^\delta_{1-i}(X)$ is then defined to be $Wh^\delta_{1-i}(X\times \R^i)$.}

We then define, for each $i\geq 0$, the bounded locally finite
Whitehead group, and bounded reduced locally finite algebraic
$K$-theory group as follows:
\begin{equation*}
  Wh_{1-i}^{bdd} (P_d(\Gamma))=
    \lim_{\delta\rightarrow \infty}Wh_{1-i}^{\delta} (P_d(\Gamma))
\end{equation*}
\begin{equation*}
  \tilde{K}_{-i}^{bdd}(P_d(\Gamma))=
\lim_{\delta\rightarrow \infty}\tilde{K}_{-i}^{\delta}(P_d(\Gamma)).
\end{equation*}

\begin{thm}\label{AsymptoticVanishingK-theoryThm} 
  Let $\Gamma$ be a bounded geometry metric space.  Let
  $\widetilde{K}_{-i}^{bdd}(P_d(\Gamma))$ denote the reduced bounded
  locally finite algebraic $K$-theory group and let
  $Wh_{1-i}^{bdd}(P_d(\Gamma))$ denote the bounded locally finite
  Whitehead group of the Rips complex $P_d(\Gamma)$.  If $\Gamma$ has
  finite decomposition complexity then
  \begin{eqnarray*}
    \lim_{d\rightarrow \infty}\widetilde{K}_{-i}^{bdd}(P_d(\Gamma)) 
                    &= 0 \\ 
    \lim_{d\rightarrow\infty}Wh_{1-i}^{bdd}(P_d(\Gamma)) &= 0 
  \end{eqnarray*}
for each $i\geq 0$.
\end{thm}

Our second essential result asserts that an appropriate
assembly map is an isomorphism.  To state the result we introduce the
following notation: $\L(e)$ denotes the simply connected surgery
spectrum with $\pi _n(\L(e))=L_n(\Z \{e\})$; $L_n^{bdd}(X)$ denotes
the bounded, locally finite and free $L$-theory of the locally compact
metric space $X$.  Recall that $L_n^{bdd}(X)$ is defined using locally
finite, free geometric modules and that a geometric module is locally
finite if its support is locally finite.
More precisely, for a locally compact
metric space $X$ and for each $\delta\geq 0$ and $n\geq 0$ the
$\delta$-controlled locally finite and free $L$-group in degree $n$ is denoted
$L^{\delta}_{n}(X)$.  This group is defined in
\cite{RY2}.\footnote{The group we denote $L^\delta_n(X)$ corresponds
  to the $\delta$-controlled locally finite and free $L$-theory group
  $L_n^{\delta,\delta}(X;p_X,\Z)$ in \cite{RY2}, where
  again $p_X$ is the identity map $X\to X$.}
We then define the bounded locally finite $L$-group as follows:
\begin{equation*}
L_{n}^{bdd} (P_d(\Gamma))=
    \lim_{\delta\rightarrow \infty}L_{n}^{\delta} (P_d(\Gamma)).  
\end{equation*}

\begin{thm}\label{AssemblyMapThm} 
  Let $\Gamma$ be a metric space with bounded geometry and finite
  decomposition complexity.  The assembly map
\begin{equation*}
  A : \lim_{d\to\infty} H_n(P_d(\Gamma),\L(e))\to 
           \lim_{d\to\infty}L_n^{bdd}(P_d(\Gamma))
\end{equation*}
is an isomorphism. 
\end{thm}

In the statement, the domain of assembly is the locally finite
homology of the Rips complex with spectrum $\L(e)$, and the range
is the bounded, locally finite and free $L$-theory of the same
Rips complex.

The proofs of our essential results,
Theorems~\ref{AsymptoticVanishingK-theoryThm} and
\ref{AssemblyMapThm}, are accomplished using appropriate controlled
Mayer-Vietoris arguments and shall be presented in
Sections~\ref{vanishing} and \ref{assembly}, respectively. Our proofs
require the full strength of the finite decomposition complexity
hypothesis; we are unable to prove the results under the hypothesis of
weak finite decomposition complexity. The remainder of the present
section is devoted to a description of how the results themselves are
used to deduce the rigidity statements in the introduction.

\subsection{The bounded category}

Being invariant under coarse equivalence, finite decomposition
complexity is well-adapted to a topological setting where the geometry
appears only `at large scale' and the topological properties are
`uniformly' locally trivial.  These ideas are formalized in the {\it
  bounded category}.

A {\it coarse metric manifold\/} is a topological manifold $M$
equipped with a continuous {\hbox{(pseudo-)}} metric in which balls are
precompact.  Although Riemannian manifolds, equipped with the path
length metric, are motivating examples of coarse metric manifolds, we
want to make clear that our definition entails {\it no\/} assumption
on the metric at `small scale' and that the manifold $M$ is {\it
  not\/} assumed to be smooth.  A continuous map $f:M\to N$, between
two coarse metric manifolds is {\it bounded\/} if there exists a
coarse equivalence $\phi: N\to M$ and a constant $K>0$ such that
\begin{equation*}
 d(x,\phi\circ f(x))\leq K 
\end{equation*}
for all $x\in M$.  Coarse metric manifolds and bounded continuous maps
comprise the {\it bounded category\/}.\footnote{In \cite{CFY}, the
  authors give an essentially equivalent definition of the bounded
  category in which an auxiliary metric space $X$ is introduced. An
  object is a pair $(M,p)$ where $p:M\to X$ has precompact
  preimages. To obtain a coarse metric manifold, one must merely
  pull back the metric from $X$ to $M$.}

Before discussing rigidity in the bounded category, we must introduce
appropriate notions of homeomorphism and homotopy.  A {\it bounded
  homeomorphism\/} between coarse metric manifolds is a map $M\to N$
which is simultaneously a homeomorphism and a coarse equivalence.
These are the isomorphisms in the bounded category.

Two bounded continuous maps $f$, $g:M\to N$ are {\it boundedly
  homotopic\/} if there exists a bounded homotopy between them; in
other words, if there exists a continuous map $F:M\times [0,1]\to N$,
for which $F(0,\cdot)=f$, $F(1,\cdot)=g$ and for which the family
$(F(t,\cdot))_{t\in [0,1]}$ is bounded (uniformly in $t$, in the
obvious sense).  A bounded continuous map $f:M\to N$ is a {\it bounded
  homotopy equivalence} if there exists a bounded continuous map
$g:N\to M$ such that the compositions $f\circ g$ and $g\circ f$ are
boundedly homotopic to the identity.

\begin{defn}
  A coarse metric manifold $M$ is {\it boundedly rigid\/} if the
  following condition holds:  every bounded homotopy equivalence
  $M\to N$ to another coarse metric manifold is boundedly homotopic to
  a (bounded) homeomorphism.
\end{defn}

A coarse metric manifold $M$ is {\it uniformly contractible\/} if for
every $r>0$, there exists $R\geq r$ such that every ball in $M$ with
radius $r$ is contractible to a point within the larger ball of radius
$R$ and the same center. Uniform contractibility is invariant under
bounded homotopy equivalence.

A coarse metric manifold has {\it bounded geometry\/} if there exists
$r>0$ with the following property: for every $R>0$ there exists $N>0$
such that every ball of radius $R$ is covered by $N$ or fewer balls of
radius $r$.\footnote{Traditionally, a Riemannian manifold is said to
  have bounded geometry if its curvature is bounded from below and its
  radius of injectivity is bounded away from zero. Such local
  conditions are known to imply our global condition.}

Perhaps the most important, and motivating, example of a coarse metric
manifold is the universal cover $\tilde{M}$ of a closed (topological)
manifold $M$. To realize the structure of a coarse metric manifold on
$\tilde{M}$ we need only equip it with a continuous $\Gamma$-invariant
pseudo-metric in which balls 
are precompact, where $\Gamma$ is the fundamental group of $M$.
Let $c$ be a continuous and compactly supported cut-off function on $\tilde{M}$ 
in the sense that $c$ is  a non-negative function on $\tilde{M}$ satisfying
\begin{equation*}
  \sum_{g \in \Gamma} c(gx)=1
\end{equation*}
for all $x\in \tilde{M}$.
We define a pseudo-metric $d$ on $\tilde{M}$ as follows:
\begin{equation*}
  d(x,y)=\sum_{g, h\in \Gamma} c(g^{-1}x)c(h^{-1}y) d_{\Gamma}(g,h)
\end{equation*}
for all $x$ and $y$ in $ \tilde{M}$, where $d_{\Gamma}$ is a word
metric on $\Gamma$.  Equipped with this pseudo-metric, $\tilde{M}$ is
coarsely equivalent to $\Gamma$. As a consequence, the coarse metric
manifold structure on $\tilde{M}$ is independent of the choices made
in the construction.


\subsection{Application to bounded rigidity}

The {\it bounded Borel isomorphism conjecture\/} asserts that an
appropriate assembly map is an isomorphism. Precisely this conjecture
asserts that for a locally finite metric space $\Gamma$ the assembly
map 
\begin{equation} 
\label{bbc}
  A : \lim_{d\to\infty} H_n(P_d(\Gamma),\L(e))\to 
            \lim_{d\to\infty}L_n^{bdd, s}(P_d(\Gamma))
\end{equation}
is an isomorphism: as in the previous section, the domain of assembly
is the locally finite homology of the Rips complex of $\Gamma$ with
spectrum $\L(e)$, the simply connected surgery spectrum with
$\pi_n(\L(e))=L^s_n(\Z\{e\})= L_n(\Z \{e\})$; the range of assembly is
the bounded simple $L$-theory of the Rips complex of $\Gamma$ defined
using locally finite free geometric modules.


\begin{thm}
\label{BoundedBorelThm}
  The bounded Borel isomorphism conjecture is true for metric
  spaces with bounded geometry and finite decomposition complexity.
\end{thm}
\begin{proof} 
  By Theorem \ref{AsymptoticVanishingK-theoryThm} and the
  Ranicki-Rothenberg sequence in the controlled setting \cite{FP}, we
  have
\begin{equation*}
  \lim_{d\rightarrow\infty}L^{bdd, s}(P_d(\Gamma))\cong 
         \lim_{d\rightarrow\infty}L^{bdd}(P_d(\Gamma)).
\end{equation*}
The result now follows from Theorem \ref{AssemblyMapThm}.
\end{proof}

As is the case for the classical Borel isomorphism conjecture, the
bounded Borel isomorphism conjecture has strong topological
implications. These implications, which we now describe, are to
questions of topological rigidity in the bounded category of coarse
metric manifolds.  Our principal result in this direction is the
following theorem.

\begin{thm}[Bounded Rigidity Theorem]
\label{thm:brig}
 A uniformly contractible coarse metric manifold with bounded
 geometry, finite decomposition complexity, and dimension at least
 five is boundedly rigid.
\end{thm}

While we shall present the proof of this theorem at the end of this
subsection let us, for the moment, apply it in the case of the
universal cover of a closed aspherical manifold to deduce the
following result stated in the introduction.

\begin{cor}
\label{UniversalCoverCor}
 Let $M$ be a closed aspherical manifold of dimension at least five
 whose fundamental group has finite decomposition complexity
 \textup{(}as a metric space with a word metric\textup{)}.  For every
 closed manifold $N$ and homotopy equivalence $M\to N$ the
 corresponding bounded homotopy equivalence of universal covers
 is boundedly homotopic to a homeomorphism.
\end{cor}
\begin{proof}
 The universal cover of a closed manifold has bounded geometry as a
 coarse metric manifold.  Further, the universal cover of a closed
 aspherical manifold is uniformly contractible as a coarse metric
 manifold.  Thus, the previous theorem applies.
\end{proof}

Let $M$ be a coarse metric manifold. A {\it net\/} in $M$ is a metric
subspace $\Gamma\subset M$ which is both {\it uniformly discrete\/} --
the distance between distinct points of $\Gamma$ is bounded uniformly away from
zero -- and {\it coarsely dense\/} in $M$ -- for some $C>0$, every
ball $B(x,C)$ in $M$ intersects $\Gamma$. Clearly, the inclusion of a
net into $M$ is a coarse equivalence, so that any two nets are
coarsely equivalent. If $M$ has bounded geometry (as a coarse metric
manifold) then any net in $M$ has bounded geometry (as a discrete
metric space).

\begin{prop}
\label{Rips/manifoldThm}
Let $M$ be a uniformly contractible coarse metric manifold having
bounded geometry and dimension at least five.  Let $\Gamma$ be a net
in $M$.  The assembly map \textup{(}\ref{bbc}\textup{)} of
the bounded Borel isomorphism conjecture for $\Gamma$ identifies
with the assembly map for $M$:
\begin{equation}
\label{bbcM}
  A: H_n(M,\L(e)) \to L_n^{bdd, s}(M).
\end{equation}
Precisely, there are isomorphisms 
\begin{equation*}
  H_n(M,\L(e))\cong \lim_{d\to\infty} H_n(P_d(\Gamma),\L(e)) 
      \quad\text{and}\quad
  L_n^{bdd, s}(M)\cong \lim_{d\rightarrow\infty}L^{bdd,s}(P_d(\Gamma))
\end{equation*}
commuting with the assembly maps.
\end{prop}

\begin{remark}
  The bounded geometry condition is essential here; Dranishnikov,
  Ferry and Weinberger have constructed an example of a uniformly
  contractible manifold $M$ for which the first asserted isomorphism
  fails \cite{DFW}.
\end{remark}


Results analogous to the proposition are typically proved under the
stronger assumption that $M$ is a Riemannian manifold with bounded
geometry, or at least that $M$ is a simplicial complex equipped with a
proper path metric.  See, for example, \cite[Section 3]{HR}.  Our
proof will follow the standard arguments, based on the following
lemma.  For the statement define a {\it coarse metric CW-space\/} to
be a CW-complex equipped with a continuous (pseudo-)metric in which
balls are relatively compact, and in which the cells have uniformly
bounded diameter. The latter property can always be achieved by
refining the CW-structure.

\begin{lem}
\label{boundedhomotopyLem}
Let $X$ be a uniformly contractible coarse metric finite dimensional
CW-space. Suppose that  $X$ admits a bounded geometry net $\Gamma$.  For
every sufficiently large $d>0$ there exist continuous coarse equivalences
\begin{equation*}
  f_d:X\to P_d(\Gamma) \quad\text{and}\quad g_d: P_d(\Gamma)\to X 
\end{equation*}
with the following properties:
\begin{ilist}
  \item $g_d\circ f_d$ is boundedly homotopic to the identity map of
    $X$;
  \item $i_{dd'}\circ f_d\circ g_d$ is boundedly homotopic to the
    inclusion $i_{dd'}: P_d(\Gamma)\to P_{d'}(\Gamma)$, for $d'>d$ 
    sufficiently large. 
\end{ilist}
\end{lem}

\begin{proof}
The proof is inspired by \cite[Section 4] {BR} and \cite[Section 3]{HR}.  
We define $f_d$.  Since $\Gamma$ is a net, there exists a $C>0$
such that the balls of radius $C$ centered at the points of $\Gamma$
cover $X$.  Take a continuous partition of unity
$(\phi_{\gamma})_{\gamma\in \Gamma}$ in $X$ subordinate to this
cover.   Define
\begin{equation*}
  f_d:X\to P_d(\Gamma), \quad  
       f_d(x)=\sum_{\gamma\in \Gamma} \phi_{\gamma}(x)\gamma,  
\end{equation*}
for all $x\in X$.  Whenever $\phi_\gamma(x)$ and $\phi_{\gamma'}(x)$
are simultaneously nonzero the distance between $\gamma$ and $\gamma'$
is at most $2C$.  Thus, $f_d$ is properly defined provided $d>2C$.
Observe that $f_d$ is continuous and uniformly expansive.

We define $g_d$ recursively. On the $0$-skeleton 
$\Gamma\subset P_d(\Gamma)$ we define $g_d$ in the obvious way.
Assuming we have defined $g_d$ on the $k$-skeleton, we extend it to
the $(k+1)$-skeleton using the uniform contractibility of $X$.
Observe that $g_d$ is continuous and uniformly expansive.  Indeed,
this follows from the finite dimensionality of $P_d(\Gamma)$.

Further, $f_d$ and $g_d$ are inverse coarse equivalences -- the
compositions $f_d\circ g_d$ and $g_d\circ f_d$ are close to the
identity maps on $X$ and $P_d(\Gamma)$, respectively.  This is most
easily verified by noting that the compositions are uniformly
expansive, and close to the identity on the copies of $\Gamma$ in
$P_d(\Gamma)$ and $X$, respectively.

Then it is easy to see that for $d'$ large enough, there exits a
linear homotopy from $f_d\circ g_d$ to the identity inside
$P_{d'}(\Gamma)$.

The case of $g_d\circ f_d$ is done by \cite[Lemma 4.4]{BR}. We
reproduce the proof here. Consider the map $\Phi: \{0,1\}\times X\to
X$ defined by $\phi(0,x)=g_d\circ f_d(x)$ and $\Phi(1,x)=x$. Note that
$W=[0,1]\times X$, with the product metric, is a coarse metric
CW-space, and $\{0,1\}\times X$ is a sub-complex of $W$.  Let $W_n$ be
the union of $\{0,1\}\times X$ with the $n$-skeleton of $W$. Using the
uniform contractibility of $W$, one can extend inductively $\Phi$ to a
bounded continuous map $\Phi_n$ defined on $W_n$. Now, since $W$ has
finite dimension, $W=W_n$ for some $n$ and we have constructed a
bounded homotopy equivalence from $f_d\circ g_d$ to the identity.
\end{proof}

\begin{proof}[Proof of Proposition~\ref{Rips/manifoldThm}]
  A topological manifold of dimension at least five admits the
  structure of a 
  CW-complex \cite{KS}.\footnote{This is the only point at which
    we require the dimension to be $\geq 5$ -- the question of whether a
    manifold admits the structure of a CW-complex remains open in low
    dimensions. One could give an alternative proof of
    Proposition~\ref{Rips/manifoldThm} using a Mayer-Vietoris
    argument, which would allow us to remove the dimension
    restriction.}  Thus a coarse metric manifold of dimension at least
  five is a coarse metric CW-space and the lemma applies.
\end{proof}

\begin{proof}[Proof of Theorem~\ref{thm:brig}]
 Let $M$ be as in the statement.  Let $N$ be another coarse metric
 manifold and suppose that $N$ is boundedly homotopy equivalent to
 $M$. According to the bounded surgery exact sequence \cite{FP}, the
 bounded Borel isomorphism conjecture for $M$ implies that $N$ is
 homeomorphic to $M$, assuming that $\dim M\geq 5$.
\end{proof}

\subsection{Application to stable rigidity}

The {\it bounded Farrell-Jones $L$-theory isomorphism conjecture\/}
asserts that a certain assembly map is an isomorphism. Precisely this
conjecture asserts that for a locally finite metric space $\Gamma$ the
assembly map 
\begin{equation*}
  A : \lim_{d\to\infty} H_n(P_d(\Gamma),\L(e))\to 
         \lim_{d\to\infty}L_n^{bdd, <-\infty>}(P_d(\Gamma))
\end{equation*}
is an isomorphism.  Here, for a metric space $X$ and natural number
$n$, we define $L^{bdd,<-\infty>}_n (X)$ to be the direct limit of the bounded
locally finite and free $L$-groups $L_n ^{bdd}(X\times \R^k)$ with the
maps given by crossing with $\R$.  Recall that $\L(e)$, the (simply)
connected surgery spectrum, satisfies 
$\pi _n(\L(e))=L_n^{<-\infty>}(\Z \{e\})$.

\begin{thm}
\label{BoundedFarrelJonesThm}
The bounded Farrell-Jones $L$-theory isomorphism conjecture is true
for metric spaces spaces with bounded geometry and finite
decomposition complexity.
\end{thm}
\begin{proof}
  Immediate from Theorem \ref{AssemblyMapThm} and from the observation
  that if $X$ has finite decomposition complexity, then so does
  $X\times \R^n$ for all $n$.
\end{proof}

The bounded Farrell-Jones $L$-theory isomorphism conjecture has
implications to question of stable rigidity.  Let $M$ be a closed,
aspherical manifold.  By the arguments presented in the previous
section, the bounded Farrell-Jones $L$-theory isomorphism conjectures
for the universal cover of $M$ and for the fundamental group of $M$
are equivalent.  According to the {\it descent principle\/} they imply
the {\it integral Novikov conjecture\/} -- a
detailed argument is contained in the proof of \cite[Theorem 5.5]{CP}.
For a nice exposition of the descent principle in the context of
$C^*$-algebra $K$-theory see \cite{R}.

Recall now from the introduction that a closed manifold $M$ is {\it
  stably rigid\/} if there exists a natural number $n$ with the
following property: for every closed manifold $N$ and every homotopoy
equivalence $M\to N$ the map $M\times \R^n\to N\times\R^n$ is
homotopic to a homeomorphism.  The {\it stable Borel conjecture\/}
asserts that closed aspherical manifolds are stably rigid.  The fact
that the integral Novikov conjecture implies the stable Borel
conjecture was stated without proof in \cite{FP}; for a detailed
treatment see \cite[Proposition 2.8]{J}.  From this discussion, and
our previous results, we conclude:

\begin{thm} 
 The stable Borel conjecture holds for closed aspherical manifolds
 whose fundamental groups have finite decomposition complexity. \qed
\end{thm}

\noindent
No restriction on the dimension is required -- low dimension can
compensated by increasing $n$.  Moreover, if $\dim(M)\geq 5$, then one
can take $n=3$.

\begin{remark}
  In analogy with the bounded Farrell-Jones $L$-theory isomorphism
  conjecture, we could state a bounded
  version of the Farrell-Jones isomorphism conjecture for 
  algebraic $K$-theory.  The Farrell-Jones conjecture for bounded
  locally finite algebraic $K$-theory implies vanishing of the bounded
  Whitehead group
and the bounded reduced algebraic $K$-group described previously.
\end{remark}





\section{Vanishing theorem}
\label{vanishing}

We devote this section to the proof of
Theorem~\ref{AsymptoticVanishingK-theoryThm}, our vanishing result for
the bounded Whitehead and bounded reduced lower algebraic $K$-theory
groups.  In view of the definitions, we obtain
Theorem~\ref{AsymptoticVanishingK-theoryThm} as an immediate
consequence of the following result:

\begin{thms}\label{VanishingK-theoryThm} 
  Let $\Gamma$ be a locally finite metric space with bounded geometry and  finite
  decomposition complexity.  The controlled locally
  finite Whitehead group 
  and the controlled reduced locally finite algebraic $K$-theory group
  vanish asymptotically.  Precisely,
  given $i\geq 0$, $\delta>1$ and $a>1$ there exists $b>1$ such that,
  for any $Z\subset \Gamma$ the natural homomorphisms:
  \begin{equation}
    \label{phi}
      Wh^{\delta}_{1-i}(P_a (Z)) \to 
         Wh^{\delta}_{1-i}(P_{b} (Z))    
  \end{equation}
  \begin{equation}
    \label{psi}
      \widetilde{K}_{-i}^{\delta}(P_{a}(Z))\to
          \widetilde{K}_{-i}^{\delta}(P_{b}(Z))    
  \end{equation}
  are zero.  Here $P_a(\Gamma)$ is equipped with the simplicial metric
  and $P_a(Z)\subset P_a(\Gamma)$ with the subspace
  metric \textup{(}and similarly for $P_b(Z)$\textup{)}. The constant
  $b$ depends only on $i$, $\delta$, $a$ and $\Gamma$, and not on $Z$.
\end{thms}

\begin{remarks}
  To emphasize the dependence among the various constants and metric
  families we shall encounter we shall write, for example, $f=f(g,h)$
  when $f$ depends on $g$ and $h$; if additionally $g=g(p,q)$ and
  $h=h(q,r)$ we write $f=f(g,h)=f(p,q,r)$.
\end{remarks}

In preparation for the proof of
Theorem~\ref{VanishingK-theoryThm} we formalize the notion of a
vanishing family: a collection $\FF$ of metric subspaces of $\Gamma$
is a {\it vanishing family\/} if for every $i\geq 0$, $\delta>1$,
$a>1$, $t>1$ and $p\geq 0$ there exists $b>1$ such that for every
$X\in\FF$ and every $Z\subset N_t(X)$ the homomorphisms
\begin{equation}
  \label{phi2}
        Wh^{\delta}_{1-i}(P_a (Z)\times T^p) \to 
        Wh^{\delta}_{1-i}(P_{b} (Z)\times T^p)
\end{equation}
\begin{equation}
  \label{psi2}
        \widetilde{K}_{-i}^{\delta}(P_{a}(Z)\times T^p)\to
        \widetilde{K}_{-i}^{\delta}(P_{b}(Z)\times T^p)
\end{equation}
are zero, where $N_t(X)$ is the $t$-neighborhood of $X$ in $\Gamma$,
i.e. $N_t(X) =\{y\in \Gamma: d(y, X)\leq t\}$ .  Here, $T^p$ is the
$p$-dimensional torus with the standard Riemannian metric of diameter
one.  Note that $b=b(i,p,t,a,\delta,\FF)$.  We denote the collection
of vanishing families by $\VV$.

Observe that in the definition of vanishing family we have not
specified the metric to be used on $P_a(Z)$ and $P_b(Z)$.  Indeed,
this was intentional as we shall need to employ {\it two\/} different
metrics in the proof of Theorem~\ref{VanishingK-theoryThm}.  The first
is the {\it simplicial metric\/} on $P_a(Z)$ and the second is the
{\it subspace metric\/} inherited from $P_a(\Gamma)$.  
%
%
Similarly we consider the simplicial and subspace metrics on
$P_b(\Gamma)$.

\begin{props}
\label{metricindependence}
The notion of vanishing family is independent of the choice of
metric on $P_a(Z)$ and $P_b(Z)$.  
\end{props}
\begin{proof}
The subspace metric is always smaller than the simplicial metric.
Consequently there is a hierarchy among the four ({\it a priori}
different) definitions of vanishing family.  The weakest version
of vanishing states:
\begin{center}
\begin{minipage}{0.80\linewidth}
For every $a$ (\ref{phi2}) and (\ref{psi2}) are zero for sufficiently
large $b$, when $P_a(Z)$ is equiped with the simplicial metric and
$P_b(Z)$ with the subspace metric;
\end{minipage}
\end{center}
whereas the strongest version states: 
\begin{center}
\begin{minipage}{0.80\linewidth}
For every $a$ (\ref{phi2}) and (\ref{psi2}) are zero for sufficiently
large $b$, when $P_a(Z)$ is equiped with the subspace metric and
$P_b(Z)$ with the simplicial metric.
\end{minipage}
\end{center}
It suffices to show that the weak version of vanishing implies the
strong version.  We shall focus on the Whitehead groups (the case of
the $K$-groups being similar).  Suppose that $Z$ is a vanishing family
in the weak sense.  We shall show that, for sufficiently large
$a'$ depending on $a$ and $\delta$, there exist maps
\begin{equation}
\label{comparingmetrics}
  Wh^{\delta}_{1-i}(P_a^{\ind}(Z)\times T^p)) \to 
        Wh^{\delta}_{1-i}(P_{a'}^{\simp}(Z)\times T^p));
\end{equation}
here, and below, the superscript makes clear which metric is to be
employed, either the subspace or the simplicial.  Assuming this for the
moment, the proof of the proposition is completed by considering the
diagram
\begin{equation*}
\xymatrix{%
   Wh^{\delta}_{1-i}(P_a^{\ind}(Z)\times T^p) \ar[r] \ar[d] & 
         Wh^{\delta}_{1-i}(P_b^{\simp}(Z)\times T^p) \\
   Wh^{\delta}_{1-i}(P_{a'}^{\simp}(Z)\times T^p) \ar[r] &
         Wh^{\delta}_{1-i}(P_{b'}^{\ind}(Z)\times T^p); \ar[u] }     
\end{equation*}
given $a$ we choose $a'$ to ensure existence of the left hand vertical
map as in (\ref{comparingmetrics}); according to the weak version of
vanishing we choose $b'$ so that the bottom horizontal map is zero;
finally, we choose $b$ to ensure existence of the right hand vertical
map as in (\ref{comparingmetrics}).  

It remains to verify the existence of the maps
(\ref{comparingmetrics}).  This follows from the following two
observations.  First, for $a'$ sufficiently large, the inclusion
\begin{equation*}
  P_a^{\ind}(Z) \to P_{a'}^{\simp}(Z)
\end{equation*}
is {\it $1$-Lipschitz at scale $100\delta$\/} -- meaning that
whenever $x$, $y\in P_a^{\ind}(Z)$ satisfy $d(x,y)\leq 100\delta$ then 
the distance between $x$ and $y$ in $P_{a'}^{\simp}(Z)$ is not greater
than their distance in $P_a^{\ind}(Z)$.
Indeed, choose  $a'\geq a$ to be large enough  such that any pair of points of
$P^{\ind}_a(Z)$ at distance less than $100\delta$ lie in a common
simplex in $P_{a'}(Z)$ -- this is possible because the map
$P_a(\Gamma)\to \Gamma$ associating to a point some vertex of the
smallest simplex containing it is uniformly expansive.
Now, the first map in the composition
\begin{equation*}
  P_a^{\ind}(Z) \to P_{a'}^{\ind}(Z) \to P_{a'}^{\simp}(Z)
\end{equation*}
is contractive.  The second map is isometric {\it for pairs of points
  in a simplex\/} -- the subspace and simplicial metrics on $P_a(Z)$
coincide for pairs of points belonging to a common simplex,
essentially because each simplex is a convex subspace of
$P_a(\Gamma)$. 

Second, the ${\delta}$-controlled Whitehead groups are
independent of the behavior of the metric at scales much larger than
$\delta$.  More precisely, an injection $X\to Y$ which is 
{\it $1$-Lipschitz at scale $100\delta$\/} induces a map
$Wh^{\delta}_{1-i}(X)\to Wh^{\delta}_{1-i}(Y)$.  This follows from the
definitions of these groups \cite{RY1}. 
\end{proof}

Finally, before turning to the proof of
Theorem~\ref{VanishingK-theoryThm} we pause to outline the strategy.
We wish to show that a subspace of $\Gamma$ (rather, a family of
subspaces) is a vanishing space from the knowledge that it may be
decomposed as a union of arbitrarily well-separated vanishing
families.  Denote the constituent families $\Cc$ and $\D$ -- here and
below we freely employ the notations of Appendix~\ref{Rips} for Rips
complexes. The following diagram  motivates our proof --
unfortunately, it does not exist in the controlled
setting and must be loosely interpreted:
\begin{equation*}
\xymatrix{ Wh(P_a(\Cc)) \oplus Wh(P_a(\D)) \ar[r]\ar[d] & 
              Wh(P_a(\Cc)\cup P_a(\D)) \ar[r]\ar[d] & 
                 \widetilde K_0(P_a(\Cc)\cap P_a(\D)) \ar^{i}[d] \\
           Wh(P_b(\Cc)) \oplus Wh(P_b(\D)) \ar[r]\ar^{j}[d] & 
              Wh(P_b(\Cc)\cup P_b(\D)) \ar[r]\ar[d] & 
                 \widetilde K_0(P_b(\Cc)\cap P_b(\D)) \ar[d] \\
           Wh(P_c(\Cc)) \oplus Wh(P_c(\D)) \ar[r] & 
              Wh(P_c(\Cc)\cup P_c(\D)) \ar[r] & 
                 \widetilde K_0(P_c(\Cc)\cap P_c(\D)). }
\end{equation*}
In this diagram $a\leq b \leq c$, the horizontal rows are pieces of
Mayer-Vietoris sequences, and the induction hypothesis applies to the
first and third columns.  Thus, given $a$, choose $b$ large enough so
that $i=0$; then choose $c$ large enough so that $j=0$; a simple
diagram chase reveals that the composite of the two maps in the middle
column is zero.

This heuristic does not reveal the need for the introduction of the
relative Rips complex, which we have remarked is an important aspect
of the proof.  Roughly stated, in the diagram above upon passing from
$a$ to $b$ we have no way to ensure that the separation of the
individual spaces comprising the families $\Cc$ and $\D$ does not
evaporate -- the purpose of the relative Rips complex is to
selectively rescale parts of the ambient space while maintaining the
separation between them.  In the proof below we shall point out where
this is needed.

\begin{proof}[Proof of Theorem~\ref{VanishingK-theoryThm}]
  Assuming that $\Gamma$ has
  finite decomposition complexity we shall prove that the collection
  of vanishing families contains the bounded families and, using a
  controlled Mayer-Vietoris argument based on part (5) of Theorem
  \ref{MayerVietorisK} (proved in \cite{RY1}), is closed under
  decomposability.  We thereby conclude that the family 
  $\{\, \Gamma \,\}$ is a vanishing family and the theorem follows.

A uniformly bounded family of subspaces of $\Gamma$ is a vanishing
  family, as we conclude from the following facts:
\begin{ilist}
\item If a subspace $Y\subset \Gamma$ has diameter at most $b$ for
  some $b\geq 0$, then $P_{b}(Z)$ is Lipschitz homotopy equivalent to
  a point (with Lipschitz constant one); indeed the same is true for
  any larger $b$.
\item If two metric spaces $P$ and $Q$ are Lipschitz homotopy
  equivalent (with Lipschitz constant one) then $Wh_{1-i}^{\delta}(P)$
  is isomorphic to $Wh_{1-i}^{\delta}(Q)$, and similarly
  $\widetilde{K}_{-i}^{\delta}(P)$ is isomorphic to
  $\widetilde{K}_{-i}^{\delta}(Q)$.
\item By the choice of the Riemannian metric on $T^p$ and the
  assumption $\delta>1$, $Wh^{\delta}(T^p)$ and
  $\tilde{K}_{-i}^{\delta}(T^p)$ vanish for each $p\geq 0$.
\end{ilist}

Now, let $\FF$ be a family of subspaces of $\Gamma$ and assume that
$\FF$ is decomposable over the collection of vanishing families.
We must show that $\FF$ is a vanishing family; precisely, there exists
$b=b(i,p,t,a,\delta,\FF)$ such that for every $X\in\FF$ and every
$Z\subset N_t(X)$ the maps (\ref{phi2}) and (\ref{psi2}) are zero.

Set $r=r(t,a,\delta,\lambda)$ sufficiently large, to be specified
later.  Obtain an $r$-decomposition of $\FF$ over a vanishing
family $\GG=\GG(r,\FF)$.  Let $X\in \FF$.
We obtain a decomposition:
\begin{equation*}
  X = A \cup B, \qquad A = \bigsqcup_{r} A_i, 
                 \quad B = \bigsqcup_{r} B_j,
\end{equation*}
for which all $A_i$ and $B_j\in\GG$.  Let $Z\subset N_t(X)$; setting 
$C_i=Z\cap N_{t+a}(A_i)$ and $D_j=Z\cap N_{t+a}(B_j)$ we obtain an
analogous decomposition:
\begin{equation*}
  Z = C \cup D, \qquad C = \bigsqcup_{r-2(t+a)} C_i, 
  \quad D = \bigsqcup_{r-2(t+a)} D_j.
\end{equation*}
Denote $\Cc = \{\, C_i \,\}$ and $\D = \{\, D_j \,\}$. By the separation
hypothesis we have $r-2(t+a)>a$ so that $P_a(\Cc)=P_a(C)$ and
$P_a(\D)=P_a(D)$.  Further, $P_a(Z)=P_a(C)\cup P_a(D) =
P_a(\Cc\cup\D)$.  We intend to compare the Mayer-Vietoris sequence of
this pair of subspaces of $P_a(\Gamma)$ to a Mayer-Vietoris sequence for
certain subspaces of an appropriate relative Rips complex. We enlarge the
intersection $\Cc\cap\D = \{\, C_i\cap D_j \,\}$ by setting
\begin{align*}
  W &= N_{a\beta\lambda\delta}(C)\cap N_{a\beta\lambda\delta}(D) \cap Z \\
    &= (N_{a\beta\lambda\delta}(C)\cap D) \cup (C\cap N_{a\beta\lambda\delta}(D)) \\
    &= \bigsqcup_{r-2(t+a\beta\lambda\delta)} W_{ij},
\end{align*}
where all the neighborhoods are in $\Gamma$ and 
\begin{equation*}
  W_{ij} = N_{a\beta\lambda\delta}(C_i) \cap N_{a\beta\lambda\delta}(D_j) \cap Z,
\end{equation*}
and where $\beta$ is the constant appearing in Lemma \ref{commutelemma}.
Observe that $C_i\cap D_j\subset W_{ij}$, so that denoting
$\W=\{\, W_{ij} \,\}$ we have $\Cc\cap\D\subset \W$.  Provided $a\leq
b$ we have a commuting diagram
\begin{equation}
\label{diagone}
  \xymatrix{Wh^\delta(P_a(\Cc\cup\D)) \ar[r] \ar[d] &
      \widetilde K_0^{\lambda\delta}(N_{\lambda\delta}(P_a(\Cc\cap\D))) \ar[d]\\
            Wh^\delta(P_{ab}(\Cc\cup\D,\W)) \ar[r] &
      \widetilde K_0^{\lambda\delta}(N_{\lambda\delta}(P_b(\W))). }
\end{equation}
The horizontal maps are boundary maps in controlled Mayer-Vietoris
sequences in Appendix~\ref{MVsequences}: in the top row the
neighborhood is taken in 
$P_a(\Cc\cup\D)$, and all spaces are given the subspace metric from
$P_a(\Gamma)$; in the bottom row the neighborhood is taken in
$P_{ab}(\Cc\cup\D,\W)$, and all spaces are given the subspace metric
from $P_{ab}(\Gamma,W)$.  The vertical maps are induced from the
proper contraction $P_a(\Gamma)\to P_{ab}(\Gamma,W)$.  In fact, the
right hand vertical map factors as the composite
\begin{equation}
\label{factorRHV}
  N_{\lambda\delta}(P_a(\Cc\cap\D)) \subset P_a(\W) \to 
          P_b(\W)\subset N_{\lambda\delta}(P_b(\W));
\end{equation}
in which the first two spaces are subspaces of 
$P_a(\Cc\cup\D)\subset P_a(\Gamma)$ and the last two are subspaces of
$P_{ab}(\Cc\cup\D,\W)\subset P_{ab}(\Gamma,W)$.  The first inclusion
in (\ref{factorRHV}) follows from
\begin{align*}
  N_{\lambda\delta}(P_a(\Cc\cap\D)) &= 
         \bigcup_{i,j} N_{\lambda\delta} (P_a(C_i\cap D_j)) \\
     &\subset \bigcup_{i,j} P_a(N_{a\beta\lambda\delta}(C_i) \cap 
                    N_{a\beta\lambda\delta}(D_j)) \\
     &\subset \bigcup_{i,j} P_a(W_{ij}) = P_a(\W),
\end{align*}
where we have applied Lemma~\ref{commutelemma} of the appendix for the first
inclusion -- keep in mind that the
neighborhoods on the first line are taken in $P_a(\Cc\cup\D)$.

Applying the induction hypothesis we claim that for sufficiently large
$b$ the right hand vertical map in (\ref{diagone}) is zero.  Indeed, the
components $W_{ij}\in \W$ are contained in the neighborhoods
$N_{t+{a\beta\lambda\delta}}(A_i)$ (and also of
$N_{t+{a\beta\lambda\delta}}(B_j)$) and we can apply the hypothesis with
appropriate choices of the parameters: $t'=t+a\beta\lambda\delta$,
$\delta'=\lambda\delta$, $a'=a$, etc.  In detail,
\begin{equation*}
  \xymatrix{%
       \widetilde K_0^{\lambda\delta}(P_a(\W)) \ar[r]^{\cong} &  
       \prod \widetilde K_0^{\lambda\delta}(P_a(W_{ij})) \ar[r]^{0} &
       \prod \widetilde K_0^{\lambda\delta}(P_b(W_{ij})) \ar[r] &
       \widetilde K_0^{\lambda\delta}(P_b(\W));  }
\end{equation*}
as the spaces $P_a(W_{ij})$ and $P_a(\W)$ are given the subspace metric
from $P_a(\Gamma)$ and the individual $W_{ij}$ are well-separated, the
first map is an isomorphism by Lemma~\ref{separatelemma} (which
guarantees that the various $P_a(W_{ij})$ are separated by at least
$\lambda\delta$); the spaces $P_b(W_{ij})$ are given the simplicial
metric and the middle map is $0$ for sufficiently large $b$ by
hypothesis; the space $P_b(\W)$ is given the subspace metric from
$P_{ab}(\Gamma,\W)$ and the last map is induced by proper contractions
$P_b(W_{ij})\subset P_b(\W)$ onto disjoint subspaces.

Having chosen $b=b(i,p,t',a',\delta',\GG)$ we extend the diagram
(\ref{diagone}) to incorporate the relax-control map for the bottom
sequence:
\begin{equation}\label{diagtwo}
  \xymatrix{ & Wh^\delta(P_a(\Cc\cup\D)) \ar[d] & \\
          &  Wh^\delta(P_{ab}(\Cc\cup\D,\W)) \ar[d]^{relax}\ar[r] & 
      \widetilde K_0^{\lambda\delta}(N_{\lambda\delta}(P_b(\W))) \\
\left.\txt{%
$Wh^{\lambda^2\delta}(P_{ab}(\Cc,\W) \cup N_{\lambda\delta}(P_b(\W)))$\\
     $\oplus$ \\
$Wh^{\lambda^2\delta}(P_{ab}(\D,\W) \cup N_{\lambda\delta}(P_b(\W)))$ }
\right\}
   \ar[r] & Wh^{\lambda^2\delta}(P_{ab}(\Cc\cup\D,\W)) & }
\end{equation}
We conclude from the above discussion and the controlled Mayer-Vietoris
sequence that the image of $Wh^\delta(P_a(\Cc\cup\D))$ under the
composite of the two vertical maps is contained in the image of the
bottom horizontal map.  It remains to apply the induction hypothesis to
$\Cc$ and $\D$.  The case of $\D$ being analogous, we concentrate on
$\Cc$ and shall show that for sufficiently large $c\geq b$ the composite
\begin{equation*}
   P_{ab}(\Cc,\W) \cup N_{\lambda\delta}(P_b(\W))
      \subset P_{ab}(\Cc\cup\D,\W) \to P_b(Z) \to P_c(Z),
\end{equation*}
in which the arrows are induced by proper contractions
$P_{ab}(\Gamma,W)\to P_b(\Gamma)\to P_c(\Gamma)$ is zero on the
$\lambda^2\delta$-controlled Whitehead group.  We have, as subspaces of
$P_{ab}(\Cc\cup\D,\W)\subset P_{ab}(\Gamma,W)$,
\begin{equation}
\label{uniononi}
  P_{ab}(\Cc,\W)\cup N_{\lambda\delta}(P_b(\W)) =
       \bigcup_i \left( P_a(C_i) \cup 
          \bigcup_j N_{\lambda\delta}(P_b(W_{ij})) \right),
\end{equation}
in which the spaces comprising the union over $i$ are well-separated
by Lemma~\ref{separatelemma} (which guarantees
$\lambda^2\delta$-separation).  Further, for fixed $i$ and $j$ we have
\begin{equation*}
  N_{\lambda\delta}(P_b(W_{ij})) \subset 
          P_{ab}(N_{a\beta\lambda\delta}(W_{ij}),W_{ij})) \to 
             P_b(N_{a\beta\lambda\delta}(W_{ij})) \subset
                P_b(N_{2a\beta\lambda\delta}(C_i)),
\end{equation*}
where we have applied Lemma~\ref{commutelemma} for the first
containment (we point out that this is one of the places where  the notion of relative Rips complex is important), and the arrow represents the assertion that the space on
its left maps to the space on its right under the proper contraction
$P_{ab}(\Gamma,W)\to P_b(\Gamma)$.  Accordingly, for
each fixed $i$ we have
\begin{equation*}
     P_a(C_i) \cup \bigcup_j P_b(N_{a\beta\lambda\delta}(W_{ij}))
           \to P_b(N_{2a\beta\lambda\delta}(C_i)),
\end{equation*}
where the arrow is interpreted as above.  Now, we apply our induction
hypothesis a second time, with appropriate choices of the parameters:
$t''=t+2a\beta \lambda\delta$, $\delta''=\lambda^2\delta$, $a''=b$, etc, noting
that $N_{2a\beta\lambda\delta}(C_i)\subset N_{t+2a\beta\lambda\delta}(A_i)$.  We
get $c=c(i,p,t'',a'',\delta'',\GG)$, and analyze
\begin{align*}
  Wh^{\lambda^2\delta}(P_{ab}(\Cc,\W) \cup N_{\lambda\delta}(P_b(\W)) 
     &\cong \prod Wh^{\lambda^2\delta}\left(P_a(C_i) \cup 
                \bigcup_j P_b(N_{a\beta\lambda\delta}(W_{ij}))\right) \\
   &\to \prod Wh^{\lambda^2\delta}(P_b(N_{2a\beta\lambda\delta}(C_i))) \\
   &\to \prod Wh^{\lambda^2\delta}(P_c(N_{2a\beta\lambda\delta}(C_i))) \\
   &\to Wh^{\lambda^2\delta}(P_c(Z)) \\
   & \to Wh^{\delta}(P_{\lambda^2c}(Z)) 
\end{align*}
the $\cong$ follows from the well-separatedness in (\ref{uniononi}); the
spaces $P_c(N_{2a\lambda\delta}(C_i))$ are given the simplicial metrics,
and the second arrow is $0$; the fourth arrow is induced from proper
contractions onto disjoint subspaces of $P_c(Z)$.  The last arrow
follows from the definition of the controlled Whitehead groups. Checking the
dependence of the constant $c$ we find $c=c(i,p,t,a,\lambda,\delta,\FF)$
as required. \end{proof}

\section{Assembly isomorphism}
\label{assembly}

We devote this section to the proof of Theorem~\ref{AssemblyMapThm},
which asserts that assembly is an isomorphism for spaces having finite
decomposition complexity.  In view of the definitions, we obtain
Theorem~\ref{AssemblyMapThm} as an immediate consequence of the
following result:

\begin{thms}
\label{QuantitativeLtheoryThm} 
  Let $\Gamma$ be a locally finite metric space with bounded geometry
  and finite decomposition complexity.  Assembly for $\Gamma$ is an
  asymptotic isomorphism.  Precisely, given $n\geq 0$, $\delta>1$ and
  $a>1$ there exists $b=b(a,\delta,n)\geq a$ such that, for any
  $Z\subset \Gamma$,
  \begin{ilist}
  \item the kernel of $H_n(P_a(Z))\to L_n^{\delta}(P_a(Z))$ is mapped 
              to zero in $H_n(P_b(Z))$;
  \item the image of $L_n^{\delta}(P_a(Z))\to L_n^{\delta}(P_b(Z))$ is
    contained in the image of $H_n(P_b(Z))\to L_n^{\delta}(P_b(Z)$. 
  \end{ilist}
\end{thms}

\noindent
We shall refer to condition (2) in the statement as 
{\it asymptotic surjectivity\/} and
to condition (1) as {\it asymptotic injectivity\/}.

Before turning to the proof we pause to outline the strategy.  The
proof consists essentially of a quantitative version of the five
lemma, which we shall prove using the controlled Mayer-Vietoris
sequence in $L$-theory, precisely parts (4) and (5) of Theorem
\ref{MayerVietorisL_theoryThm}.  Borrowing the notation from the
previous section, consider the following diagram, which again does not
make sense in the controlled setting and must be loosely interpreted:

\begin{equation}
\label{diagramisoheuristic}
\xymatrix{ 
    H_n(P_a(\Cc))\oplus H_n(P_a(\D))  \ar[r]\ar[d] & 
          L_n(P_a(\Cc))\oplus L_n(P_a(\D))\ar[d] \\
    H_n(P_a(\Cc\cup\D) )   \ar[r]\ar[d]  &  
        L_n(P_a(\Cc\cup\D)) \ar[d] \\ 
    H_{n-1}(P_a(\Cc\cap \D)) \ar[r]\ar[d] & 
        L_{n-1}(P_a(\Cc\cap \D))  \ar[d] \\
    H_{n-1}(P_a(\Cc))\oplus H_{n-1}(P_a(\D)) \ar[r] & 
       L_{n-1}(P_a(\Cc))\oplus L_{n-1}(P_a(\D)). }
\end{equation}
In the diagram, the vertical exact sequences are portions of
appropriate Mayer-Vietoris sequences; the horizontal maps are the
assembly maps.  The induction hypothesis applies to the first, third
and fourth rows; we are to prove that the second horizontal map is an
(asymptotic) isomorphism.  In the proof below, we shall concentrate on
(asymptotic) surjectivity -- a simple diagram chase reveals that this
follows (asymptotic) surjectivity of rows one and three and
(asymptotic) injectivity of row four.

In the proof below, to help the reader follow our trajectory we shall
adopt the following conventions: $x$, $y$ and $z$ will be used for
elements in the bounded $L$-theory for unions, intersections and
direct sums, respectively; $x'$, $y'$, $z'$ will be used for elements
in the corresponding homology groups.

As preparation for the proof we introduce the notion of an
$L$-isomorphism family: a collection $\FF$ of metric subspaces of
$\Gamma$ is an {\it $L$-isomorphism family\/} if for every $n\geq 0$,
$\delta>1$, $a>1$, and $t>1$ there exists $b=b(a,\delta,t,n)>1$ such
that for every $X\in\FF$ and every $Z\subset N_t(X)$ the assertions
(1) and (2) of the theorem are satisfied.  As was the case for
vanishing families the notion of an $L$-isomorphism family is not
sensitive to the choice of metric on $P_a(Z)$ and $P_b(Z)$.  Compare
Proposition~\ref{metricindependence} -- the proof in the present
situation is based on the same argument.

Finally, the proof employs both the relative Rips complex,
$P_{ab}(\Cc,\W)$ and the scaled Rips complex, $P_{abm}(\Cc,\W)$ -- see
Definition~\ref{relativeripsdef} and Definition~\ref{scaleripsdef},
respectively, and also Section~\ref{Ripsfamilies}.

\begin{proof}
  The proof will be much more condensed than the proof of
  Theorem~\ref{VanishingK-theoryThm} which we presented in some
  detail; while the present proof is not technically more difficult,
  it is somewhat longer.

We proceed as in the proof of Theorem~\ref{VanishingK-theoryThm}.
Assuming $\Gamma$ has finite decomposition complexity we shall show
that the collection of families that are both vanishing families and
$L$-isomorphism families contains the bounded families, and is closed
under decomposability.  We thereby conclude that the family $\{\,
\Gamma \,\}$ is an isomorphism family, and the theorem follows.

The case of bounded families is handled by the following facts:
\begin{ilist}
\item If a subspace $Y\subset \Gamma$ has diameter at most $b$ for
  some $b\geq 0$, then $P_{b}(Z)$ is Lipschitz homotopy equivalent to
  a point (with Lipschitz constant one); indeed the same is true for
  any larger $b$.
\item If two metric spaces $P$ and $Q$ are Lipschitz homotopy
  equivalent (with Lipschitz constant one) then $L_n^{\delta}(P)$
  is isomorphic to $L_n^{\delta}(Q)$.
\end{ilist}

Now, let $\FF$ be a family of subspaces of $\Gamma$, and assume $\FF$
is decomposable over the collection of families that are both
vanishing and $L$-isomorphism families.  It follows from the proof of
Theorem~\ref{VanishingK-theoryThm} that $\FF$ itself is a vanishing
family and we are to prove that $\FF$ is an $L$-isomorphism family.
We shall concentrate on proving asymptotic surjectivity; asymptotic
injectivity can be proved in essentially the same manner.

Set $r=r(t,a,\delta,\lambda)$ sufficiently large, to be specified
later -- precisely, when a union below is called {\it
  well-separated\/}, this will mean for a sufficiently good choice of
$r$, and the reader will verify that this choice depends only on the
parameters $t$, $a$, $\delta$ and $\lambda$.  Obtain an
$r$-decomposition of $\FF$ over an $L$-isomorphism (and vanishing)
family $\GG=\GG(r,\FF)$.

Let $X\in \FF$.  Let $Z$, $\Cc$, $\D$ and $\W$ be as in the proof of
Theorem~\ref{VanishingK-theoryThm}.  
Let $x\in L_n^{\delta}(P_a(\Cc)\cup P_a(\D))$.  We need to prove that
$x$ is in the image of the assembly map up to increasing $a$.

\begin{itemize}

\item[{\bf Step 1.}]
Using the well-separatedness of $\W$, and the vanishing assumption for
the family $\W$, we can find $b=b(a,\delta, t, n)$ such that the map  
\begin{equation}\label{Kvanishing}
   \widetilde K_0^{\lambda_n\delta}(P_a(\W))\to 
     \widetilde K_0^{\lambda_n\delta}(P_b(\W))
\end{equation} 
is zero. This allows us to consider the boundary map
\begin{equation*}
  \partial: \;L_n^{\delta}(P_a(\Cc\cup\D))\to 
      L_{n-1}^{\lambda_n\delta}(P_b(\W)),
\end{equation*}
where $\partial$ is the boundary map in Theorem B.2 of Appendix B and
$P_b(\W)$ is seen as a subspace of $P_{ab}(Z,\W)$. 
 
\item[{\bf Step 2.}]
Lemma \ref{separatelemma} implies that $P_b(\W)$ is well separated, as
a subspace of $P_{ab}(Z,\W)$.  Hence  
\begin{equation*}
     L_{n-1}^{\lambda_n\delta}(P_b(\W))\cong 
       \prod_{i,j}L_{n-1}^{\lambda_n\delta}(P_b(W_{ij})).
\end{equation*}
Hence, by the surjectivity assumption for $\W$, there exists
$c=c(a,\delta,n,t)\geq b$ and $y'\in H_{n-1}(P_c(\W))$ mapping to (the
image of) $x$  in $L^{\lambda_n\delta}_{n-1}(P_c(\W))$, which we will
simply write $A(y')=\partial (x)$.  

\item[{\bf Step 3.}] 
By Theorem \ref{MayerVietorisL_theoryThm} in Appendix B, part (5), and
(\ref{Kvanishing})  (using that $c\geq b$), we have
$i_*\circ \partial=0$ in  
\begin{equation*}
  L_n^{\delta}(P_a(\Cc\cup\D))\overset{\partial}{\longrightarrow}
       L^{\lambda_n\delta}_{n-1}(P_c(\W))\overset{i_*}{\longrightarrow}
          L^{\lambda_n\delta}_{n-1}(P_{ac}(\Cc,\W))\oplus
             L^{\lambda_n\delta}_{n-1}(P_{ac}(\D,\W)).
\end{equation*}
In particular,  $i_*\circ\partial (x)= i_*\circ A (y')=0$. Considering
the following  commutative diagram 
\begin{equation*}
  \xymatrix{H_{n-1}(P_c(\W)) \ar[r]^-{i_*} \ar[d]^A &
      H_{n-1}(P_{ac}(\Cc,\W))\oplus H_{n-1}(P_{ac}(\D,\W))  \ar[d]^A\\
           L^{\lambda_n\delta}_{n-1}(P_c(\W))   \ar[r]^-{i_*} &
      L^{\lambda_n\delta}_{n-1} (P_{ac}(\Cc,\W))\oplus  
             L^{\lambda_n\delta}_{n-1} (P_{ac}(\D,\W)), }
\end{equation*}
we deduce $A\circ i_*(y')=0$.

\item[{\bf Step 4.}]
By the injectivity assumption for $\W$, there exists $d=d(a,\delta, n,
t)\geq c$ such that the map 
\begin{equation*}
  H_{n-1}(P_{ac}(\Cc,\W))\oplus H_{n-1}(P_{ac}(\D,\W)) \to  
     H_{n-1}(P_{d}(\Cc,\W))\oplus H_{n-1}(P_{d}(\D,\W))
\end{equation*}
sends $i_*(y')$ to $0$.

\item[{\bf Step 5.}]
By exactness of the sequence
\begin{equation*}
  H_{n}(P_{d}(\Cc\cup \D,\W)) \overset{\partial}{\longrightarrow}  
      H_{n-1}(P_{d}(\W)) \overset{i_*}{\longrightarrow} 
        H_{n-1}(P_{d}(\Cc,\W))\oplus H_{n-1}(P_{d}(\D,\W)),
\end{equation*}
there exists $x'\in H_{n}(P_{d}(\Cc\cup \D,\W))$ such that
$y'=\partial(x').$ 

\item[{\bf Step 6.}]
If $m$ is large enough, the metric subfamily  $P_d(\W)$ of
$P_{adm}(\Cc\cup \D,\W)$ is well-separated by  
Lemma \ref{separatelemma}. Hence,
\begin{equation*}
  \widetilde K_{0}^{\lambda_n^2\delta}(N_{\lambda_n^2\delta}(P_d(\W)))\cong 
       \prod_{i,j}\widetilde K_{0}^{\lambda_n^2\delta}(N_{\lambda_n^2\delta}(P_d(W_{ij}'))).
\end{equation*}
On the other hand, by Lemma \ref{RetractionLemma}, when $m$ is large
enough, $N_{\lambda_n^2\delta}(P_d(\W))$ is $2$-Lipschitz  homotopy
equivalent to a subset of  $P_d(\W')$ (just take the homotopy
equivalence $F$ of Lemma \ref{RetractionLemma}, restricted to $V$,
which in our case is $N_{\lambda_n^2\delta}(P_d(\W))$) 
where $\W'=N_{a\beta\lambda_n^2\delta}(\W)$ ($\beta$ is as in Lemma
\ref{RetractionLemma}) and $ P_d(\W')$  is  viewed  as subspace of
$P_{adm}(\Cc\cup \D,\W')$.  
Hence there exists $e=e(a,\delta,n,t)$ such that\footnote{as up to
  increasing $e$, one can change $2\lambda_n\delta$ to
  $\lambda_n\delta$ in the right-hand term.} 

\begin{equation}
\label{vanishingK}
   \widetilde K_{0}^{\lambda_n\delta}(N_{\lambda_n^2\delta}(P_d(\W)))
   \longrightarrow \widetilde K_{0}^{2\lambda_n\delta}(P_d(\W'))
       \overset{0}{\longrightarrow} \widetilde K_{0}^{\lambda_n\delta}(P_e(\W')). 
\end{equation}
We can thus define the boundary map
\begin{equation*}
  L_n^{\lambda_n\delta}(P_{adm}(\Cc\cup \D,\W))
       \overset{\partial}{\longrightarrow}L_{n-1}^{\lambda_n^2\delta}(P_e(\W')).
\end{equation*}

\item[{\bf Step 7.}]
Remember that  $P_{d}(\Cc\cup \D,\W)$ and $P_{adm}(\Cc\cup \D,\W)$ are
the same topological space equipped with two different metrics.  
Considering the following commutative diagram, 
\begin{equation*}
  \xymatrix{H_{n}(P_{d}(\Cc\cup \D,\W) )\ar[r]^-{\partial} \ar[d]^A &
      H_{n-1}(P_{d}(\W))  \ar[d]^A\\
           L^{\lambda_n\delta}_{n}(P_{adm}(\Cc\cup \D,\W))   \ar[r]^-{\partial} &
      L^{\lambda_n^2\delta}_{n-1} (P_{e}(\W')), }
\end{equation*}
we obtain $\partial \circ A (x')=A\circ \partial (x')=A(y')=\partial (x).$
In other words, $\partial (x-A(x')))=0$ in $ L^{\lambda_n^2\delta}_{n-1} (P_{e}(\W')).$
Up to replacing $x$ by $x-A(x')$, we can therefore suppose that $\partial (x)=0$.

\item[{\bf Step 8.}]
Applying part (4) of Theorem \ref{MayerVietorisL_theoryThm} with
\begin{equation*}
        \xymatrix{ & L^{\lambda_n\delta}_{n}(P_{adm}(\Cc\cup \D,\W))\ar[r]^-\partial \ar[d] 
                 &  L^{\lambda_n^2\delta}_{n-1} (\V) \\
           L_n^{\lambda_n^3\delta}(P_{aem}(\Cc,\W')\cup \V)\oplus L_n^{\lambda_n^3\delta}(P_{aem}(\D,\W')\cup \V)\ar[r]^-{j_*} \ar[d] & 
                  L_n^{\lambda_n^3\delta}(P_{aem}(\Cc\cup\D,\W'))\ar[d]  & \\
              L_n^{2\lambda_n^3\delta}(P_{aem}(\Cc,\W'))\oplus L_n^{2\lambda_n^3\delta}(P_{aem}(\D,\W'))\ar[r]^-{j_*} & 
                  L_n^{2\lambda_n^3\delta}(P_{aem}(\Cc\cup\D,\W')), & }
 \end{equation*}
where $\V$ is the $\beta\lambda_n^2\delta$-neighborhood  of $P_d(\W')$
in $P_{adm}(\Cc\cup \D,\W)$. The lower part of the diagram follows
from the Lipschitz-homotopy lemma (see Lemma \ref{RetractionLemma}). 
Together with (\ref{vanishingK}), we deduce  the existence of $z$ such
that $x=j_*(z)$, where $j_*$ is the map defined above. 

\item[{\bf Step 9.}]   
We have $P_{aem}(\Cc,\W'))=\bigcup_iP_{aem}(C_i,\cup_jW_{ij}'))$,
where the union over $i$ is well-separated provided $m$ was chosen
large enough. Moreover, since $\W'\subset
N_{2a\beta\lambda_n^2\delta}(\Cc\cap \D)$, we have the following
contractive inclusion 
\begin{equation*}
  P_{aem}(C_i,\cup_jW_{ij}'))\subset P_e(N_{2a\beta\lambda_n^2}(C_i)).
\end{equation*}
We therefore get a map
\begin{equation*}
  L_n^{2\lambda_n^3\delta}(P_{aem}(\Cc,\W'))\to 
     \prod_i  L_n^{2\lambda_n^3\delta}(P_e(N_{2a\beta\lambda_n^2}(C_i))).
\end{equation*}
The similar statement is true for $\D$.

\item[{\bf Step 10.}]
By the surjectivity assumption applied to  the families $\Cc$ and
$\D$, there exists $f=f(a,\delta,n,t)$ such that the range of 
\begin{equation*}
   L_n^{2\lambda_n^3\delta}(P_e(N_{2a\beta\lambda_n^3}(C_i)))\to
     L_n^{2\lambda_n^3\delta}(P_f(N_{2a\beta\lambda_n^3}(C_i)))
\end{equation*}
is contained in the range of 
\begin{equation*}
   H_n(P_f(N_{2a\beta\lambda_n^3}(C_i))) \to
      L_n^{2\lambda_n^3\delta}(P_f(N_{2a\beta\lambda_n^3}(C_i))),
\end{equation*}
and similarly for $D_i$, for all $i$.  Hence there exists $z'$ in
\begin{equation*}
\prod_i \left(H_n(P_f(N_{2a\beta\lambda_n^2}(C_i)))\oplus
  H_n(P_f(N_{2a\beta\lambda_n^2}(D_i)))\right)\cong
H_n(P_f(N_{2a\beta\lambda_n^2}(\Cc)))\oplus
H_n(P_f(N_{2a\beta\lambda_n^2}(\D)))
\end{equation*}
such that $A(z')=z$ where $z$ is identified with its image through the
map  
\begin{equation*}
 \prod_i \left(L_n^{2\lambda_n^3\delta}(P_f(N_{2a\beta\lambda_n^2}(C_i)))\oplus  
     L_n^{2\lambda_n^3\delta}(P_f(N_{2a\beta\lambda_n^2}(D_i))\right)\to 
     L_n^{2\lambda_n^3\delta}(P_f(N_{2a\beta\lambda_n^2}(\Cc)))\oplus  
         L_n^{2\lambda_n^3\delta}(P_f(N_{2a\beta\lambda_n^3}(\D))).
\end{equation*}

\item[{\bf Step 11.}]
Finally we use the commutative diagram
\begin{equation*}
  \xymatrix{ H_n(P_f(N_{2a\beta\lambda_n^2}(\Cc)))\oplus H_n(P_f(N_{2a\beta\lambda_n^2}(\D))) \ar[r]^-{j_*}\ar[d]^A&
      H_{n}(P_{f}(Z) ) \ar[d]^A\\
          L_n^{2\lambda_n^3\delta}(P_f(N_{2a\beta\lambda_n^2}(\Cc)))\oplus  L_n^{2\lambda_n^3\delta}(P_f(N_{2a\beta\lambda_n^2}(\D)))  \ar[r]^-{j_*} &
     L_n^{2\lambda_n^3\delta}(P_f(Z)), }
\end{equation*}
to get $x=j_*(z)=j_*(A(z'))=A(j_*(z'))$, viewed in
$L_n^{2\lambda_n^3\delta}(P_f(Z))$.  The first two equalities following
from steps 8 and 10. 
We have therefore proved that $x$ is in the range of 
\begin{equation*}
  H_n(P_f(Z))\to L_n^{2\lambda_n^3\delta}(P_f(Z)),
\end{equation*}
which is enough to conclude, as up to increasing $f$, we can replace
$2\lambda_n^3\delta$ by $\delta$ in the right-hand term.  
\end{itemize}
\end{proof}

\section{Concluding remarks}
\label{conclusion}

\subsection{The Novikov conjecture for linear groups}

The methods presented here could be adapted to prove the coarse
Baum-Connes conjecture for spaces with finite decomposition
complexity, and therefore the Novikov conjecture for finitely
decomposable groups.  Of course, since finite decomposition complexity
implies Property $A$ the coarse Baum-Connes conjecture is already
verified for spaces with finite decomposition complexity, as a
consequence of the main result in \cite{Y2}.  The alternate proof that
we suggest here would, however, not rely on infinite dimensional
methods.  Rather, it would be based on the more elementary approach of
\cite{Y1}, where the coarse Baum-Connes conjecture is proved for
groups with finite asymptotic dimension.  We believe
that such an alternative proof would be of some interest.

\subsection{Weak finite decomposition complexity}

A main motivation for this project was to find a weak form of finite
asymptotic dimension satisfied by linear groups.  As described in
Remark~\ref{weakFDCrem}, we first defined weak finite decomposition
complexity.  We abandoned this property when we realized we were
unable to prove the bounded Borel and bounded Farrell-Jones isomorphism
conjectures.  

\begin{que}
  Are the bounded Borel and bounded Farrell-Jones isomorphism
  conjectures true for bounded geometry metric spaces having weak
  finite decomposition complexity?
\end{que}

We speculate that to answer this question one would probably need to
replace the controlled Mayer-Vietoris sequences used here by
controlled spectral sequences in $K$- and $L$-theory.

\appendix

\section{Variations on the Rips complex}
\label{Rips}

In this appendix, we introduce the {\it relative Rips
  complex\/} and the {\it scaled \textup{(}relative\textup{)} Rips
  complex\/} and prove several useful results about their geometry.
These complexes, and the assorted technical results presented here,
play a crucial role in the proofs of
Theorems~\ref{AsymptoticVanishingK-theoryThm} and
\ref{AssemblyMapThm}.  The appendix is designed to be read
independently and, in spite of their technical nature, we believe
that the results presented may be useful in other contexts.

The appendix is organized as follows. In the first subsection, we
shall introduce the {\it relative Rips complex\/} and the {\it scaled
  Rips complex\/}. In the second, we extend the definitions to the
setting of metric families, relevant for the proofs
Theorems~\ref{AsymptoticVanishingK-theoryThm} and
\ref{AssemblyMapThm}. The final subsection contains a collection of
lemmas, also necessary for the proofs of
Theorems~\ref{AsymptoticVanishingK-theoryThm} and
\ref{AssemblyMapThm}.  While we shall state and prove the lemmas in
the context of metric spaces they generalize immediately to the
context of metric families.

Throughout, $\Gamma$ is a locally finite metric space with the
property that $d(x,y)\geq 1$ for each pair of distinct points $x$ and
$y\in\Gamma$. The Rips complex was defined previously (see Definition \ref{RipsDef} and the surrounding discussion).

\subsection{The relative Rips complex and the scaled Rips complex}

In this subsection, we shall introduce the {\it relative Rips
  complex\/} and the {\it scaled Rips complex\/}.  These play
important roles in the proofs of
Theorems~\ref{AsymptoticVanishingK-theoryThm} and
\ref{AssemblyMapThm}, respectively.

\begin{defn}\label{relativeripsdef}
  Let $\Sigma$ be a subset of $\Gamma$.  For $1\leq a\leq b$ we define the
  relative Rips complex $P_{ab}(\Gamma,\Sigma)$ to be the simplicial
  polyhedron with vertex set $\Gamma$ and in which a finite subset 
  $\{\, \gamma_0,\dots,\gamma_n \,\}$ spans a simplex if one of the
  following conditions hold:
  \begin{ilist}
    \item $d(\gamma_i,\gamma_j)\leq a$ for all $i$ and $j$;
    \item $d(\gamma_i,\gamma_j)\leq b$ for all $i$, $j$, and
      $\gamma_i\in\Sigma$ for all $i$.
  \end{ilist}
The relative Rips complex is equipped with the simplicial metric.
\end{defn}

If $C$ is a subspace of $\Gamma$, then
$P_d(C)$ is, in a natural way, a {\it subset\/} of $P_d(\Gamma)$.
When $P_d(C)$ and $ P_d(\Gamma)$ are equipped with the simplicial
metric, the inclusion $P_d(C)\subset P_d(\Gamma)$ is contractive.
Observe that $P_d(C)$ carries, in addition to the simplicial metric, a
subspace metric inherited from $P_d(\Gamma)$.  If $C\subset \Gamma$
and $W\subset\Sigma$ we have inclusions of sets
\begin{equation*}
  P_a(C)\subset P_{ab}(\Gamma,\Sigma),\quad
          P_b(W) \subset P_{ab}(\Gamma,\Sigma).
\end{equation*}
If $P_b(W)$ is equipped with the intrinsic
metric the second inclusion is contractive; the analogous statement is
generally false if $P_b(W)$ is equipped with the subspace metric
inherited from $P_b(\Gamma)$.  Similar remarks apply for $P_a(C)$.

\begin{defn}
\label{scaleripsdef}
Let $W$ be a subset of the metric space $\Gamma$.  For $1\leq
a\leq b$ and a sequence of positive integers
$\overline{m}=m_1,\ldots,m_n,\ldots $, we define the metric space
$P_{ab\overline{m}}(\Gamma; W)$ to be the polyhedron $P_{b}(\Gamma)$
with the metric defined as follows:
\begin{ilist}
  \item each simplex $K$ spanned by a finite subset 
     $\{\gamma_0, \gamma_1, \cdots, \gamma_n\}$ of $\Gamma$ is given
     by the (pseudo) Riemannian metric defined inductively on $n$: 
\begin{alist} 
  \item if $K$ is a simplex in $P_{ab}(\Gamma;  W)$,  then the simplex 
     is endowed the standard simplicial Riemannian metric;  
  \item  if $K$  is not a simplex in $P_{ab}(\Gamma;  W)$ and we have
     inductively defined the (pseudo) Riemannian metric $g_{n-1}$ on its
     $(n-1)$-skeleton $K^{(n-1)}$, then we identify $K$ with the cone  
\begin{equation*}
    ([0,1]\times K^{(n-1)})/(0\times K^{(n-1)})
\end{equation*}
and define a (pseudo) Riemannian metric $g_n$ on $K$ by:
\begin{equation*}
  g_n = m_n^2 dt^2 + t^2 g_{n-1}
\end{equation*}
for $t\in [0,1]$.
\end{alist}
  \item the (pseudo) Riemannian metrics on simplices of 
      $P_{ab}(\Gamma; W)$ can be used to define the length of any
      piecewise smooth path in the polyhedron.  For any pair of points
      $x$ and $y$ in $P_{ab\overline{m}}(\Gamma; W)$, $d(x, y)$ is
      defined to be the infimum of the lengths of all piecewise smooth
      paths in $P_{ab\overline{m}}(\Gamma;W)$ connecting $x$ and $y$. 
\end{ilist} 
\end{defn}

\begin{remark}
We shall actually only use the case  $\overline{m}=(m,m,\ldots)$ in
the proofs of Theorems~\ref{AsymptoticVanishingK-theoryThm} and
\ref{AssemblyMapThm}, where we will denote
$P_{ab\overline{m}}(\Gamma;  W)$ by $P_{abm}(\Gamma; W)$. We however
chose to introduce the more general notion since it will streamline the
proofs of several results in this appendix.  
\end{remark}

\subsection{Extension of the definitions  for metric families}
\label{Ripsfamilies}

In this subsection, we  introduce some further notations in order to
deal with families of subsets of $\Gamma$ instead of just one subspace
at a time. In particular, we will introduce the Rips complex and the
relative Rips complex for metric families. We will not treat the case
of the scaled Rips complex since it is a straightforward adaptation of
the case of the relative Rips complex.

For a family $\Cc=\{\, C \,\}$
of subspaces of $\Gamma$ we define
\begin{equation*}
  P_d(\Cc) = \bigcup_{C\in\Cc} P_d(C) \subset P_d(\Gamma),
\end{equation*}
which we shall always equip with the subspace metric.  Typically, we
shall employ this notation when the family $\Cc$ is disjoint.  Note that
if the family $\Cc$ is $d$-disjoint and $\tilde C$ is the union of
the $C\in\Cc$ then
\begin{equation*}
  P_d(\Cc) = P_d(\tilde C).
\end{equation*}
If the union of families is defined naively, and the intersection of
families is {\it defined\/} to be the family of intersections 
$\Cc\cap \D = \{\, C\cap D : C\in\Cc, D\in\D \,\}$ we have
\begin{equation*}
  P_d(\Cc\cup \D) = P_d(\Cc)\cup P_d(\D), \quad 
           P_d(\Cc\cap \D) = P_d(\Cc)\cap P_d(\D).
\end{equation*}

Just as for the standard Rips complex, we can extend the definition of
the relative Rips complex to families.  For families $\Cc=\{\, C \,\}$
and $\W=\{\, W \,\}$ with each $C\subset \Gamma$ and each
$W\subset\Sigma$ we define
\begin{equation*}
  P_{ab}(\Cc,\W) = \bigcup_{C\in\Cc} P_a(C) \cup 
             \bigcup_{W\in\W} P_b(W),
\end{equation*}
as subspaces of $P_{ab}(\Gamma,\Sigma)$. If $\Sigma$ is not explicitly
specified, then 
$\Sigma$ is understood to be the union of all $W$ in $\W$.
In the special case $a=b$ we have $P_{aa}(\Gamma,\Sigma)=P_a(\Gamma)$
and, more generally $P_{aa}(\Cc,\W) = P_a(\Cc\cup\W)$.  As for the
standard Rips complex, we have the elementary equalities
\begin{equation*}
  P_{ab}(\Cc\cup\D,\W) = P_{ab}(\Cc,\W)\cup P_{ab}(\D,\W), \quad
     P_{ab}(\Cc\cap\D,\W) = P_{ab}(\Cc,\W)\cap P_{ab}(\D,\W)
\end{equation*}
as subspaces of $P_{ab}(\Gamma,\Sigma)$.

\subsection{A few technical results}

In this subsection, we prove a several useful results about the
geometry of the (relative) Rips and scaled Rips complex.  These
results are important tools in the proofs of
Theorems~\ref{AsymptoticVanishingK-theoryThm} and
\ref{AssemblyMapThm}.

Henceforth, we assume $\Gamma$ has bounded geometry.

\begin{lem}[Comparison lemma]
  Let $a\geq 1$, and let
  $P_a(\Gamma)$ be equipped as usual with the
  simplicial metric.  For $x$ and $y\in\Gamma$ we have
  \begin{equation*}
    d_{\Gamma}(x,y) \leq a\;\alpha\; d_{P_a(\Gamma)}(x,y),
  \end{equation*}
for some constant $\alpha$ depending only on the dimension of
$P_a(\Gamma)$. 
\end{lem}
\begin{proof}
As this lemma is classical, we will only sketch its proof. 
Let $x,y\in X$ and suppose that there exists a path $\gamma$ of length
$l$ between $x$ and $y$ in $P_a(X)$ (since otherwise, the distance
between $x$ and $y$ in $P_a(X)$ is infinite). 
Let $n$ be the minimal integer such that $\gamma$ is contained in the
$n$-skeleton of $P_a(X)$. We will prove the inequality
$d_{\Gamma}(x,y)\leq a \alpha l$ for some $\alpha$ depending only on
$n$ by induction on $n$. First note that if $n=1$, then the inequality
$d_\Gamma(x,y) \leq a l$ is trivial. Now consider
a simplex $\Delta$ of dimension $n$ whose interior intersects
$\gamma$. Let $u<v$ be such that $\gamma(s)$ lies in $\Delta$ for all
$u\leq s\leq v$ and the path $\gamma$ intersects  $\partial \Delta$ at
$u$ and $v$ (when $s$ is $u$ and $v$), 
where $\gamma$ is parametrized by arclength $s$.  It is easy to see
that there exists a constant $C$ depending only on $n$ such that the
portion of $\gamma$ lying between $u$ and $v$ can be replaced by a
path in $\partial \Delta$ of length $\leq C(v-u).$ Doing this on each
such $n$-simplex, we obtain a path in the $(n-1)$-skeleton of length
$\leq Cl$ and we conclude by induction.  
\end{proof}

\begin{lem}\label{scalecompLemma}[Comparison lemma for the scaled complex]
Let $a\geq 1$, and let $C$ be a subspace of $\Gamma$.  There exists $\beta\geq 1$ depending only on the dimension of
$P_a(\Gamma)$
such that for all $b\geq a$, there exists $M>0$ for which
$$d_{\Gamma}(x,C) \leq a\;\beta\; d(x,P_b(C)),$$
for all $x\in \Gamma$, provided $m_k\geq M$ for all $k$,
where the distance for the right-hand term is taken in $P_{ab\overline m}(\Gamma,C)$.
 \end{lem}
\begin{proof}
It is enough to show that if $\gamma$ is a path of length $l$ in  $P_{ab\overline{m}}(\Gamma,C)$, parametrized by its arc length with respect to the (pseudo) riemannian metric,
between $x\in \Gamma$ and $P_b(C)$, then 
\begin{equation}\label{length inequality}
d_{\Gamma}(x,C)\leq a\beta l.
\end{equation}
We proceed by induction on $n$, the minimal integer such that $\gamma$
is contained in the union of $P_{a}(\Gamma)$ and the $n$-skeleton of
$P_{ab\overline{m}}(\Gamma,C)$.   
Precisely, our induction hypothesis will be the following: for all
$\beta>\alpha$, where $\alpha$ appears in the comparison lemma for
$P_a(\Gamma)$, and every path of length $l$ contained in the union of
$P_{a}(\Gamma)$ and the $n$-skeleton, there exists $M$ such that
(\ref{length inequality}) holds for all $\overline{m}$ such that
$m_k\geq M$ for all $1\leq k \leq n$. 

Let us start with the case $n=1$. Note that up to replacing $\gamma$ by a sub-path, we can always suppose that it does not intersect $C$ at any $t<l$.  We can also suppose that if $\gamma$ meets the interior of an edge not belonging to $P_a(\Gamma)$, then this edge is completely contained in $\gamma$. 
Hence traveling along $\gamma$ means that, either we stay in $P_a(\Gamma)$, or we jump between two points in $\Gamma$, at distance $\leq b$, through  
an edge of length $m_1$. Hence choosing $M=b$, we conclude thanks to the comparison lemma in $P_a(\Gamma).$

Now let us suppose that $n\geq 2$.
Fix some $\beta_1>\beta_2>\alpha$ and choose an $M$ such that the induction hypothesis applies for $\beta=\beta_2$. We assume moreover that $M\leq m_k\leq K$ for all $1\leq k \leq n-1$, where $K$ is some integer.
Let us assume that $\gamma$ meets at least a simplex $\Delta$ of dimension $n$ which does not belong to $P_a(\Gamma)$. 
Let $u<v$ be such that $\gamma(t)\in \Delta$ for $u\leq t\leq v$, and 
$\gamma$ meets the boundary of $\Delta$ at $u$ and $v$. We start with two observations. Let $\Delta= ([0,1]\times \partial \Delta)/(0\times \partial \Delta)$ and let $\eta\in (0,1).$

First, note that if $\gamma$ meets $[0,1-\eta]\times \partial \Delta$,
then $v-u\geq \eta m_n$. But the diameter of $\partial \Delta$ is less
than $\rho K$, for some $\rho$ depending only on $n$.  Hence we can
replace the portion of $\gamma$ between $u$ and $v$ by a path
contained in $\partial \Delta$, of length 
$\leq \rho K\leq \rho K (v-u)/(\eta m_n)$.

Second, if $\gamma$ is contained in  $[1-\eta,1]\times \partial \Delta$, then observe that the retraction of $[1-\eta,1]\times \partial \Delta$  onto $\partial \Delta$ is a $(1-\eta)^{-1}$-Lipschitz map, and hence, projecting $\gamma$ to the boundary increases its length by at most $(1-\eta)^{-1}$. Hence there exists a path $\gamma'$ completely contained in the union of $P_{a}(\Gamma)$ and the $(n-1)$-skeleton whose length $l'$ satisfies 
$$l'\leq (\rho K/(\eta m_n)) l+ (1-\eta)^{-1}l.$$ Applying the induction hypothesis to $\gamma'$ yields
$$d_{\Gamma}(x,C)\leq a\beta_2 l'\leq a\beta_2((\rho K/(\eta m_n)+ (1-\eta)^{-1})l.$$
First fix $\eta$ such that 
$$\beta_2(1-\eta)^{-1}<\beta_1.$$  We then take $M'\geq M$ big enough so that  
$$\beta_2((\rho  K/(\eta m_n)+ (1-\eta)^{-1})\leq \beta_1 $$ for all $m_n\geq M'$.  
This gives the desired inequality 
$$d_{\Gamma}(x,C)\leq a\beta_1 l,$$
under the assumption that $M\leq m_k\leq K$ for $1\leq k\leq n-1$, and $m_n\geq M'$. But since increasing  $m_k$ can only increase $l$, this inequality remains true under the  condition that $m_k\geq M'$ for all $1\leq k\leq n$.  
\end{proof}

Next we make the following observation, from which we will immediately deduce the {\it neighborhood}  and the {\it separation} lemmas below.
 
\begin{lem}\label{NeighPrelimlemma}
Let $C$ be a subspace of $\Gamma$ and let $\varepsilon\geq 1$ and $a\geq 1$.  There exists $\beta\geq 1$  depending only on the dimension of
$P_a(\Gamma)$ such that the following statements are true. Viewing  $P_a(C)$ as a subspace of $P_a(\Gamma)$ we have
  \begin{equation*}
    N_{\varepsilon}(P_a(C))\cap \Gamma \subset N_{a\varepsilon\beta}(C),
  \end{equation*}
Similarly for the relative Rips complex, viewing
$P_b(C)$ as a subspace of  $P_{ab}(\Gamma,C)$ ($b\geq a$) we have
\begin{equation*}
  N_\varepsilon(P_b(C))\cap \Gamma \subset N_{a\varepsilon\beta}(C).
  \end{equation*}
 Finally, for the scaled complex, viewing
$P_b(C)$ as a subspace of  $P_{ab\overline{m}}(\Gamma,C)$ we have
\begin{equation*}
  N_\varepsilon(P_b(C))\cap \Gamma \subset N_{a\varepsilon\beta}(C).
  \end{equation*}
 provided that $\overline{m}$ is large enough in sense that  $m_k\geq
 M$ for all $k$, where $M$  depends only on $b$.  
\end{lem}
 
\begin{proof}
The diameter of a simplex in $P_a(\Gamma)$ is bounded by a universal
constant.  The first assertion therefore follows from the comparison
lemmas upon approximating a given $z\in P_a(C)$ by a vertex in $C$.

Observe that  a path in $P_{ab}(\Gamma,C)$ leaving $P_b(C)$
must pass through $P_a(C)$. Hence, moving away from 
$P_b(C)$ within $P_{ab}(\Gamma,C)$ is the same as moving away from
$P_a(C)$ in $P_a(\Gamma)$. The second assertion therefore follows from
the first one.  

Finally, the third assertion follows immediately from the previous lemma.
\end{proof}

The following lemma is an easy consequence of the previous and is left
to the reader.

\begin{lem}[Neighborhood lemma]\label{commutelemma}
Let $C\subset\Gamma$, $\varepsilon\geq 1$ and $a\geq 1$.  Viewing 
  $P_a(C)\subset P_a(\Gamma)$ we have
  \begin{equation*}
    N_{\varepsilon}(P_a(C)) \subset P_a(N_{a\varepsilon\beta}(C)),
  \end{equation*}
for some constant $\beta$ depending only on the dimension of
$P_a(\Gamma)$.  Similarly for the relative Rips complex, viewing
$P_b(C)\subset P_{ab}(\Gamma,C)$ ($b\geq a$) we have
\begin{equation*}
  N_\varepsilon(P_b(C))\subset P_{ab}(N_{a\varepsilon\beta}(C),C).
  \end{equation*}
\end{lem}

\begin{lem}[Separation lemma]\label{separatelemma} 
  Let $\varepsilon \geq 1$ and $a\geq 1$.  If the family $\Cc$ of subsets of 
  $\Gamma$ is $\varepsilon$-separated, then the family $P_a(\Cc)$
  (resp. $P_b(\Cc)$) is $\varepsilon(a\beta)^{-1}$-separated in
  $P_a(\Gamma)$ (resp. in $P_{ab}(\Gamma, \Cc) $ for $b\geq a$, and in
  $P_{ab\overline m}(\Gamma, \Cc)$ for $b\geq a$ if $\overline m$ is
  large enough), where $\beta$ only depends on the dimension of  $P_a(\Gamma)$.
\end{lem}
\begin{proof}
The first two cases are direct consequences of the neighborhood lemma
above. For the scaled complex, it follows from Lemma \ref{scalecompLemma}. 
\end{proof}

Note that the neighborhood lemma does not apply to the scaled  Rips
complex. Instead, we have the following slightly weaker statement. 

\begin{lem}[Lipschitz homotopy lemma]\label{RetractionLemma}  
Let  $C$ and $W$ be subspaces of the metric space $\Gamma$. Let
$\varepsilon \geq 1$ and $b\geq a\geq 1$. 
Let $V$ be the $\varepsilon$-neighborhood of $P_b(W)$ in
$P_{ab\overline{m}}(\Gamma,W)$, let $W'$ be the
$a\beta\varepsilon$-neighborhood of $W$ in $\Gamma$, where $\beta$ is
the constant appearing in Lemma \ref{NeighPrelimlemma}. Then, for all
$c\geq b$, there exist $M>0$ and  a proper continuous map  
\begin{equation*}
  F: (P_{ac\overline{m}}(C,W')\cup V)\times [0,1]\to 
          P_{ac\overline m}(C, W')\cup V
\end{equation*}  
such that
\begin{ilist}
  \item $F(\cdot, t)$ is $2$-Lipschitz  for all $t\in [0,1],$ provided
      that  $m_k\geq M$ for all $k$,  
  \item for each $t\in [0,1],$ $F(\cdot, t)$ restricts to the identity
      map on $P_{ac\overline{m}}(C, W')$, 
  \item $F(\cdot, 0)$ is the identity map on
      $P_{ac\overline{m}}(C,W')\cup V$, and the image of $F(\cdot,1)$
      lies in $P_{ac\overline{m}}(C,W')$.   
\end{ilist}
Moreover,  the constant $M$ depends only on  $\varepsilon$ and the
dimension of $P_c(\Gamma)$. 
\end{lem}

\begin{proof}
By Lemma \ref{NeighPrelimlemma}, we have $V\cap \Gamma\subset W'$. Now
for each $n$, let $X_n$ be the union of $P_{ac\overline{m}}(C, W')$
with the $k$-simplices of $P_{ab\overline{m}}(\Gamma,W)$ whose
interiors meet $V$, for all $k\leq n$. Let $Y_n=X_n\cap
(P_{ac\overline{m}}(C, W')\cup V)$. 
We will prove the following stronger statement by induction on $n$. 

For all $c\geq b$, and all $\eta>0$, if $M$ is large enough, there
exists  a  proper continuous map $F: Y_n\times [0,1]\to Y_n$ where for
each $t\in [0,1],$ $F(\cdot, t)$ is $(1+\eta)$-Lipschitz if $m_k\geq
M$ for all $1\leq k\leq n$, and restricts to the identity map on
$P_{ac\overline m}(C, W')$. Moreover $F(\cdot, 0)$ is the identity map
on $Y_n$, and the image of $F(\cdot,1)$ lies in $P_{ac\overline
  m}(C,W')$. 
This is obvious for $n=0$. We therefore assume $n\geq 1$.  

We suppose that this statement holds for $Y_{n-1}$.  Note that it is
enough to show that for every $\eta>0$, there exists some proper
continuous map $H:Y_n\times [0,1]\to Y_{n}$ such that for every $t\in
[0,1],$ $H(\cdot, t)$ is $(1+\eta)$-Lipschitz as soon as $m_n$ is
large enough, $H$ restricts to the identity map on
$P_{ac\overline{m}}(C, W')$, $H(\cdot, 0)$ is the identity map on
$Y_{n}$, and the image of $H(\cdot,1)$ lies in $Y_{n-1}$.   

Let $\Delta= ([0,1]\times \partial \Delta)/(0\times \partial \Delta)$
be a simplex in $X_n$, which is not in  $X_{n-1}$. Observe that
$Y_n\cap \Delta \subset [1-\varepsilon/m_n]\times \partial \Delta$.
Consider
the standard homotopy between $[1-\varepsilon/m_n]\times \partial \Delta$ and $\partial \Delta$.
We define $H$ to be the restriction of this homotopy to $(Y_n\cap
\Delta)\times [0,1]$. Clearly the Lipschitz constant is at most
$1/(1-\varepsilon/m_n)$. Hence it suffices to take $m_n$ large enough
to make this number less than $1+\eta$. 
\end{proof}

\section{Mayer-Vietoris sequences in bounded $K$ and $L$-theory}
\label{MVsequences}

In this section, we recall from \cite{RY1,RY2} the controlled
Mayer-Vietoris sequences in $K$ and $L$-theory.  These are important
tools in our proof of the bounded Borel conjecture for spaces with
finite decomposition complexity.

\begin{thms}\label{MayerVietorisK}
Let $X$ be a metric space, written as the union of closed subspaces
$X=A\cup B$.  There exists a universal constant $\lambda>1$
(independent of $X$, $A$ and $B$) such that  for each $\delta> 0$, 
  \begin{ilist}
    \item in $Wh^\delta(A\cap B) \overset{i_*}{\rightarrow} Wh^\delta(A)\oplus Wh^\delta(B)
                 \overset{j_*}{\rightarrow}  Wh^\delta(X)$, we have $j_* i_*=0$;
                 
    \item if $N_{\lambda\delta}(A\cap B)\subset W$, then the
        relax-control image of the kernel of $j_*$ in 
        $Wh^{\lambda^2\delta}(A\cup W)\oplus Wh^{\lambda^2\delta}(B\cup W)$ is contained in
      the image of $i_*$ below
      \begin{equation*}
      \xymatrix{  & Wh^\delta(A) \oplus Wh^\delta(B) \ar[r]^-{j_*} \ar[d]
                      & Wh^\delta(X)\\
             Wh^{\lambda^2\delta}(W) \ar[r]^-{i_*} & Wh^{\lambda^2\delta}(A\cup W) \oplus
                     Wh^{\lambda^2\delta}(B\cup W) & },
      \end{equation*}
      where $N_{\lambda \delta}(A\cap B) = 
             \{\, x\in X \colon d(x,A\cap B)\leq \lambda \delta \,\}$;
             
    \item if $N_{\lambda \delta}(A\cap B)\subset W$, then in 
    $$Wh^\delta(A)\oplus Wh^\delta(B) \overset{j_*}{\rightarrow}
                Wh^\delta(X) \overset{\partial}{\rightarrow} 
                    \tilde{K}_{0}^{\lambda\delta}(W),$$
                 we have $\partial j_*=0$;
                
    \item if $N_{\lambda\delta}(A\cap B)\subset W$, then
     the relax-control image of the kernel of $\partial$ in
      $Wh^{\lambda^2\delta}(X)$ is contained in the image of $j_*$
      below
      \begin{equation*}
        \xymatrix{ & Wh^\delta(X) \ar[r]^-\partial \ar[d] 
                 & \tilde{K}_0^{\lambda\delta}(W) \\
            Wh^{\lambda^2\delta}(A\cup W)\oplus
            Wh^{\lambda^2\delta}(B\cup W) \ar[r]^-{j_*} & 
                  Wh^{\lambda^2\delta}(X) }
      \end{equation*}
      
    \item if $N_{\lambda\delta}(A\cap B)\subset W$, then
    in 
    $$Wh^\delta(X) \overset{\partial}{\rightarrow} 
               \tilde{K}_0^{\lambda\delta}(W) \overset{i_*}{\rightarrow}
      \tilde{K}_0^{\lambda\delta}(A\cup W) \oplus
      \tilde{K}_0^{\lambda\delta}(B\cup W),$$  we have $i_* \partial=0$;
      
    \item if $N_{\lambda\delta}(A\cap B)\subset W$, then
    the relax-control image of the kernel of $i_*$ in
      $\tilde{K}_0^{\lambda^2\delta}(W)$ is contained in the image of
      $\partial$
      \begin{equation*}
        \xymatrix{ & \tilde{K}_0^\delta(A \cap B) \ar[r]^-{i_*} \ar[d] &
                 \tilde{K}_0^\delta(A) \oplus \tilde{K}_0^\delta(B) \\
             Wh^{\lambda\delta}(X) \ar[r]^{\partial} & 
             \tilde{K}_0^{\lambda^2\delta}(W) & } 
      \end{equation*}
  \end{ilist}
\end{thms}

The precise $L$-theory version we require is the following result  where, for each metric space $Y$, each integer $n\geq 0$ and  $\delta>0$,  
$L_n ^{\delta} (Y)$ is the $\delta$-controlled  locally finite and  free
 L-theory of $Y$ \cite{RY2}.  This result is a consequence of Theorem~7.3 and Proposition 4.6 in \cite{RY2}),  Proposition 3.2 and Proposition 3.4 in \cite{RY1}.

\begin{thms}\label{MayerVietorisL_theoryThm} Let $P$ be a locally compact polyhedron and $P'$ a subpolyhedron of $P$. Assume that $P$ and $P'$ are respectively given with metrics $d$ and $d'$ satisfying $d(x,y)\leq d'(x,y)$ for all $x$ and $y$ in $P'$. 
Let $X$ be a  metric subspace of $P'$. Assume that $X$ is  written as the union of closed subspaces
$X=A\cup B$.  For every integer $n\geq 2$ there exists $\lambda_n>1$, which
depends only on $n$, such that  for each $\delta> 0$,
  \begin{ilist}
    \item in $L_n^\delta(A\cap  B) \overset{i_*}{\rightarrow} L_n^\delta(A)\oplus L_n^\delta(B)
                 \overset{j_*}{\rightarrow}  L_n^\delta(X)$, we have $j_* i_*=0$, where 
                 the metrics on $A\cap B$, $A$, $B$ and $X$ are inherited from the metric of $P'$;
                 
    \item if $N_{\lambda_n\delta}(A\cap B)\subseteq  W \subseteq P$ and the natural homomorphism from 
    $\tilde{K}_0^{\lambda_n\delta}( N_{\lambda_n\delta}(A\cap B))$ to $\tilde{K}^{\lambda_n\delta}_0(W)$ is zero,  then the
        relax-control image of the kernel of $j_*$ in 
        $$L_n^{\lambda_n^2\delta}(A\cup W)\oplus L_n^{\lambda_n^2\delta}(B\cup W)$$ is contained in
      the image of $i_*$ below
      \begin{equation*}
      \xymatrix{  & L_n^\delta(A) \oplus L_n^\delta(B) \ar[r]^-{j_*} \ar[d]
                      & L_n^\delta(X)\\
             L_n^{\lambda_n\delta}(W) \ar[r]^-{i_*} & L_n^{\lambda_n^2\delta}(A\cup W) \oplus
                     L_n^{\lambda_n^2\delta}(B\cup W) & },
      \end{equation*}
      where $N_{\lambda_n\delta}(A\cap B) = 
             \{\, x\in X \colon d(x,A\cap B)\leq \lambda_n\delta \,\}$ is given the  metric of $P'$,   the metrics on  $A$, $B$ and $X$ are inherited from the  metric of $P'$, and the metrics on  $W$, $A\cup W$ and  $B\cup W$  are inherited from the  metric of $P$;
             
    \item if $N_{\lambda_n\delta}(A\cap B)\subseteq W \subseteq P$ and the natural homomorphism from 
    $\tilde{K}^{\lambda_n\delta}_0( N_{\lambda_n\delta}(A\cap B))$ to $\tilde{K}_0^{\lambda_n\delta}(W)$ is zero, then in 
    
    $$L_n^\delta(A)\oplus L_n^\delta(B) \overset{j_*}{\rightarrow}
                L_n^\delta(X) \overset{\partial}{\rightarrow} L_{n-1}^{\lambda_n\delta}(W),$$
                 we have
                $\partial j_*=0$,
               where $N_{\lambda_n\delta}(A\cap B) = 
             \{\, x\in X \colon d(x,A\cap B)\leq \lambda_n\delta \,\}$ is given the  metric of $P'$,   the metrics on  $A$, $B$ and $X$ are inherited from the  metric of $P'$, and the metric on  $W$ is inherited from the  metric of $P$;
                
    \item if $N_{\lambda_n\delta}(A\cap B)\subseteq W \subseteq P$ and the natural homomorphism from 
    $\tilde{K}_0^{\lambda_n\delta}( N_{\lambda_n\delta}(A\cap B))$ to $\tilde{K}^{\lambda_n\delta}_0(W)$ is zero, then
     the relax-control image of the kernel of $\partial$ in
      $L_n^{\lambda_n^2\delta}(X)$ is contained in the image of $j_*$
      below
      \begin{equation*}
        \xymatrix{ & L_n^\delta(X) \ar[r]^-\partial \ar[d] 
                 & L_{n-1}^{\lambda_n\delta}(W) \\
            L_n^{\lambda_n^2\delta}(A\cup W)\oplus
            L_n^{\lambda_n^2\delta}(B\cup W)\ar[r]^-{j_*} & 
                  L_n^{\lambda_n^2\delta}(X\cup W) }
      \end{equation*}
   where   $N_{\lambda_n\delta}(A\cap B) = 
             \{\, x\in X \colon d(x,A\cap B)\leq \lambda_n\delta \,\}$ is given the  metric of $P'$,   the metric on $X$ is inherited from the metric of $P'$, and the metrics on  $W$, $A\cup W$, $B\cup W$ and $X\cup W$ are inherited from the  metric of $P$; 
      
    \item if $N_{\lambda_n\delta}(A\cap B)\subseteq W \subseteq P$ and the natural homomorphism from 
    $\tilde{K}^{\lambda_n\delta}_0( N_{\lambda_n\delta}(A\cap B))$ to $\tilde{K}^{\lambda_n\delta}_0(W)$ is zero, then
    in 
    
    $$L_n^\delta(X) \overset{\partial}{\rightarrow} L_{n-1}^{\lambda_n\delta}(W) \overset{i_*}{\rightarrow}
      L_{n-1}^{\lambda_n\delta}(A\cup W) \oplus
      L_{n-1}^{\lambda_n\delta}(B\cup W),$$  we have $i_* \partial=0$,
      where   $N_{\lambda_n\delta}(A\cap B) = 
             \{\, x\in X \colon d(x,A\cap B)\leq \lambda_n\delta \,\}$ is given the  metric of $P'$,   the metric on $X$ is inherited from the metric of $P'$, and the metrics on  $W$, $A\cup W$ and $B\cup W$  are inherited from the  metric of $P$; 
      
    \item if $N_{\lambda_n\delta}(A\cap B)\subseteq W \subseteq P$ and the natural homomorphism from 
    $\tilde{K}^{\lambda_n\delta}_0( N_{\lambda_n\delta}(A\cap B))$ to $\tilde{K}^{\lambda_n\delta}_0(W)$ is zero, then
    the relax-control image of the kernel of $i_*$ in
      $L_{n-1}^{\lambda_n^2\delta}(W)$ is contained in the image of
      $\partial$
      \begin{equation*}
        \xymatrix{ & L_{n-1}^\delta(A \cap B) \ar[r]^-{i_*} \ar[d] &
                 L_{n-1}^\delta(A) \oplus L_{n-1}^\delta(B) \\
             L_n^{\lambda_n\delta}(X)\ar[r]^{\partial} & 
             L_{n-1}^{\lambda_n^2\delta}(W) & } 
      \end{equation*}
      where   $N_{\lambda_n\delta}(A\cap B) = 
             \{\, x\in X \colon d(x,A\cap B)\leq \lambda_n\delta \,\}$ is given the  metric of $P'$,   the metrics on $X$, $A\cap B$, $A$ and $B$  are inherited from the  metric of $P'$, and the metric on  $W$ is inherited from the  metric of $P$. 
  \end{ilist}
\end{thms}

\bigskip
\footnotesize

\end{document}